\documentclass{article}
\usepackage[%
journal=JST,    
lang=british,   
]{ems-journal}

\usepackage{amsfonts}
\usepackage{tikz}
\usepackage{color} 
   \definecolor{cites}{rgb}{0.50 , 0.00 , 0.00}  
   \definecolor{urls} {rgb}{0.00 , 0.00 , 0.50}  
   \definecolor{links}{rgb}{0.00 , 0.00 , 0.50}   
\definecolor{sccol}{rgb}{0,0,0.5}

\usepackage{colortbl}
\usepackage{soul}

\newcommand\A{{\mathcal A}}
\newcommand\B{{\mathcal B}}
\newcommand\cC{{\mathcal C}}
\newcommand\cD{{\mathcal D}}
\newcommand\cE{{\mathcal E}}
\newcommand\cF{{\mathcal F}}
\newcommand\cI{{\mathcal I}}
\newcommand\cS{{\mathcal S}}
\newcommand\C{{\mathbb C}}
\newcommand\D{{\mathbb D}}

\newcommand\N{{\mathbb N}}

\newcommand\T{{\mathbb T}}
\newcommand\R{{\mathbb R}}

\newcommand\Z{{\mathbb Z}}
\newcommand\Hto{\,
   \unitlength0.1ex
   \begin{picture}(30,15)
   \put(13,16){\makebox(0,0)[]{\tiny\rm H}}
   \put(15,5){\makebox(0,0)[]{$\to$}}
   \end{picture}\,
}
\newcommand\per{{\sf per}}

\newcommand\ph{{\varphi}}
\newcommand\eps{{\varepsilon}}
\newcommand\spec{{\rm spec}\,}  
\newcommand\specn{{\rm spec}}   
\newcommand\speps{{\rm spec}_\eps}

\newcommand\Specn{{\rm Spec}}   
\newcommand\Spec{{\rm Spec}\,}  
\newcommand\Speps{{\rm Spec}_\eps}

\newcommand\dist{{\rm dist}}
\newcommand\conv{{\rm conv}}

\newcommand\diag{{\rm diag}}
\newcommand\Diag{{\rm Diag}}
\newcommand\ri{{\rm i}}

\theoremstyle{plain}
\newtheorem{theorem}{Theorem}[section]
\newtheorem{lemma}[theorem]{Lemma}
\newtheorem{corollary}[theorem]{Corollary}
\newtheorem{proposition}[theorem]{Proposition}

\theoremstyle{definition}
\newtheorem{definition}[theorem]{Definition}

\newtheorem{example}[theorem]{Example}
\newtheorem{remark}[theorem]{Remark}

\newenvironment{proofof}[1]
 {\par\noindent{\em Proof of #1.}}
 {\qed\par\pagebreak[2]}

\numberwithin{figure}{section}  

% ...

\numberwithin{equation}{section}

\begin{document}

\title{On Spectral Inclusion Sets and Computing the Spectra and Pseudospectra of Bounded Linear Operators}
\titlemark{On Spectral Inclusion Sets and Computing Spectra and Pseudospectra}



\emsauthor{1}{
	\givenname{Simon}
	\surname{Chandler-Wilde}
	\mrid{271253}
	\orcid{0000-0003-0578-1283}}{S.~N.~Chandler-Wilde}
\emsauthor{2}{
	\givenname{Ratchanikorn}
	\surname{Chonchaiya}
	\mrid{964274}
	\orcid{0000-0003-1865-3790}}{R.~Chonchaiya}
\emsauthor{3}{
	\givenname{Marko}
	\surname{Lindner}
	\mrid{687128}
	\orcid{0000-0001-8483-2944}}{M.~Lindner}

\Emsaffil{1}{
	\department{Department of Mathematics and Statistics}
	\organisation{University of Reading}
	\rorid{05v62cm79}
	\zip{RG6 6AX}
	\city{Reading}
	\country{UK}
	\affemail{s.n.chandler-wilde@reading.ac.uk}}
\Emsaffil{2}{
	\department{Faculty of Science}
	\organisation{King Mongkut's University of Technology Thonburi}
	\rorid{0057ax056}
	\zip{10140}
	\city{Bangkok}
	\country{Thailand}
\affemail{hengmath@hotmail.com}
}
\Emsaffil{3}{
	\department{Institut f\"ur Mathematik}
	\organisation{TU Hamburg (TUHH)}
	\rorid{04bs1pb34}
	\zip{D-21073}
	\city{Hamburg}
	\country{Germany}
\affemail{lindner@tuhh.de}
}
\classification[47B35,46E40,47B80]{47A10}

\keywords{band matrix, band-dominated matrix, solvability complexity index, pseudoergodic}

\begin{abstract}
In this paper we derive novel families of inclusion sets for the spectrum and pseudospectrum of large classes of bounded linear operators, and establish convergence of particular sequences of these inclusion sets to the spectrum or pseudospectrum, as appropriate. Our results apply, in particular, to bounded linear operators on a separable Hilbert space that, with respect to some orthonormal basis, have a representation as a bi-infinite matrix that is banded or band-dominated. More generally, our results apply in cases where the matrix entries themselves are bounded linear operators on some Banach space. In the scalar matrix entry case we show that our methods, given the input information we assume, lead to a sequence of approximations to the spectrum, each element of which can be computed in finitely many arithmetic operations, so that, with our assumed inputs, the problem of determining the spectrum of a band-dominated operator has solvability complexity index one, in the sense of Ben-Artzi et al.\ [arXiv:1508.03280, 2020].
As a concrete and substantial application, we apply our methods to the determination of the spectra of non-self-adjoint bi-infinite tridiagonal matrices that are pseudoergodic in the sense of Davies [Commun.\ Math.\ Phys. 216 (2001) 687--704]. 
\end{abstract}

\maketitle

\begin{center}
{\it \large Dedicated to Prof E.~Brian Davies on the occasion of his 80th birthday.}
\end{center}

\tableofcontents

\section{Introduction and overview}

The determination of the spectra of bounded linear operators is a fundamental problem in functional analysis with applications across science and engineering (e.g., \cite{TrefEmbBook,Davies2007:Book,Hansen:nPseudo,ColbRomanHansen,Colb2022}). In this paper\footnote{This paper, which has been a long time in preparation, is based in large part on Chapters 3 and 4 of the 2010 PhD thesis of the second author \cite{HengThesis}, carried out under the supervision of the first and third authors. The title of earlier  drafts of this paper was ``Upper bounds on the spectra
and pseudospectra of Jacobi and related operators''. With this title this paper is referenced from other work,  that derives in part from the results and ideas in this paper, by one or more of the three of us with collaborators Hagger, Seidel, and Schmidt \cite{CW.Heng.ML:tridiag,LiSei:BigQuest,HagLiSei,LiSchmidt:Givens}. An announcement of some of the results of this paper has appeared in the conference proceedings \cite{PAMM}.} we establish novel families of inclusion sets for the spectrum and  pseudospectrum of large classes of bounded linear operators, and establish convergence of particular sequences of these inclusion sets to the spectrum or  pseudospectrum, as appropriate. Our results apply, in particular, to bounded linear operators $A$ on a separable Hilbert space $Y$ that, with respect to some orthonormal basis  $(e_i)_{i\in \Z}$ for $Y$, have a matrix representation $[a_{i,j}]_{i,j\in \Z}$, where $a_{i,j} := (Ae_j,e_i)$ with $(\cdot,\cdot)$ the inner product on $Y$, that is banded or band-dominated (as defined in \S\ref{sec:keynot} below). More generally, our results apply in the case that $Y$ is the Banach space $Y=\ell^2(\Z,X)$, where $X$ is a Banach space, and $A:Y\to Y$ is a bounded linear operator that has the banded or band-dominated matrix representation $[a_{ij}]_{i,j\in \Z}$, where each $a_{i,j}\in L(X)$, the space of bounded linear operators on $X$.

\subsection{The main ideas, their provenance and significance} \label{sec:ideas} 
Let us sketch the key ideas of the paper, deferring a more detailed account of the main results to \S\ref{sec:main}. 
To explain these ideas, focussing first on the case that $A=[a_{i,j}]_{i,j\in \Z}$ is  a bi-infinite tridiagonal matrix, let $A(i{:}j,k{:}\ell)$ denote the $(j{+}1{-}i)\times (\ell{+}1{-}k)$ submatrix (or finite section) of $A$ consisting of the elements in rows $i$ through $j$ and columns $k$ through $\ell$. Abbreviate the $n\times n$ principal submatrix $A(k{+}1{:}k{+}n,k{+}1{:}k{+}n)$ as $A_{n,k}$, and the $(n{+}2) \times n$ submatrix  $A(k{:}k{+}n{+1},k{+}1{:}k{+}n)$ as $A^+_{n,k}$ (see Figure \ref{fig:matrices}). Further, let $A_n:= A_{2n+1,-n-1}$ and $A^+_n:= A^+_{2n+1,-n-1}$, so that $A^+_n$ contains all the non-zero elements in columns $-n$ through $n$. Similarly, for $x=(x_i)_{i\in \Z}$,  let $x_{n,k}$ denote the (column) vector of length $n$ consisting of entries $k+1$ through $k+n$ of $x$.

Suppose that we wish to compute an approximation to $\Spec A$, the spectrum of a bi-infinite tridiagonal matrix $A$
acting as a bounded linear operator on $\ell^2(\Z)$. It is well known that if  $\lambda \in \Spec A$ then, for every $\eps>0$, there exists $x\in \ell^2(\Z)$ such that $x$ is an {\em $\eps$-pseudoeigenvector of $A$ for $\lambda$}, meaning that $x\neq 0$ and $\|(A-\lambda I)x\| \leq \eps \|x\|$, or there exists an  $\eps$-pseudoeigenvector of $A^*$ for $\lambda$, where $A^*$ is the matrix that is the transpose of $A$.  Our first key observation is the following: 
\begin{quotation}
\noindent {\em If $x$ is an $\eps$-pseudoeigenvector of $A$ for $\lambda$ then, for each $n\in \N$, there is a computable $\eps_n> 0$, with $\eps_n\to0$ as $n\to\infty$, such that $x_{n,k}$ is an $(\eps+\eps_n)$-pseudoeigenvector of $A_{n,k}$ for $\lambda$, for some $k\in \Z$.}
\end{quotation}
This observation implies that, for some $k\in \Z$, $\Spec A \subset \Speps A_{n,k}$, the (closed) $\eps$-pseudospectrum of the matrix $A_{n,k}$ (see \S\ref{sec:keynot} for definitions). Indeed, this inclusion holds for all $\eps>0$, which we will see implies that
\begin{equation} \label{eq:incl}
\Spec A\ \subset\ \Sigma_{\eps_n}^n(A)\ :=\ \overline{\bigcup_{k\in  \Z} \Specn_{\eps_n} A_{n,k}},
\end{equation}
for some positive null sequence $(\eps_n)_{n\in \N}$.
Optimising the above idea (see \S\ref{sec:incl}) leads, concretely,  to the conclusion that \eqref{eq:incl} holds with $\eps_n = r(A)\eta_n$, where 
\begin{equation} \label{eq:rA}
r(A)\ :=\ \sup_{i\in \Z}|a_{i,i+1}|+\sup_{i\in \Z}|a_{i,i-1}| \quad \mbox{and} \quad 2\sin\left(\frac{\pi}{4n+2}\right)\ \leq\ \eta_n\ \leq\ 2\sin\left(\frac{\pi}{2n+4}\right).
\end{equation}

In terms of provenance, the above observation and  \eqref{eq:incl} are reminiscent of the classical Gershgorin theorem \cite{Gershgorin,Varga} that provides an enclosure for the eigenvalues of an arbitrary $n\times n$ matrix.  Indeed, in the simplest case $n=1$,  $A_{n,k}=[a_{k{+}1,k{+}1}]$ is the $(k{+}1)$th diagonal entry of the matrix $A$, the above inequality implies that $\eta_n=1$, and $\Specn_{\eps_n} A_{n,k}$ is just $a_{k{+}1,k{+}1} + r(A) \overline{\D}$, the closed disc of radius $r(A)$ centred on $a_{k{+}1,k{+}1}$. Thus \eqref{eq:incl} in the case $n=1$ is  simply
\begin{equation} \label{eq:Gersch}
\Spec A\ \subset\ \overline{\bigcup_{k\in  \Z} (a_{k,k} + r(A) \D)},
\end{equation}
an enclosure of $\Spec A$ by discs centred on the diagonal entries of $A$ that is an immediate consequence of (an infinite matrix version) of the classical Gershgorin theorem\footnote{An infinite matrix version of the Gershgorin theorem that implies \eqref{eq:Gersch} is proved as \cite[Theorem 2.50]{HengThesis}. Alternatively, the inclusion \eqref{eq:Gersch} follows by applying \eqref{eq:spepspert2} below, with $B$ the diagonal of $A$, noting that $\|A-B\|\leq r(A)$ and $\Specn_{r(A)} B = \Spec B + r(A)\overline{\D}$.}. So our enclosures \eqref{eq:incl} can be seen as an extension of Gershgorin's theorem to create a whole family of inclusion sets for the spectra of bi-infinite tridiagonal matrices.

An attraction of a sequence $(\Sigma_{\eps_n}^n(A))_{n\in \N}$ of inclusion sets is that one can think of taking the limit. We will see significant examples below, including well-studied examples of tridiagonal pseudoergodic matrices in the sense of Davies \cite{Davies2001:PseudoErg}, for which $\Sigma_{\eps_n}^n(A)\Hto \Spec A$ as $n\to\infty$ ($\Hto$ is the Hausdorff convergence that we recall in \S\ref{sec:keynot}). But, in general, the sequence $(\Sigma_{\eps_n}^n(A))_{n\in \N}$, based on computing spectral quantities of square matrix finite sections of $A$, suffers from {\em spectral pollution} \cite{DaviesPlum}: there are limit points of the sequence $(\Sigma_{\eps_n}^n(A))$ that are not in $\Spec A$. 

The second key idea, dating back to Davies \& Plum \cite[pp.~423-434]{DaviesPlum} (building on Davies \cite{Davies1998:Encl}) for the self-adjoint case,  and Hansen \cite[Theorem 48]{HansenJFA08} (proved in \cite{Hansen:nPseudo}) and \cite{Hansen:nPseudo} for the general case (see also \cite{HeinPotLind08,SCI,ColbRomanHansen,ColHan2023}), is to avoid spectral pollution by working with rectangular rather than square finite sections. In our context this means to replace the $n\times n$ matrix $A_{n,k}$ by the $(n+2)\times n$ matrix $A_{n,k}^+$, containing all the non-zero entries of $A$  in columns $k+1$ to $k+n$. It is a standard characterisation of the pseudospectrum (see \eqref{eq:spepsnu} and \eqref{eq:lowernorm2}, or \cite[{\S}I.2]{TrefEmbBook}) that 
$$
\Specn_{\eps_n} A_{n,k}\ =\ \left\{\lambda\in \C: \min\left(s_{\min}((A-\lambda I)_{n,k}), s_{\min}((A^*-\lambda I)_{n,k}\right)\leq \eps_n\right\},
$$
where $s_{\min}$ denotes the smallest singular value. This implies (with a little work, see Proposition \ref{prop:same}) that
$$
\Sigma^n_{\eps_n}(A)\ =\ \left\{\lambda\in \C: \inf_{k\in \Z}\,\min\left(s_{\min}((A-\lambda I)_{n,k}), s_{\min}((A^*-\lambda I)_{n,k}\right)\leq \eps_n\right\}.
$$
Replacing square with rectangular finite sections in the above expression gives an alternative sequence of approximations to $\Spec A$, viz.~$(\Gamma_{\eps''_n}^n(A))_{n\in \N}$, where
\begin{equation} \label{eqGn}
\Gamma^n_{\eps''_n}(A)\ :=\ \left\{\lambda\in \C: \inf_{k\in \Z}\,\min\left(s_{\min}((A-\lambda I)^+_{n,k}), s_{\min}((A^*-\lambda I)^+_{n,k}\right)\leq \eps''_n\right\}.
\end{equation}
Again $\Spec A \subset \Gamma^n_{\eps''_n}(A)$ for each $n$ (equation \eqref{incl:met2}), provided $\eps''_n$ is large enough. (We show in \S\ref{sec:tau1} that if $\eps''_n:= r(A)\eta''_n$ then $\eta''_n = 2\sin(\pi/(2n+2))$ is the smallest choice of $\eta''_n$ that maintains this inclusion for all tridiagonal $A$.) Crucially, and this is the benefit of replacing square by rectangular finite sections, also $\Gamma^n_{\eps''_n}(A)\Hto \Spec A$ as $n\to\infty$ (Theorem \ref{thm:converge}),  for every tridiagonal $A$.

The third idea is to realise that the entries in the tridiagonal matrix can themselves be matrices (or indeed operators on some Banach space). This extends the above results to general banded matrices (see \S\ref{sec:main} below). Further, by perturbation arguments, we can obtain inclusion sets, and  a sequence of approximations converging to the spectrum, also for the band-dominated case (\S\ref{sec:bdo}).

The final main idea of the paper is that, through a further perturbation argument (see \S\ref{sec:comput}), the union and infimum over all $k\in \Z$, in \eqref{eq:incl} and \eqref{eqGn}, can be replaced by a union and infimum over $k\in K_n$, where $K_n\subset \Z$ is a finite, $A$- and $n$-dependent subset. Thus, provided that one has access through some oracle to the sets $K_n$, for $n\in \N$,  a sequence of approximations to $\Spec A$ is obtained that can be computed in a finite number of arithmetic operations. In the language of the  solvability complexity index (SCI) of \cite{SCIshort,SCI,Colb2022,ColHan2023,Matt}  these results (see \S\ref{sec:SCI} and \S\ref{sec:scalar}) lead to a proof that the computational problem of determining the spectrum of a band-dominated operator has SCI equal to one, if one chooses an evaluation set (in the sense of \cite{SCI,Colb2022,ColHan2023}) that provides access to appropriately chosen  submatrices, i.e.\ to $A^+_{n,k}$  and $(A^*)^+_{n,k}$, for $n\in \N$ and $k\in K_n$. To showcase  this result we provide a concrete realisation of the associated algorithm (concrete in the sense that the sets $K_n$ can be constructed) for  the class of tridiagonal matrices that are pseudoergodic in the sense of Davies \cite{Davies2001:PseudoErg} (see \S\ref{sec:pseudoergodic}), illustrating this algorithm by numerical results taken from \cite{CW.Heng.ML:tridiag}.

To clarify the significance of this combination of ideas, let us make comparison to the classical idea of approximating the spectrum or pseudospectrum of $A$ by those of a single large finite section of $A$. To avoid spectral pollution we take a rectangular finite section: precisely (cf.~\eqref{eqGn}), consider the approximation 
\begin{equation} \label{eq:specfs}
\mathcal{S}^n_\eps(A)\ :=\ \left\{\lambda\in \C: \min\left(s_{\min}((A-\lambda I)^+_{n}), s_{\min}((A^*-\lambda I)^+_{n})\right)\leq \eps\right\},
\end{equation}
defined in terms of smallest singular values of $2n{+}3\times 2n{+}1$ finite sections of $A$ and $A^*$, which is a particular instance of the approximation sequence proposed by Hansen  \cite{HansenJFA08}. Then it follows from \cite[Theorem 48]{HansenJFA08} (the proof deferred to \cite{Hansen:nPseudo}) that, for every $\eps>0$,
$\mathcal{S}^n_\eps(A)\Hto\Speps A$ as $n\to\infty$. Since also $\Speps A\Hto \Spec A$ as $\eps\to 0$, we have that
$$
\Spec A = \lim_{\eps\to0} \, \lim_{n\to\infty}\, \mathcal{S}^n_\eps(A).
$$

This result trivially implies that there exists a positive null sequence $(\eps_n)_{n\in \N}$ such that $\mathcal{S}^n_{\eps_n}(A)\Hto$ $\Spec A$ as $n\to\infty$. But, in contrast to the new approximation families that we introduce in this paper, this null sequence $(\eps_n)$ depends on $A$, in an ill-defined, and unquantifiable way. Indeed, the SCI result arguments  of \cite{SCI}, that we discuss in \S\ref{sec:SCI} and \S\ref{sec:other}, imply that there exists no sequence of universal algorithms, indexed by $n\in \N$, operating on the set of all tridiagonal matrices,  such that the $n$th algorithm takes just finitely many values of the matrix entries $a_{i,j}$ as input and gives $\eps_n$ as output, in  such a way that $\mathcal{S}^n_{\eps_n}(A)\Hto\Spec A$ for all tridiagonal $A$ \cite{Matt}. In the language of \cite{SCI}, the SCI of computing the spectrum of a tridiagonal operator is two \cite{Matt} if the evaluation set is restricted to the mappings $A\mapsto a_{i,j}$, for $i,j\in \Z$ (i.e., to only {\it local information} about the matrix $A$). In contrast, our results in this paper show that the SCI is one if we expand the evaluation set to include also the mappings $A\mapsto K_n$ for $n\in \N$ (in some sense {\it global information} about the matrix $A$). We are hopeful that changing the evaluation set in analogous ways, to include global as well as local information, should lead to reductions in SCI for a broad range of problems, in particular for the  computation of a range of other spectral quantities (cf.~\cite{SCI,ColbRomanHansen,Colb2022,ColHan2023}).


\subsection{The structure of the paper} \label{sec:struct}
In the next subsection we briefly recall key notations and definitions necessary to read our main results. In \S\ref{sec:main} we give an overview of  these main results. We state our results on inclusion sets for the spectrum and pseudospectrum in the case that the matrix representation of $A$ is banded  in \S\ref{sec:incl}, prove our results on convergence of our inclusion set sequences to the spectrum and pseudospectrum for the banded case in \S\ref{sec:conv}, detail the extension to the band-dominated case in \S\ref{sec:bdo}, and summarise our results on computability and the solvability complexity index in \S\ref{sec:comput} and \S\ref{sec:SCI}. In \S\ref{sec:exam} we give a first example (the shift operator) that serves to illustrate our inclusions \eqref{incl:met1}, \eqref{incl:met1*} and \eqref{incl:met2} below and demonstrate their sharpness, and in \S\ref{sec:other} we survey related work, expanding on the discussion above.

In \S\ref{sec:pseud} we detail further properties of pseudospectra needed for our later arguments and introduce the so-called Globevnik property. Sections \ref{sec:tau}--\ref{sec:tau1} prove the inclusions  \eqref{incl:met1}, \eqref{incl:met1*}, and \eqref{incl:met2}, respectively. In \S\ref{sec:thmgen} we prove Theorem \ref{thm:fin} that we state in \S\ref{sec:comput}, that justifies replacement of the infinite set of sub-matrices that are part of the definitions of our inclusion sets with a finite subset. This result is extended to the general band-dominated case in Theorem \ref{thm:fin2}.  In \S\ref{sec:scalar} we give details of implementation, leading to our solvability complexity index results. We  discuss concrete applications of our general results in \S\ref{sec:pseudop}, notably to the case where $A$ is tridiagonal and pseudoergodic in the sense of Davies \cite{Davies2001:PseudoErg}. In \S\ref{sec:ext} we indicate open problems and directions for further work. 

\subsection{Key notations} \label{sec:keynot} 
As usual, we write $\Z,\,\R,\,\C$ for the integer, real, and complex numbers, $\N\subset\Z$ is the set of positive integers,
$\D\subset\C$ the open unit disk, $\T\subset \C$ the unit circle, and, for $n\in\N$, $\T_n\subset \T$ denotes the set of
$n$th roots of unity, $\T_n = \{z\in \C:z^n=1\}$. $\overline S$ denotes the closure of a set $S\subset\C$ and $\bar z$ denotes the complex conjugate of $z\in \C$. Throughout, $X$ will denote some Banach space: at particular points in the text we will make clear where we are specialising to the case that $X$ is a Hilbert space or is finite-dimensional. Further,
\begin{equation} \label{Edef}
E:=\ell^2(\Z,X)
\end{equation}
will denote the space of bi-infinite
$X$-valued sequences $x=(x_j)_{j\in\Z}$ that are
square-summable, a Banach space with the norm $\|\cdot\|$ defined
by $\|x\| = (\sum_{j\in\Z} |x_j|_X^2)^{1/2}$. We will abbreviate $\ell^2(\Z,\C)$ as $\ell^2(\Z)$.

\paragraph{The spectrum, pseudospectrum, and lower norm} Let $Y$ be a Banach space (throughout the paper our Banach and Hilbert spaces are assumed complex). Given $A\in L(Y)$, the space of bounded linear operators on $Y$, we denote the spectrum of $A$ by $\Spec A $. With the convention that, for $B\in L(Y)$, $\|B^{-1}\|:= \infty$ if $B$ is not invertible, and that $1/\infty:=0$, and where $I$ is the identity operator on $Y$, we define the {\em closed $\eps$-pseudospectrum},
$\Specn_\eps A$, by
\begin{equation} \label{eq:pse}
\Specn_\eps A := \{\lambda\in \C: 1/\|(A-\lambda I)^{-1}\|\leq \eps\}, \quad \eps>0.
\end{equation}
Clearly $\Spec A \subset \Specn_\eps A\subset \Specn_{\eps^\prime}A$, for $0<\eps<\eps^\prime$, and 
\begin{equation} \label{eq:speceps}
\Spec A = \bigcap_{\eps>0} \Specn_\eps A .
\end{equation}

To write certain inclusion statements more compactly, we will define also $\Specn_0A:= \Spec A$. For $\eps>0$ we define also the {\it open $\eps$-pseudospectrum}\footnote{We refer readers unfamiliar with the pseudospectrum, its properties and applications, to \S\ref{sec:pseud} below, and to B\"ottcher \& Lindner \cite{BoeLi:Pseudospectrum}, Davies \cite{Davies2007:Book}, Trefethen \& Embree \cite{TrefEmbBook}, and Hagen, Roch, Silbermann \cite{HaRoSi2}. We note that the $\eps$-pseudospectrum, as defined in \cite{TrefEmbBook,Davies2007:Book} and most recent literature, is our open $\eps$-pseudospectrum, but in part of the literature, including  \cite{HaRoSi2}, $\eps$-pseudospectrum means our closed $\eps$-pseudospectrum. As we recall in \S\ref{sec:pseud}, if the Banach space $X$ has the so-called {\em Globevnik property}, which holds in particular if $X$ is a Hilbert space or is finite-dimensional, these two pseudospectra are related simply by $\Speps(B)=\overline{\speps(B)}$, for $\eps>0$ and $B\in L(X)$.},
$\specn_\eps A$, given by \eqref{eq:pse} with the $\leq$ replaced by $<$; the above inclusions and \eqref{eq:speceps} hold also with $\Specn_\eps A$ replaced by $\specn_\eps A$. As a consequence of standard perturbation arguments (e.g., \cite[Theorem 1.2.9]{Davies2007:Book}), $\Spec_\eps A$ is closed, for $\eps\geq 0$, and $\speps A$ open, for $\eps>0$. 

Given Banach spaces $X$ and $Y$ and $B\in L(X,Y)$, the space of bounded linear operators from $X$ to $Y$, one can study, as a counterpart to the operator norm $\|B\|:=\sup_{\|x\|=1}\|Bx\|$, the quantity
\[
\nu(B)\ :=\ \inf_{\|x\|=1}\|Bx\|
\]
that is sometimes (by abuse of notation) called the {\sl lower norm} (and sometimes, e.g.~in \cite{Seidel:Neps}, the {\sl injection modulus}) of $B$. In the case that $Y$ is a Hilbert space and $A\in L(Y)$, $\|A\|$ is the largest singular value
of $A$ and $\nu(A)$ is the smallest\footnote{Recall that the singular values of $A\in L(Y)$, defined whenever $Y$ is a Hilbert space, are the points in the spectrum of $(A'A)^{1/2}$, where $A'$ is the Hilbert space adjoint.}. For every Banach space $Y$ and $A\in L(Y)$ the equality
\begin{equation} \label{eq:invNorm}
1/\|A^{-1}\|\ =\min(\nu(A),\nu(A^*)) =: \mu(A),
\end{equation}
holds, where $A^*$ is the Banach space adjoint (see below). 
In particular, $A$ is invertible if and only if $\nu(A)$ and $\nu(A^*)$ are both nonzero, i.e., if and only if $\mu(A)\neq 0$, in which case
$\nu(A)=\nu(A^*)=\mu(A)$. Further, if $A$ is Fredholm of index zero, in particular if $Y$ is finite-dimensional, then $\nu(A)=0$ if and only if $\nu(A^*)=0$, so that $\mu(A)=\nu(A)$.

From the definitions of $\speps A$ and $\Speps A$ and \eqref{eq:invNorm} it follows that
\begin{equation} \label{eq:spepsnu}
\begin{aligned}
\speps A\ &=\ \{\lambda\in\C:\mu(A-\lambda I)<\eps\} , \;\; \eps>0, \quad \mbox{and}\\ 
\Speps A\ &=\ \{\lambda\in\C:\mu(A-\lambda I)\leq\eps\}, \;\; \eps\geq 0.
\end{aligned}
\end{equation}
Simple but important properties of the lower norm are that, if $A,B\in L(X,Y)$ and $Z$ is a closed subspace of $X$, then 
\begin{equation} \label{eq:lnProp}
\nu(A|_Z) \geq \nu(A)\qquad \mbox{and} \qquad |\nu(A)-\nu(B)|\leq \|A-B\|;
\end{equation}
see, e.g., \cite[Lemma 2.38]{LiBook} for the second of these properties. It is easy to see (since $\mu(A)=\nu(A)$ if $\mu(A)\neq 0$) that also
\begin{equation} \label{eq:lnPro2}
|\mu(A)-\mu(B)|\leq \|A-B\|.
\end{equation}

\paragraph{Hausdorff distance and notions of set convergence}
Let $\C^B$, $\C^C$  denote, respectively, the sets of bounded and compact non-empty subsets of $\C$. For $a\in \C$ and $S\in \C^B$, let $\dist(a,S) := \inf_{b\in S}|a-b|$. For $S_1,S_2\subset \C^B$ let
\[
d_H(S_1,S_2)\ :=\ \max\left( \sup_{s_1\in S_1} \dist(s_1,S_2)\,,\, \sup_{s_2\in S_2}\dist(s_2,S_1)\right),
\]
which is termed the {\em Hausdorff distance} between $S_1$ and $S_2$. It is well known (e.g., \cite{HausBook,HaRoSi2}) that $d_H(\cdot,\cdot)$ is a metric on $\C^C$, the so-called {\em Hausdorff metric}. For $S_1,S_2\in \C^B$ it is clear that $d_H(S_1,S_2)=d_H(\overline{S_1},\overline{S_2})$, so that $d_H(S_1,S_2)=0$ if and only if $\overline{S_1}=\overline{S_2}$, and $d_H(\cdot,\cdot)$ is a pseudometric on $\C^B$. 
For a sequence $(S_n)\subset \C^B$ and $S\in \C^B$ we write $S_n\Hto S$ if $d_H(S_n,S)\to 0$. This limit is in general not unique: if $S_n\Hto S$ and $T\in \C^B$ then $S_n\Hto T$ if and only if $\overline{S}=\overline{T}.$

Simple results that we will use extensively are that if $(S_n) \subset \C^C$ and $S_1\supset S_2\supset \ldots$, in which case $S := \cap_{n=1}^\infty S_n$ is non-empty, then $S_n\Hto S$, and that if $S\in \C^B$, $(S_n), (T_n)\subset \C^B$, $S\subset T_n\subset S_n$ and $S_n\Hto S$, then $T_n\Hto S$.  (These results can be seen directly, or via the equivalence of Hausdorff convergence with other notions of set convergence, see \cite[p.~171]{HausBook} or \cite[Proposition 3.6]{HaRoSi2}.)

\paragraph{Band and band-dominated operators} 
Let $A\in L(E)$, in which case (e.g., \cite[\S1.3.5]{LiBook}) $A$ has a matrix representation $[A]=[a_{i,j}]$, that we will denote again by $A$, with each $a_{i,j}\in L(X)$, such that, for every $x=(x_j)_{j\in \Z}\in E$ with finitely many non-zero entries, $Ax = y=(y_i)_{i\in \Z}$, where 
\begin{equation} \label{eq:matmult}
y_i = \sum_{j\in \Z} a_{i,j}x_j, \quad i\in \Z.
\end{equation}
 The main body of our results are for the case where $A\in BO(E)$, where $BO(E)$ denotes the linear subspace of those  $A\in L(E)$ whose matrix representation is {\em banded} with some {\em band-width $w$}, meaning that $a_{i,j}=0$ for $|i-j|>w$, in which case \eqref{eq:matmult} holds (and is a finite sum) for all $x\in E$. By perturbation arguments we also extend our results to the case where $A\in BDO(E)$, where $BDO(E)$, the space of {\em band-dominated operators}, is the closure of $BO(E)$ in $L(E)$ with respect to the operator norm (see, e.g., \cite[\S1.3.6]{LiBook}). 

\paragraph{Adjoints} In many of our arguments we require the (Banach space) adjoint $A^*$ of $A$ which we think of as an operator on $E^*:=\ell^2(\Z,X^*)$, where $X^*$ is the dual space of $X$. This has the matrix representation $[a_{j,i}^*]$, where $a_{j,i}^*\in L(X^*)$ is the Banach space adjoint of $a_{j,i}\in L(X)$. In the case that $X=\C^n$, equipped with the $\ell^p$ norm, with $1\leq p\leq \infty$, we will identify $X^*$ with $\C^n$ equipped with the $\ell^q$ norm, where $q$ is the usual conjugate index ($1/p+1/q=1$), so that $a_{j,i}\in \C^{n\times n}$ is just an $n\times n$ matrix with scalar entries, and $a_{j,i}^*$ is just $a_{j,i}^T$, the transpose of the matrix $a_{j,i}$. 

\subsection{Our main results} \label{sec:main}

Having introduced key notations and definitions, let us now state our main results.

\paragraph{Reduction to the tridiagonal case}  A simple but important observation is that, to construct inclusion sets for the spectrum and pseudospectrum of $A\in BO(E)$ it is enough to consider the case when the matrix representation has band-width one, so that the matrix is tridiagonal\footnote{In the case $X=\C$, in which case $E=\ell^2(\Z)$ is a Hilbert space, and if $A$ is self-adjoint, reduction of computation of $\Spec A$ (and the pseudospectra of $A$) to consideration of the tridiagonal case is discussed in another sense 
in \cite[Corollary 26]{HansenJFA08}, viz.~the sense of reduction to tridiagonal form by application of a sequence of Householder transformations (which preserve spectrum and pseudospectra).}. For suppose that  $A$ is banded with band-width $w\in \N$ and, for $b\in \N$, let $\mathcal{I}_b:E\to E_b:=\ell^2(\Z,X^b)$ be the mapping $x=(x_i)_{i\in \Z}\mapsto \xi = (\xi_i)_{i\in \Z}$, with $\xi_i=(x_{bi+1},\ldots, x_{b(i+1)})\in X^b$. Then $\mathcal{I}_b$ is an isomorphism, indeed an isometric isomorphism if we equip the product space $X^b$ with the norm 
\begin{equation} \label{eq:norm}
\|(\xi_1,\ldots,\xi_b)\|_{X^b} := \left(\sum_{j=1}^b \|\xi_j\|_X^2\right)^{1/2}
\end{equation}
(which we assume hereafter). Thus $A\in L(E)$ and $A_b\in L(E_b)$, given by $A_b=\mathcal{I}_bA\mathcal{I}_b^{-1}$, share the same spectrum and pseudospectra and, provided $b\geq w$, the matrix representation of $A_b$ is tridiagonal.

Thus in the statement of our main results we focus on the case where $A\in L(E)$ has a  matrix representation that is tridiagonal, so that, introducing the abbreviations $\alpha_j:=a_{j+1,j}$, $\beta_j:= a_{j,j}$, $\gamma_j := a_{j-1,j}$,
\begin{equation} \label{eq:A}
A\ =\ \left(\begin{array}{ccccccc} \ddots&\ddots\\
\ddots&\beta_{-2}&\gamma_{-1}\\
&\alpha_{-2}&\beta_{-1}&\gamma_{0}\\
\cline{4-4}
&&\alpha_{-1}&\multicolumn{1}{|c|}{\beta_0}&\gamma_1\\
\cline{4-4}
&&&\alpha_{0}&\beta_1&\gamma_2\\
&&&&\alpha_1&\beta_2&\ddots\\
&&&&&\ddots&\ddots
\end{array}\right),
\end{equation}
where the box encloses  $\beta_0=a_{0,0}$. As a special case of \eqref{eq:matmult},  $y:=Ax\in E$
has $i$th entry $y_i\in X$ given by
\[
y_i \ =\  \alpha_{i-1} x_{i-1} \, +\, \beta_i x_i \, +\, \gamma_{i+1} x_{i+1}, \qquad i \in\Z.
\]
It is easy to see that the mapping $A$ given by this rule is bounded, i.e.\ $A\in L(E)$, if and only if $(\alpha_i)$,
$(\beta_i)$, and $(\gamma_i)$ are bounded sequences in $L(X)$, i.e.\ are sequences in $\ell^\infty(\Z,L(X))$, a Banach space equipped with the norm $\|\cdot\|_\infty$ defined by $\|x\|_\infty := \sup_{i\in \Z}\|x_i\|_{L(X)}$, 
in which case
\begin{equation} \label{eq:normbound2}
\|A\| \leq r(A) + \|\beta\|_\infty, \quad \mbox{ where } r(A) := \|\alpha\|_\infty+\|\gamma\|_\infty.
\end{equation}
Conversely, 
 noting that $\|a_{i,j}\|_{L(X)}\leq \|A\|$, for every $i,j\in \Z$,
\begin{equation} \label{eq:normbound}
r(A) = \|\alpha\|_\infty+\|\gamma\|_\infty\ \leq\ 2\|A\|. 
\end{equation}

\subsubsection{Our spectral inclusion sets for the tridiagonal case: the $\tau$, $\pi$, and $\tau_1$ methods} \label{sec:incl}
Since to derive spectral inclusion sets for $A\in BO(E)$ it is enough to study the tridiagonal case, we restrict to this case in this subsection (and in much of the rest of the paper); the action of $A$ is that of multiplication by the tridiagonal matrix \eqref{eq:A}. In this paper we derive for this tridiagonal case three different families  of
spectral inclusion sets that enclose the spectrum or
pseudospectra of the operator $A$. Each spectral inclusion set is defined as the 
union of the
pseudospectra (or related sets) of certain finite matrices. These
finite matrices are either principal submatrices of the infinite matrix \eqref{eq:A}  (we call the corresponding approximation method
the ``$\tau$ method'', where $\tau$ is for {\em truncation}), or are circulant-type
modifications of these submatrices (the ``$\pi$ method'', $\pi$ for {\em periodised truncation}), or they are connected to
infinite submatrices with finite row- or column number (this is our ``$\tau_1$ method'', $\tau_1$ for {\em one-sided truncation}). 

In each method, as discussed below \eqref{eq:epsn}, there is a ``penalty term'', $\eps_n$, $\eps'_n$, and $\eps''_n$, for the $\tau$, $\pi$, and $\tau_1$ methods, respectively, which arises from replacing infinite matrices by finite matrices. As will be demonstrated by the simple example where $A$ is the shift operator in \S\ref{sec:exam}, our values for $\eps_n$, $\eps'_n$, and $\eps''_n$ are optimal; there is at least one tridiagonal $A\in L(E)$ such that, if we make any of these values smaller, the claimed inclusions fail. In this limited sense our inclusion sets are sharp.

To define these methods, 
 for $k\in\Z$ and $n\in\N$, let $P_{n,k}:E\to E$ be the projection operator given by
\begin{equation} \label{eq:Pdef}
(P_{n,k}x)_j\ =\ \left\{ \begin{array}{cc}
                         x_j, & j \in \{k+1,k+2, \dots, k+n\}, \\
                         0, & \mbox{otherwise},
                       \end{array}\right.
\end{equation}
and let $E_{n,k} := P_{n,k}(E)$ be the 
range of $P_{n,k}$.
In the $\tau$ and $\pi$ methods, we construct inclusion sets for
the spectrum and pseudospectrum of $A$ from those of the finite
section operators $A_{n,k} := P_{n,k}AP_{n,k}|_{E_{n,k}}:E_{n,k}\to
E_{n,k}$ and their ``periodised'' versions $A^{\per,t}_{n,k}:E_{n,k}\to
E_{n,k}$, with $t\in\T$, so that, for $n\in \N$, $k\in \Z$, $A_{n,k}$ acts by multiplication by the $n\times
n$ principal submatrix
\begin{equation} \label{eq:Ank}
A_{n,k}\ =\ \left(\begin{array}{ccccc}
\beta_{k+1} & \gamma_{k+2}& & & \\
\alpha_{k+1} & \beta_{k+2} & \gamma_{k+3} & & \\
 & \ddots & \ddots & \ddots & \\
 & & \alpha_{k+n-2} & \beta_{k+n-1} & \gamma_{k+n} \\
 & & & \alpha_{k+n-1} & \beta_{k+n} \end{array} \right)
\end{equation}
of $A$ and $A^{\per,t}_{n,k}$ acts by multiplication by the matrix $A^{\per,t}_{n,k}:= A_{n,k}+B_{n,k}^t$, where $B_{n,k}^t$ is the $n\times n$ matrix whose entry in row $i$, column $j$ is $\delta_{i,1}\delta_{j,n} t\alpha_k+ \delta_{i,n}\delta_{j,1}\bar t\gamma_{k+n+1}$, where $\delta_{ij}$ is the Kronecker delta, so that, for $n\geq 3$,
\begin{equation} \label{eq:Ankper}
A^{\per,t}_{n,k} = \left(\begin{array}{ccccc}
\beta_{k+1} & \gamma_{k+2}& & & t\alpha_{k} \\
\alpha_{k+1} & \beta_{k+2} & \gamma_{k+3} & & \\
 & \ddots & \ddots & \ddots & \\
 & & \alpha_{k+n-2} & \beta_{k+n-1} & \gamma_{k+n} \\
 \bar t\gamma_{k+n+1}& & & \alpha_{k+n-1} & \beta_{k+n} \end{array}
 \right)
\end{equation}
(see Figure \ref{fig:matrices} for visualisations of these finite section matrices).

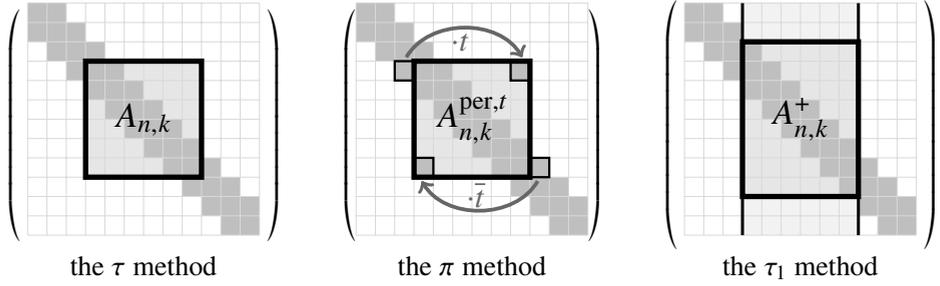
\begin{figure}[t]
\begin{center}
\begin{tabular}{cp{0mm}cp{0mm}c}
$
\begin{pmatrix}~
\begin{tikzpicture}[scale=0.51]
\fill[gray!20] (1.5,4.5) rectangle ++(3,-3);
\foreach \k in {0,0.5,...,5}
  \fill[gray!50] (\k,6-\k) rectangle ++(1.0,-1.0);
\draw[gray!30,very thin,step=0.5cm] (0,0) grid (6,6);
\draw[black,line width=0.7mm] (1.5,4.5) rectangle ++(3,-3);
\node at (3,3) {\Large $A_{n,k}$};
\end{tikzpicture}
~\end{pmatrix}
$
&&
$
\begin{pmatrix}
\begin{tikzpicture}[scale=0.51]
\fill[gray!20] (1.5,4.5) rectangle ++(3,-3);
\foreach \k in {0,0.5,...,5}
  \fill[gray!50] (\k,6-\k) rectangle ++(1.0,-1.0);
\fill[gray!50] (4.5,4.5) rectangle ++(-0.5,-0.5);
\fill[gray!50] (1.5,1.5) rectangle ++(0.5,0.5);
\draw[gray!30,very thin,step=0.5cm] (0,0) grid (6,6);
\draw[black,line width=0.3mm] (4.5,4.5) rectangle ++(-0.5,-0.5);
\draw[black,line width=0.3mm] (1.5,1.5) rectangle ++(0.5,0.5);
\draw[black,line width=0.3mm] (4.5,1.5) rectangle ++(0.5,0.5);
\draw[black,line width=0.3mm] (1.5,4.5) rectangle ++(-0.5,-0.5);
\draw[,line width=0.7mm] (1.5,4.5) rectangle ++(3,-3);
\draw[->,black!60,line width=0.5mm] (4.7,1.4) to[bend angle=-60,bend right] ++(-3,0);
\draw[->,black!60,line width=0.5mm] (1.3,4.6) to[bend angle=-60,bend right] ++(3,0);
\node at (3,3) {\Large $A_{n,k}^{{\rm per},t}$};
\node[black!60] at (2.7,5.0) {\large $\cdot t$};
\node[black!60] at (3.1,1.05) {\large $\cdot \overline t$};
\end{tikzpicture}
\end{pmatrix}
$
&&
$
\begin{pmatrix}~
\begin{tikzpicture}[scale=0.51]
\fill[gray!10] (1.5,6) rectangle ++(3,-6);
\fill[gray!20] (1.5,5) rectangle ++(3,-4);
\foreach \k in {0,0.5,...,5}
  \fill[gray!50] (\k,6-\k) rectangle ++(1.0,-1.0);
\draw[gray!30,very thin,step=0.5cm] (0,0) grid (6,6);
\draw[black,line width=0.7mm] (1.5,5) rectangle ++(3,-4);
\draw[black,line width=0.4mm] (1.5,6) -- ++(0,-6);
\draw[black,line width=0.4mm] (4.5,6) -- ++(0,-6);
\node at (3,3) {\Large $A^+_{n,k}$};
\end{tikzpicture}
~\end{pmatrix}
$
\\
the $\tau$ method
&&
the $\pi$ method
&&
the $\tau_1$ method
\end{tabular}
\end{center}
\caption{The finite section matrices $A_{n,k}$, $A_{n,k}^{\per, t}$, and $A_{n,k}^+$, given by \eqref{eq:Ank}, \eqref{eq:Ankper}, and \eqref{eq:An+}, respectively, that arise in the $\tau$, $\pi$, and $\tau_1$ methods, respectively.} \label{fig:matrices}
\end{figure}

\paragraph{The $\tau$ method} Defining\footnote{In the following equation $\speps A_{n,k}$ denotes the open $\eps$-pseudospectrum of the operator $A_{n,k}:E_{n,k}\to
E_{n,k}$, where $E_{n,k}\subset E$ is equipped with the norm of $E$ or, what is the same thing, the open $\eps$-pseudospectrum of the matrix \eqref{eq:Ank}, where the matrix norm is that induced by the norm \eqref{eq:norm} on $X^n$. A similar comment holds for the other $\eps$-pseudospectra in \eqref{eq:sigdef} and \eqref{eq:pidef}.}, for $n\in \N$,
\begin{equation} \label{eq:sigdef}
\sigma^n_{\eps}(A)\ :=\ \bigcup_{k\in\Z} \speps A_{n,k}, \quad \eps>0,
\qquad\textrm{and}\qquad \Sigma^{n}_{\eps}(A)\ :=\ \overline{\bigcup_{k\in\Z}
\Speps A_{n,k}},\quad \eps\geq 0,
\end{equation}
the $\tau$-method inclusions sets are the right hand sides of the following inclusions (recall that $\Specn_0 A := \Spec A$ so that taking $\eps=0$ in the first of these inclusions provides an inclusion for $\Spec A$):
\begin{equation} 
\label{incl:met1}
\begin{array}{|c|} \hline \\[-2.5mm]
\Speps A \ \subset\
\Sigma^n_{\eps+\eps_n}(A), \quad \eps \geq 0, \quad\textrm{and}
\quad \speps A\ \subset\ \sigma^n_{\eps+\eps_n}(A), \quad \eps>0. \\[2mm]\hline
\end{array}
\end{equation}
These inclusions  (proved as Theorem \ref{method0}) hold for all $n\in \N$, the second inclusion with no constraint on the Banach space $X$, the first with the constraint that $X$ satisfies Globevnik's property (see \S\ref{sec:pseud}), which holds in particular if $X$ is finite-dimensional or a Hilbert space.  
The term $\eps_n$, defined in Theorem \ref{method0} (and see  Corollary \ref{minimum_weighted_norm_two} for the special case that $A$ is bidiagonal), enters \eqref{incl:met1} as
a ``truncation penalty'' when
passing from the infinite matrix $A$ to a finite
matrix of size $n\times n$.  (Similar comments apply to the terms $\eps'_n$ and $\eps''_n$ in the $\pi$ and $\tau_1$ methods below.)  Corollary \ref{CorolA} provides an upper bound for $\eps_n$  that implies that
\begin{equation} \label{eq:epsn}
0\ \le\ \eps_n \ \le \
2(\|\alpha\|_\infty+\|\gamma\|_\infty)\,\sin\frac{\pi}{2(n+2)}\ \le\
(\|\alpha\|_\infty+\|\gamma\|_\infty)\,\frac{\pi}{n+2}.
\end{equation}


\paragraph{The $\pi$ method} Similarly, defining, for $n\in \N$ and $t\in \T$,
\begin{equation} \label{eq:pidef}
\pi^{n,t}_{\eps}(A)\ :=\ \bigcup_{k\in\Z} \speps A^{\per,t}_{n,k}, \quad \eps>0,
\qquad\textrm{and}\qquad \Pi^{n,t}_{\eps}(A)\ :=\ \overline{\bigcup_{k\in\Z}
\Speps A^{\per,t}_{n,k}}, \quad \eps\geq 0,
\end{equation}
we prove in Corollary \ref{fn_circulant_corollary}, for $n\in \N$ and $t\in \T$, the $\pi$-method inclusions
\begin{equation} 
\label{incl:met1*}
\begin{array}{|c|} \hline \\[-2.5mm]
\Speps A \ \subset\ \Pi^{n,t}_{\eps+\eps'_n}(A), \quad \eps\geq 0, \quad\textrm{and}
\quad \speps A\ \subset\ \pi^{n,t}_{\eps+\eps'_n}(A), \quad \eps>0,\\[2mm]\hline
\end{array}
\end{equation}
where 
\begin{equation} \label{eq:epsnd}
\eps'_n \ := \
2(\|\alpha\|_\infty+\|\gamma\|_\infty)\,\sin\frac{\pi}{2n},
\end{equation}
and the second inclusion in \eqref{incl:met1*} holds whatever the Banach space $X$, while the first requires that $X$ has the Globevnik property. 

Our motivation for the $\pi$ method and the definitions \eqref{eq:pidef} come from consideration of the case when the diagonals of $A$ are constant, i.e., $A$ is a so-called {\em block-Laurent} matrix.  If $A$ is tridiagonal and block-Laurent and $X$ is a finite-dimensional Hilbert space, then $\Speps A = \cup_{t\in \T} \Pi^{n,t}_{\eps}(A)$, for $\eps\geq 0$, and $\speps A = \cup_{t\in \T} \pi^{n,t}_{\eps}(A)$, for $\eps> 0$ (see Theorem \ref{thm:periodic}).


\paragraph{The $\tau_1$ method} The $\tau_1$ method modifies these constructions using ideas from
\cite{Davies1998:Encl,DaviesPlum,HansenJFA08} (see the discussion in \S\ref{sec:ideas}). To see the similarities but distinction between the $\tau$ and $\tau_1$ methods, for $B\in L(E)$ and $n\in \N$ let
\begin{equation} \label{eq:mudag}
\mu^\dag_n(B)\  :=\ \inf_{k\in \Z} \, \mu(P_{n,k}B|_{E_{n,k}})  = \inf_{k\in \Z} \, \min\Big(\nu\left(P_{n,k}B|_{E_{n,k}}\right),\,\nu\left(P_{n,k}B^{*}|_{E_{n,k}}\right)\Big).
\end{equation}
Then, see Proposition \ref{prop:same},
\begin{equation} \label{eq:sigmaAAlt}
\sigma^n_{\eps}(A)\ =\  \{\lambda\in \C:\mu^\dag_n(A-\lambda I)<\eps\} \quad \mbox{and} \quad \Sigma^n_{\eps}(A)\ =\  \{\lambda\in \C:\mu^\dag_n(A-\lambda I)\leq \eps\},
\end{equation} 
for $\eps>0$, the first of these identities holding for every Banach space $X$, the second if $X$ has the Globevnik property.
The operators $P_{n,k}B|_{E_{n,k}}$ are two-sided truncations of $B$, each corresponding to an $n\times n$ matrix. By contrast, in the $\tau_1$ method we make one-sided truncations, dropping the $P_{n,k}$'s in \eqref{eq:mudag} and so replacing $\mu^\dag_n(B)$ by 
$$
\mu_n(B)\   := \inf_{k\in \Z} \min\Big(\nu\left(B|_{E_{n,k}}\right),\,\nu\left(B^{*}|_{E_{n,k}}\right)\Big).
$$
Precisely, for $n\in\N$ let
\begin{equation} \label{eq:gamdef}
\begin{aligned}
\gamma^n_{\eps}(A) &:=  \left\{\lambda\in \C : \mu_n(A-\lambda I)<\eps\right\}, \;\; \eps>0, \;\; \mbox{and} \\
\Gamma^n_{\eps}(A) &:=  \left\{\lambda\in \C : \mu_n(A-\lambda I)\leq \eps\right\}, \;\; \eps\geq 0.
\end{aligned}
\end{equation}
Then the $\tau_1$ method inclusions are
\begin{equation} 
\label{incl:met2}
\begin{array}{|cc|} \hline &\\[-2.5mm]
 \Gamma^n_{\eps}(A)\ \subset\ \Speps A\
\subset\ \Gamma^{n}_{\eps+\eps''_n}(A), \;\; \eps \geq 0,
& \textrm{and}\\ 
\gamma^n_{\eps}(A)\ \subset\ \speps A\
\subset\ \gamma^n_{\eps+\eps''_n}(A), \;\; \eps >0,&\\[2mm]\hline
\end{array}
\end{equation}
which hold for all $n\in \N$ (whatever the Banach space $X$), where
\begin{equation} \label{eq:epsn_method2}
\eps''_n := 2(\|\alpha\|_\infty+\|\gamma\|_\infty) \sin \frac{\pi}{2(n+1)}.
\end{equation}
In contrast to \eqref{incl:met1} and \eqref{incl:met1*}, \eqref{incl:met2} provides two-sided inclusions, which is significant for the convergence of the $\tau_1$ method inclusion sets in the limit $n\to\infty$ (see Theorem \ref{thm:converge}). We will establish the inclusions from the right in \S\ref{sec:tau1a} (see Corollary \ref{cor:one_sided_truncation}), but the inclusions from the left, i.e.\ that $\Gamma^n_{\eps}(A) \subset \Speps A$, for $\eps\geq 0$,  and  $\gamma^n_{\eps}(A) \subset \speps A$, for $\eps>0$, are immediate consequences of the definitions and the first inequality in \eqref{eq:lnProp}.

Like the $\tau$ and $\pi$ method inclusion sets  \eqref{eq:sigdef} and \eqref{eq:pidef}, also the $\tau_1$ method sets \eqref{eq:gamdef} can be expressed in terms of finite submatrices. For $n\in \N$ and $k\in \Z$, $(A-\lambda I)|_{E_{n,k}}$ corresponds to the $\infty\times n$ matrix consisting of columns $k+1$ to $k+n$ of $A$. Let $A^+_{n,k}$ denote the matrix consisting of the $n+2$ rows of this $\infty\times n$ matrix that are non-zero, by the tridiagonal structure of $A$, and let $I^+_n$ denote the corresponding submatrix of the identity operator, so that $A^+_{n,k}$ and $I^+_n$ are the $(n+2)\times n$ matrices
\begin{equation} \label{eq:An+}
\begin{aligned}
A^+_{n,k} &:= \begin{pmatrix}
\gamma_{k+1} & 0&\cdots&0& 0\\\hline
& & A_{n,k}\\\hline
0& 0 &\cdots&0 & \alpha_{k+n}
\end{pmatrix} \quad \mbox{and}\\
 I^+_n &:= \begin{pmatrix}
0&\cdots&0\\\hline
&I_n &\\\hline
0&\cdots&0
\end{pmatrix}, \quad\mbox{where} \quad
I_n := \begin{pmatrix}
I_X\\
&\ddots\\
&&I_X
\end{pmatrix}
\end{aligned}
\end{equation}
is an $n\times n$ identity matrix, and $I_X$ is the identity operator on $X$ (see Figure \ref{fig:matrices} for a visualisation of $A_{n,k}^+$). Then, for $n\in \N$, $k\in \Z$, and $\lambda\in \C$, $\nu((A-\lambda I)|_{E_n,k})=\nu\left(A^+_{n,k}-\lambda I^+_n\right)$ and $\nu((A-\lambda I)^*|_{E_n,k})=\nu\left((A^*)^+_{n,k}-\lambda I^+_n\right)$, so that
\begin{equation} \label{eq:munmat}
\mu_n(A-\lambda I) = \inf_{k\in \Z} \min\Big(\nu\left(A^+_{n,k}-\lambda I^+_n\right),\,\nu\left((A^*)^+_{n,k}-\lambda I^+_n\right)\Big).
\end{equation}

\subsubsection{Our convergence results for the banded case} \label{sec:conv}
The inclusions \eqref{incl:met1}, \eqref{incl:met1*}, \eqref{incl:met2}, for the $\tau$, $\pi$, and $\tau_1$ methods,  are main results of the paper.  A further key result, a  corollary of the two-sided inclusion \eqref{incl:met2}, is the convergence result, stated as Theorem \ref{thm:converge} below, that the $\tau_1$ method inclusion sets,  $\Gamma^n_{\eps+\eps''_n}(A)$ and $\gamma^n_{\eps+\eps''_n}(A)$, converge to the spectral sets that they include as $n\to\infty$.

The one-sided inclusions for the $\tau$ and $\pi$ methods, \eqref{incl:met1} and \eqref{incl:met1*}, do not imply convergence of the corresponding inclusion sets to the spectral sets that they include; indeed we will exhibit examples where this convergence is absent (see Example \ref{ex:shift}, \S\ref{sec:shift}, and \S\ref{sec:perturb_periodic}). But these methods are convergent if they do not suffer from spectral pollution for the particular $A$ in the sense of the following definition\footnote{The standard notion of spectral pollution, e.g.\ \cite{DaviesPlum}, is that a sequence of linear operators $(A_n)$ approximating $A$ {\em suffers from spectral pollution} if there exists a sequence $z_n\in \Spec A_n$ such that $z_n\to z$ and $z\not\in \Spec A$. Spectral pollution is said to be absent if this does not hold. One might, alternatively, require absence of spectral pollution to mean that also every subsequence of $(A_n)$ does not suffer from spectral pollution. Absence of spectral pollution in this stronger sense is equivalent to a requirement that, for some positive null sequence $(\eta_n)$, $\Spec A_n\subset \specn_{\eta_n} A$ for each $n$. Our notion of absence of spectral pollution is a version of this definition that is concerned with  pseudospectra not just spectra, and is uniform with respect to the parameter $k\in \Z$.} (cf.~\cite{DaviesPlum}). This is a strong assumption; indeed, to say that the $\tau$ method does not suffer from spectral pollution is equivalent to saying that, for some positive null sequence $\eta_n$, $\sigma_\eps^n(A)\subset  \specn_{\eps+\eta_n} A$, for $\eps>0$, and similarly for the $\pi$ method, which inclusion takes the place, for the $\tau/\pi$ methods, of the left hand inclusions in \eqref{incl:met2}, for the  $\tau_1$ method. Nevertheless, we will see below, in \S\ref{sec:bL} and \S\ref{sec:pseudoergodic}, important applications where this absence-of-spectral-pollution assumption is satisfied.

\begin{definition}[Absence of spectral pollution] \label{def:specpol} We say that the $\tau$ method (the $\pi$ method for a particular $t\in \T$) {\em does not suffer from spectral pollution} for a particular tridiagonal $A\in L(E)$ if there exists a positive null sequence $(\eta_n)_{n\in \N}$ such that, for every $\eps>0$, $\speps A_{n,k} \subset \specn_{\eps+\eta_n} A$ ($\speps A^{\per,t}_{n,k} \subset \specn_{\eps+ \eta_n} A$) for $n\in \N$ and $k\in \Z$.
\end{definition}

The following three theorems are our main convergence results for the $\tau_1$, $\tau$, and $\pi$ methods. These are theorems  for the case that $A\in L(E)$ is tridiagonal, but recall from \S\ref{sec:main} that every $A\in BO(E)$ can be written in this tridiagonal form. The proofs of these results, in which $\Hto$ denotes the standard Hausdorff convergence of sets introduced in \S\ref{sec:keynot}, are so short that we include them here.

\begin{theorem}[Convergence of the $\tau_1$ method]  \label{thm:converge} Suppose that the matrix representation of $A\in L(E)$ is tridiagonal. Then, as $n\to \infty$,
\begin{equation} \label{eq:Gconv}
\Gamma^n_{\eps+\eps''_n}(A)\Hto \Speps A, \;\; \mbox{for } \eps\geq 0, \quad \mbox{in particular } \quad \Gamma^n_{\eps''_n}(A)\Hto \Spec A,
\end{equation}
If $X$ has Globevnik's property (see \S\ref{sec:pseud}) then also 
\begin{equation} \label{eq:gconv}
\gamma^n_{\eps+\eps''_n}(A)\Hto \speps A, \;\; \mbox{for } \eps> 0.
\end{equation}
\end{theorem}
\begin{proof} It follows from \eqref{incl:met2} that, for $\eps\geq 0$,
\begin{equation} \label{eq:Gamconv}
\Speps A\
\subset\ \Gamma^n_{\eps+\eps''_n}(A)\ \subset\ \Specn_{\eps+\eps''_n} A.
\end{equation}
Since $\eps''_n\to 0$ as $n\to\infty$, \eqref{eq:Gconv} follows from \eqref{eq:HDconv} and the observation at the end of the Hausdorff convergence discussion in \S\ref{sec:keynot}. By \eqref{incl:met2}, for $\eps>0$ the inclusions \eqref{eq:Gamconv} hold also with $\Speps$ replaced by $\speps$ and $\Gamma^n_{\eps+\eps''_n}(A)$  replaced by $\gamma^n_{\eps+\eps''_n}(A)$. Further,  if  $X$ has Globevnik's property then  $\specn_{\eps+\eps''_n} A\Hto \speps A$ (see the comment below \eqref{eq:HDconv}), and \eqref{eq:gconv} follows.
\end{proof}

\begin{theorem}[Convergence of the $\tau$ method]  \label{thm:converge2} Suppose that the matrix representation of $A\in L(E)$ is tridiagonal, the $\tau$ method does not suffer from spectral pollution for $A$, and $X$ has Globevnik's property. Then, as $n\to \infty$,
\begin{equation} \label{eq:Sconv}
\Sigma^n_{\eps+\eps_n}(A)\Hto \Speps A, \;\; \mbox{for } \eps\geq 0, \quad \mbox{in particular } \quad \Sigma^n_{\eps_n}(A)\Hto \Spec A,
\end{equation}
and
\begin{equation} \label{eq:sconv}
\sigma^n_{\eps+\eps_n}(A)\Hto \speps A, \;\; \mbox{for } \eps> 0.
\end{equation}
\end{theorem}
\begin{proof} It follows from \eqref{incl:met1} and that the $\tau$ method does not suffer from spectral pollution for $A$ that, for some positive null sequence $(\eta_n)$,
\begin{equation} \label{eq:Sigconv}
\Speps A\
\subset\ \Sigma^n_{\eps+\eps_n}(A)\ \subset\ \sigma^n_{\eps+1/n+\eps_n} A \subset \specn_{\eps+1/n+\eps_n+\eta_n} A\subset \Specn_{\eps+1/n+\eps_n+\eta_n} A,
\end{equation}
for $\eps\geq 0$, so that  \eqref{eq:Sconv} follows from \eqref{eq:HDconv}.  Similarly, for $\eps>0$,
$$
\speps A\
\subset\ \sigma^n_{\eps+\eps_n}(A)\ \subset\ \specn_{\eps+\eps_n+\eta_n} A.
$$
so \eqref{eq:sconv} follows from \eqref{eq:HDconv}.
\end{proof}
Arguing similarly, but using \eqref{incl:met1*} instead of \eqref{incl:met1}, we have an analogous result for the $\pi$ method.
\begin{theorem}[Convergence of the $\pi$ method]  \label{thm:converge3} Suppose that the matrix representation of $A\in L(E)$ is tridiagonal, that the  $\pi$ method does not suffer from spectral pollution for $A$ for some $t\in \T$, and that $X$ has Globevnik's property. Then, as $n\to \infty$,
\begin{equation} \label{eq:Pconv}
\Pi^{n,t}_{\eps+\eps'_n}(A)\Hto \Speps A, \;\; \mbox{for } \eps\geq 0, \quad \mbox{in particular } \quad \Pi^{n,t}_{\eps_n}(A)\Hto \Spec A,
\end{equation}
and
\begin{equation} \label{eq:pconv}
\pi^{n,t}_{\eps+\eps_n}(A)\Hto \speps A, \;\; \mbox{for } \eps> 0.
\end{equation}
\end{theorem}

\begin{remark}[Estimation of $\eps_n$, $\eps'_n$, and $\eps''_n$] \label{rem:over} It is clear, inspecting the proofs, that Theorems \ref{thm:converge}-\ref{thm:converge3} remain valid if $\eps_n$, $\eps'_n$, and $\eps''_n$ are replaced in the statements of the theorems by upper bounds $\eta_n$, $\eta'_n$, and $\eta''_n$, as long as these upper bounds tend to zero as $n\to\infty$. In particular, Theorem \ref{thm:converge2} remains true with $\eps_n$, as defined in Theorem \ref{method0}, replaced by the explicit upper bounds for $\eps_n$ obtained in Corollaries \ref{CorolA} and \ref{CorolB}.
\end{remark}

\subsubsection{The band-dominated case} \label{sec:bdo}
The above results for banded operators can be extended to the band-dominated case by perturbation arguments, precisely by using \eqref{eq:PseudInc} below, a simple perturbation inclusion for pseudospectra. Perhaps unexpectedly, this leads not just to inclusion sets but also to convergent sequences of approximations, even for the spectrum.

To see how this works, suppose that $A\in BDO(E)$, and let us focus for brevity on the $\tau_1$ method and the case that $X$ is finite-dimensional or a Hilbert space (we prove a closely related results for the general Banach space case in \S\ref{sec:thmgen} as Theorem \ref{thm:fin2}). Then, by definition, there exists a sequence $(A^{(n)})_{n\in \N}\subset BO(E)$ such that $\delta_n:= \|A-A^{(n)}\|\to 0$ as $n\to \infty$. Let  $w_n$ be the bandwidth of $A^{(n)}$ and let $A_n:= \mathcal{I}_{w_n}A^{(n)}\mathcal{I}^{-1}_{w_n}$, where $\mathcal{I}_b$ is defined, for $b\in \N$, as below \eqref{eq:norm}. Then, as discussed in \S\ref{sec:main}, the matrix representation of $A_n$ is tridiagonal. We will apply the inclusion \eqref{incl:met2} to $A_n$. Let us write $\eps''_n$ in that inclusion as $\eps''_n(A)$ to indicate explicitly its dependence on $A$, so that
\begin{equation} \label{eq:epsn''2}
\eps''_n(A) = 2r(A) \sin \frac{\pi}{2(n+1)},
\end{equation}
where $r(A)$ is defined by \eqref{eq:rA}.
Then, for $\eps\geq 0$ and $n\in\N$, using \eqref{incl:met2} and \eqref{eq:PseudInc2}, we see that
\begin{eqnarray*}
\Speps A &\subset &\Specn_{\eps+\delta_n} A^{(n)} = \Specn_{\eps+\delta_n} A_n\\
& \subset & \Gamma^n_{\eps+\delta_n+\eps''_n(A_n)}(A_n) \subset \Specn_{\eps+\delta_n+\eps''_n(A_n)}(A_n) \\
& = & \Specn_{\eps+\delta_n+\eps''_n(A_n)}(A^{(n)}) \subset \Specn_{\eps+2\delta_n+\eps''_n(A_n)}(A).
\end{eqnarray*}
Now $\eps''_n(A_n)\to 0$ as $n\to\infty$, since, by \eqref{eq:normbound}, $r(A_n) \leq 2\|A_n\|=2\|A^{(n)}\|\leq 2(\|A\|+\delta_n)$. Thus, by \eqref{eq:HDconv}, $\Specn_{\eps+2\delta_n+\eps''_n(A_n)}(A)\Hto\Speps A$ as $n\to\infty$, which gives immediately the following result.
\begin{theorem} \label{thm:bdo} Suppose that $A\in BDO(E)$ and $X$ is finite-dimensional or a Hilbert space, and  let $(A^{(n)})_{n\in \N}\subset BO(E)$ be such that $\delta_n:= \|A-A^{(n)}\|\to 0$ as $n\to\infty$, and let $A_n := \mathcal{I}_{w_n}A^{(n)}\mathcal{I}^{-1}_{w_n}$, where $w_n$ is the band-width of $A^{(n)}$, so that $A_n$ is $A^{(n)}$ written in tridiagonal form. Then, for $\eps\geq 0$ and $n\in \N$,
$$
\Speps A \subset  \Gamma^n_{\eps+\delta_n+\eps''_n(A_n)}(A_n) \subset  \Specn_{\eps+2\delta_n+\eps''_n(A_n)}(A),
$$
so that $\Gamma^n_{\eps+\delta_n+\eps''_n(A_n)}(A_n)\Hto \Speps A$ as $n\to\infty$, in particular   $\Gamma^n_{\delta_n+\eps''_n(A_n)}(A_n)$ $\Hto$ $\Spec A$.
\end{theorem}

\begin{remark}[Construction of banded approximations and estimation of $\delta_n$] It is easy to see that the statements in the above theorem remain true (cf.\ Remark \ref{rem:over}) if $\delta_n$ is replaced throughout by any upper bound $\eta_n\geq \delta_n:=\|A-A^{(n)}\|$, as long as $\eta_n\to 0$ as $n\to\infty$; in particular, $\Speps A \subset \Gamma^n_{\eps+\eta_n+\eps''_n(A_n)}(A_n)$ and $\Gamma^n_{\eps+\eta_n+\eps''_n(A_n)}(A_n)\Hto \Speps A$ as $n\to\infty$. But to compute the inclusion set $\Gamma^n_{\eps+\eta_n+\eps''_n(A_n)}(A_n)$ one needs both a concrete banded approximation $A^{(n)}$ to $A$ and the upper bound $\eta_n\geq \delta_n$. Concretely, one can take, for some $\mathfrak{p} \geq 0$, $A^{(n)}\in BO(E)$ to be the matrix with band-width $n$ and entries $[a^{(n)}_{i,j}]$ given by
\begin{equation} \label{eq:Andef}
a^{(n)}_{i,j} := \left\{\begin{array}{ll} \left(1-\frac{|i-j|}{n+1}\right)^\mathfrak{p} a_{i,j}, & |i-j| \leq n,\\ 0, & |i-j|>n,\end{array}
\right.
\end{equation}
for $n\in \N$. Provided $\mathfrak{p}>0$, it holds that $\delta_n=\|A-A^{(n)}\|\to 0$ as $n\to\infty$ (see the proof that (e)$\Rightarrow$(a) in \cite[Theorem 2.1.6]{RaRoSiBook}).
For $\mathfrak{p}=0$, $A^{(n)}$ is the obvious approximation to $A$ with band-width $n$ obtained by simply discarding the matrix entries $A_{ij}$ with $|i-j|>n$. In many cases this approximation is adequate, but it does not hold in this case that $\|A-A^{(n)}\|\to 0$ as $n\to\infty$ for all $A\in BDO(E)$ (see the example in \cite[Remark 1.40]{LiBook}).

One case where the simple choice $\mathfrak{p}=0$ is effective is where $A\in \mathcal{W}(E)$, the so-called {\em Wiener algebra} (e.g., \cite[Definition 1.43]{LiBook}). In terms of the matrix representation $[a_{i,j}]$, $A\in \mathcal{W}(A)$ precisely when there exists $\kappa=(\kappa_i)_{i\in \Z}\in \ell^1(\Z,\R)$ such that $\|a_{i,j}\|_{L(X)}\leq \kappa_{i-j}$, for $i,j\in \Z$. In that case it is clear that (see \cite[Equation (1.27)]{LiBook}), defining $A^{(n)}$ by \eqref{eq:Andef} with $\mathfrak{p}=0$,
$$
\|A-A^{(n)}\| \leq \eta_n := \sum_{|k|>n}\kappa_k, \quad n\in \N,
$$
with $\eta_n\to 0$ as $n\to\infty$.
\end{remark}

\begin{remark}[Proving invertibility of operators] \label{rem:invert} Let us flag one significant application of Theorem \ref{thm:bdo} and our other convergence theorems, Theorems \ref{thm:converge}-\ref{thm:converge3}, when coupled with our inclusion results, \eqref{incl:met1}, \eqref{incl:met1*}, and \eqref{incl:met2}. Suppose that $A\in BDO(E)$ and that $(U_n)_{n\in \N}\subset \C^\C$ is a sequence of compact sets with the properties that: (a) $\Spec A \subset U_n$, for $n\in \N$; (b) $U_n\Hto \Spec A$ as $n\to\infty$. Then it is easy to see that the following claim holds for any $\lambda\in \C$:
\begin{equation} \label{eq:invert}
\lambda I - A \mbox{ is invertible } \quad \Leftrightarrow \quad \lambda \not\in U_n, \mbox{ for some } n\in \N.
\end{equation}
In particular, $A$ is invertible if and only if $0\not\in U_n$, for some $n\in \N$.

Of course, by \eqref{incl:met2} and Theorem \ref{thm:bdo}, if $X$ is finite-dimensional or a Hilbert space, a sequence $(U_n)_{n\in \N}$ with these properties, for any $A\in BDO(E)$, is the sequence $U_n = \Gamma^n_{\eps''_n}(A)$, $n\in \N$. We will see below in Theorem \ref{prop:finite2} that, in the case that $X=\C$, another sequence with properties (a) and (b) is the sequence $U_n=\widehat \Gamma^n_{\mathrm{fin}}(A)$, $n\in \N$. The  attraction of this alternative choice for $(U_n)_{n\in \N}$ is that one can determine whether $\lambda \in U_n$ in only finitely many arithmetic operations (see Theorem \ref{prop:finite2}). Thus, if $\lambda I-A$ is invertible, this can be proved in finite time, by checking whether $\lambda\in U_n$, successively for $n=1,2,\ldots$, stopping when a $U_n$ is found such that $\lambda\not\in U_n$. (The existence of such a $U_n$ is guaranteed by \eqref{eq:invert}.)

At the end of \S\ref{sec:psFinite} we use \eqref{eq:invert}, with $U_n=\Sigma^n_{\eps_n}(A)$, for an operator $A$, the so-called Feinberg-Zee operator, for which this choice of $U_n$ satisfies (a) and (b), to show that a particular $\lambda\not\in \Spec A$, proving that $\Spec A$ is a strict subset of the closure of its numerical range.
\end{remark}

\subsubsection{Computability of our inclusion sets: reduction to finitely many matrices} \label{sec:comput}
A  criticism of the inclusion sets that we have introduced in \S\ref{sec:incl} from the perspective of practical computation is that to determine whether or not a particular $\lambda\in \C$ is in one of these inclusion sets, one has to make a computation for infinitely many finite matrices, indexed by $k\in \Z$. 

 In certain cases this infinite collection of finite matrices is in fact a finite collection. For the inclusion sets with parameter $n\geq 2$, this is the case precisely when the tridiagonal matrix $A$ has only finitely many distinct entries. (If $n=1$ it is enough, for the $\tau$ method, if the set $\{\beta_k:k\in \Z\}$ is finite.) 

When $A$ has infinitely many distinct entries, provided the set of matrix entries is a relatively compact subset of $L(X)$, which is the case, in particular, if $X$ is finite-dimensional, we can approximate the infinite collection arbitrarily closely by a finite subset. We spell out the details in the following theorem for the $\tau_1$ method under the assumption that $X$ is finite-dimensional or a Hilbert space; similar results hold for the $\tau$ and $\pi$ methods. The point of this theorem is that, for every $\eps\geq 0$, it can be established whether or not $\lambda\in\Gamma^{n,\mathrm{fin}}_{\eps}(A)$  by determining whether or not the lower norms of finitely many finite matrices are $\leq \eps$. 
Note that the definition of $\Gamma^{n,\mathrm{fin}}_{\eps}(A)$ in \eqref{eq:findef}  is identical to the definition \eqref{eq:gamdef} of $\Gamma^n_\eps(A)$, except that $\mu^n$, defined by taking an infimum over $\Z$, is replaced by $\mu_n^{\mathrm{fin}}$, defined by taking a minimum over the finite set $K_n$. Thus, for $\eps\geq 0$,
\begin{equation} \label{eq:finGG}
\Gamma^{n,\mathrm{fin}}_{\eps}(A) \subset \Gamma^n_\eps(A), \quad \mbox{indeed also} \quad \Gamma^n_\eps(A) \subset \Gamma^{n,\mathrm{fin}}_{\eps+1/n}(A), 
\end{equation}
as a consequence of \eqref{eq:lnProp}, if \eqref{eq:comp} holds.
This result, indeed an extended version of this result that applies in the general band-dominated case and for every Banach space $X$, is proved in \S\ref{sec:thmgen} as Theorem \ref{thm:fin2}. A version of this theorem for the $\tau$ and $\pi$ methods, under the assumption of absence of spectral pollution (Definition \ref{def:specpol}), is proved as Theorem \ref{thm:fin3}.

\begin{theorem} \label{thm:fin}
Suppose that $X$ is finite-dimensional or a Hilbert space, that $A\in L(E)$ is tridiagonal, and that $\{A_{ij}:i,j\in \Z\}\subset L(X)$ is relatively compact. Then, for every $n\in \N$ there exists a finite set $K_n\subset \Z$ such that:
\begin{equation} \label{eq:comp}
\forall k\in \Z \;\; \exists j\in K_n \mbox{ such that }\|A^+_{n,k}-A^+_{n,j}\|\leq 1/n \mbox{ and } \|(A^*)^+_{n,k}-(A^*)^+_{n,j}\|\leq 1/n.
\end{equation}
Further, for $\eps\geq 0$ and $n\in \N$,
$$
\Gamma^{n,\mathrm{fin}}_{\eps}(A)\ \subset\ \Speps A \subset \Gamma^{n,\mathrm{fin}}_{\eps+\eps''_n + 1/n}(A),
$$
where
\begin{equation} \label{eq:findef}
\begin{aligned}
\Gamma^{n,\mathrm{fin}}_{\eps}(A) & :=  \left\{\lambda\in \C : \mu_n^{\mathrm{fin}}(A-\lambda I)\leq \eps\right\} \quad \mbox{and}\\
\mu_n^{\mathrm{fin}}(B) &:= \min_{k\in K_n} \min\Big(\nu\left(B|_{E_{n,k}}\right),\,\nu\left(B^{*}|_{E_{n,k}}\right)\Big),
\end{aligned}
\end{equation}
for $B\in L(E)$, so that
\begin{equation} \label{eq:mufin}
\mu_n^{\mathrm{fin}}(A-\lambda I)= \min_{k\in K_n} \min\Big(\nu\left(A^+_{n,k}-\lambda I^+_n\right),\,\nu\left((A^*)^+_{n,k}-\lambda I^+_n\right)\Big), \quad \mbox{for }\; \lambda\in \C.
\end{equation}
Further, for $\eps\geq 0$, $\Gamma^{n,\mathrm{fin}}_{\eps+\eps''_n + 1/n}(A)\Hto \Speps A$ as $n\to\infty$, in particular $\Gamma^{n,\mathrm{fin}}_{\eps''_n + 1/n}(A)\Hto$ $\Spec A$.
\end{theorem}

Of course, a practical issue for computation is how to determine 
the finite set $K_n$. As a concrete, nontrivial example, we demonstrate how to reduce from an infinite set to a finite set for the class of tridiagonal operators that  are pseudoergodic in the sense of Davies \cite{Davies2001:PseudoErg} in \S\ref{sec:psInfinite}.

\subsubsection{An algorithm for computing the spectrum in the case $X=\C^p$ and its solvability complexity index} \label{sec:SCI}

In the finite-dimensional case that $X=\C^p$, for some $p\in \N$, equipped with the usual Euclidean norm, a sequence of approximations converging to $\Spec A$, in the case when $A$ is band-dominated, can be realised with each member of the sequence computed in finitely many operations, provided we have available as inputs to the computation sufficient information about $A$.
We will demonstrate this for the general band-dominated case in \S\ref{sec:scalar}. For the tridiagonal case, with $A$ given by \eqref{eq:A} and $X=\C^p$ (recall that we noted  in \S\ref{sec:main} that  every $A\in BO(E)$ with $X=\C$ can be written in this form), the $\tau_1$ version of the algorithm proceeds as follows. We use in this definition the notation (cf.~\eqref{eq:grignew})
\begin{equation} \label{eq:grid}
\mathrm{Grid}(n,r) := \frac{1}{n}(\Z+ \ri \Z)\cap r\D, \quad n\in \N, \;\; r>0.
\end{equation}

Let $\Omega^p_T$ denote the set of all tridiagonal matrices \eqref{eq:A}, with $\alpha,\beta,\gamma \in \ell^\infty(\Z,X)$ and $X=\C^p$, for some $p\in \N$. Where, as usual, $P(S)$ denotes the power set of a set $S$, the inputs we need are:
\begin{enumerate}
\item A mapping $\mathcal{A}_p:\Omega^p_T\times \N\to P(\Z)$, $(A,n)\mapsto K_n$, where $K_n\subset \Z$ is a finite set such that \eqref{eq:comp} holds.
\item A mapping $\B_p:\Omega^p_T\to \R^3$, $A\mapsto (\alpha_{\max},\beta_{\max},\gamma_{\max})$, such that $\alpha_{\max} \geq \|\alpha\|_\infty$, $\beta_{\max} \geq \|\beta\|_\infty$, $\gamma_{\max} \geq \|\gamma\|_\infty$.
\item A mapping $\cC_p:\Omega^p_T\times \Z \times \N\to (X^{(n+2)\times n})^2$, $(A,n,k)\mapsto (A^+_{n,k},(A^*)^+_{n,k})$.
\end{enumerate}
It is not clear to us how to construct a map $\A_p$ such that $\A_p(A,n)$ can be explicitly computed for general $A\in \Omega^p_T$. Indeed, establishing existence of a map $\A_p:\Omega^p_T\times \N\to P(\Z)$  with the required properties may require an application of the axiom of choice. But, as a non-trivial example, we will construct   in \S\ref{sec:psInfinite} a version of the mapping $\A_p$ for the important subset $\Omega_\Psi\subset \Omega^1_T$ of operators that are pseudoergodic in the sense of Davies \cite{Davies2001:PseudoErg}.

The sequence of approximations to $\Spec A$, each element of which can be computed in finitely many arithmetical operations, given finitely many evaluations of the above input maps, is the sequence $(\Gamma^n_{\mathrm{fin}}(A))_{n\in \N}$, defined using the notation \eqref{eq:findef} by
\begin{equation} \label{eq:finite}
\Gamma^n_{\mathrm{fin}}(A) := \Gamma^{n,\mathrm{fin}}_{\eps^*_n+ 3/n}(A) \cap \mathrm{Grid}(n,R), \qquad n\in \N,
\end{equation}
with (cf.~\eqref{eq:epsn_method2})
\begin{equation} \label{eq:eps*}
\eps_n^*\ :=\ (\alpha_{\max}+\gamma_{\max})\,\frac{22}{7(n+1)}\ \geq\ 2(\alpha_{\max}+\gamma_{\max})\,\sin\frac{\pi}{2(n+1)}\geq \eps''_n,
\end{equation}
$(\alpha_{\max},\beta_{\max},\gamma_{\max}) := \B_p(A)$, $R:= \alpha_{\max}+\beta_{\max}+\gamma_{\max}$ (so that $R$ is an upper bound for $\|A\|$), and $K_n$, in the definition \eqref{eq:findef} and \eqref{eq:mufin}, given by $K_n:= \A_p(A,n)$.

\begin{proposition} \label{prop:finite}
For $A\in \Omega^p_T$ and $n\in \N$, $\Gamma^n_{\mathrm{fin}}(A)$ can be computed in finitely many arithmetic operations, given finitely many evaluations of the functions $\A_p$, $\B_p$, and $\cC_p$, namely the evaluations:  $K_n:= \A_p(A,n)$; $(A_{n,k}^+,(A^*)^+_{n,k}):= \cC_p(A,k,n)$, for $k\in K_n$; and $(\alpha_{\max},\beta_{\max},\gamma_{\max})$ $:=$ $\B_p(A)$.  Further, $\Gamma^n_{\mathrm{fin}}(A)\Hto \Spec A$ as $n\to \infty$, and also
$$
\widehat \Gamma^n_{\mathrm{fin}}(A)\ :=\ \Gamma^n_{\mathrm{fin}}(A) + \frac{2}{n}\overline{\D}\ \Hto\ \Spec A
$$
as $n\to \infty$, with $\Spec A \subset \widehat \Gamma^n_{\mathrm{fin}}(A)$ for each $n\in \N$.
\end{proposition}
We give the proof of the above result, and a version for $A\in BDO(E)$, in \S\ref{sec:scalar}.

The above result can be interpreted as a result relating to the solvability complexity index (SCI) of \cite{SCI,ColHan2023}. Let us abbreviate $\Omega_T^1$, $\A_1$, $\B_1$, and $\cC_1$ as $\Omega_T$, $\A$, $\B$, and $\cC$, respectively, so that $\Omega_T$ is the set of tridiagonal matrices with complex number entries. With the {\em evaluation set} (in the sense of  \cite{SCI,ColHan2023})\footnote{Strictly, to fit with the definition in \cite[\S2.1]{ColHan2023}, each element of the evaluation set $\Lambda$ should be a complex-valued function on $\Omega_T$. But it is easy to express our evaluation set in this form, expressing each of our functions in terms of a finite number of complex-valued functions. In particular, where $M:= |K_n|$, the set of integers $K_n=\{k_1,\ldots,k_M\}$, with $k_1<k_2<\ldots<k_M$, which is the output of the mapping $\A(\cdot,n)$ applied to $A$, might be encoded as the complex number $x+\ri y$, with $x:= k_1$ and $y:=\Pi_{i=1}^{M} p_i^{k_i-k_1}$, where $p_i$ is the $i$th prime number.}
\begin{equation} \label{eq:Lambda}
\Lambda:= \{\A(\cdot,n),\B,\cC(\cdot,k,n):k\in \Z, \, n\in \N\},
\end{equation}
and where we equip $\C^C$, the set of compact subsets of $\C$, with the Hausdorff metric (see \S\ref{sec:keynot}),
the mappings 
$$
\Omega_T\to \C^C, \qquad A\mapsto \Gamma^n_{\mathrm{fin}}(A) \quad \mbox{and} \quad A\mapsto \widehat \Gamma^n_{\mathrm{fin}}(A),
$$
are {\em general algorithms} in the sense of \cite{SCI,ColHan2023}, for each $n\in \N$. Further, $\widehat \Gamma^n_{\mathrm{fin}}(A)$ can be computed in finitely many arithmetic operations and specified using finitely many complex numbers (the elements of $\Gamma^n_{\mathrm{fin}}(A)$ and the value of $n$). Thus, where $\Xi:\Omega_T\to \C^C$ is the mapping given by $\Xi(A) := \Spec A$, for $A\in \Omega_T$, the {\em computational problem} $\{\Xi,\Omega_T, \C^C,\Lambda\}$ has  arithmetic SCI, in the sense of \cite{SCI,ColHan2023}, equal to one; more precisely, since also $\Spec A \subset \widehat \Gamma^n_{\mathrm{fin}}(A)$, for each $n\in \N$ and $A\in \Omega_T$, this computational problem is in the class $\Pi_1^A$, as defined in \cite{SCI,ColHan2023}. 

The same observations on SCI classification hold true for the corresponding computational problem with $\Omega_T$ replaced by $\Omega:= BDO(E)$, that we consider in \S\ref{sec:scalar}; again this has arithmetic SCI equal to one. In contrast, as we recall in \S\ref{sec:other}, the computational problem of determining the spectrum of a band-dominated operator, indeed even the restricted problem of determining the spectrum of a tridiagonal matrix,  has been shown in \cite{Matt}, adapting the arguments of \cite{SCI}, to have an SCI of two when the evaluation set, the set of allowed inputs to the computation, is restricted to the mappings $\Omega_T\to \C$, $A\mapsto a_{i,j}$, for $i,j\in \Z$, providing evaluation of the matrix elements.

\subsection{A first example demonstrating the sharpness of our inclusion sets} \label{sec:exam}
We will explore the utility and properties of these new families of inclusion sets in a range of examples in \S\ref{sec:pseudop}. But, to aid comprehension of the $\tau$, $\pi$, and $\tau_1$ inclusion sets  introduced in \S\ref{sec:incl}, and demonstrate their sharpness, let us pause to illustrate them as applied to one of the simplest scalar examples, an example with $X=\C$ so that the matrix entries are just complex numbers. We will return to this example in \S\ref{sec:shift}. 
\begin{example}[The shift operator]\label{ex:shift}
Let $V$ denote the shift operator on $\ell^2(\Z)$, taking $x=(x_j)_{j\in\Z}\in\ell^2(\Z)$ to $y=(y_j)_{j\in\Z}\in\ell^2(\Z)$
with $x_j=y_{j+1}$ for all $j\in\Z$. Its matrix is of the form \eqref{eq:A} with
$\alpha_j=1$, $\beta_j=0$ and $\gamma_j=0$ for all $j\in\Z$.
For $A:=V$, the matrices \eqref{eq:Ank} and \eqref{eq:Ankper} are given by
\begin{equation} \label{eq:Vn}
A_{n,k}=V_{n}\ :=\ \left(\begin{array}{cccc}0&\\
1&0&\\
&\ddots&\ddots&\\
&&1&0
\end{array}\right)_{n\times n}
\ \textrm{and}\quad
A^{\per,t}_{n,k}=V^{\per,t}_{n}\ :=\ \left(\begin{array}{cccc}0&&&t\\
1&0&\\
&\ddots&\ddots&\\
&&1&0
\end{array}\right)_{n\times n},
\end{equation}
for all $n\in\N$, independent of $k\in\Z$.
The circulant matrix $V_n^{\per,1}$ is normal and has spectrum $\spec V^{\per, 1}_n=\T_n=\{z\in\C:z^n=1\}$ (e.g., \cite[Thm.~7.1]{TrefEmbBook}). For $t\in \T$, where $\alpha=t^{1/n}$ denotes any one of the $n$th roots of $t$, and $D_n:= \diag(\alpha^1,\ldots,\alpha^n)$, it holds that $D_n^{-1}(V^{\per,t}_n-\lambda I_n)D_n = \alpha (V^{\per,1}_n-\lambda \bar \alpha I_n)$, so that $V^{\per,t}_n$ is also normal and
\begin{equation} \label{eq:spec_Vpert}
\Spec V^{\per,t}_n=t^{1/n}\Spec V^{\per,1}_n=t^{1/n}\T_n.
\end{equation}
(Note that the set $t^{1/n}\T_n$ is independent of which $n$th root of $t$ we select.)
 Thus, for each $t\in \T$, the spectrum of $V$, $\Spec
V=\T=\{z\in\C:|z|=1\}$, is well approximated as $n\to\infty$ by
$\Spec V^{\per, t}_n$, but clearly not by $\Spec
V_n=\{0\}$. 

\begin{figure}[t]
\begin{center}
\begin{tabular}{|c|cccc|}
\hline
& $n=4$ & $n=8$ & $n=16$ & $n=32$\\
\hline
&&&&\\[-1em]
$\stackrel{\text{$\tau$ method}}{\rule{0pt}{10mm}}$ &
\includegraphics[width=22mm]{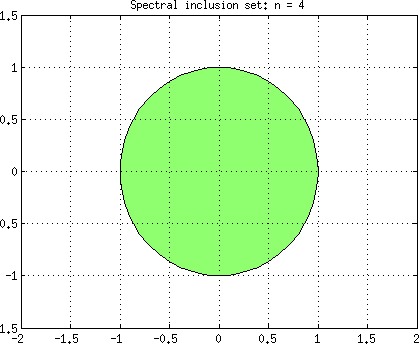} &
\includegraphics[width=22mm]{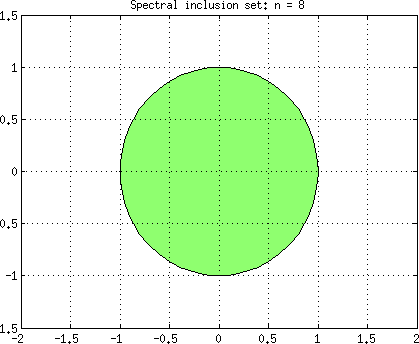} &
\includegraphics[width=22mm]{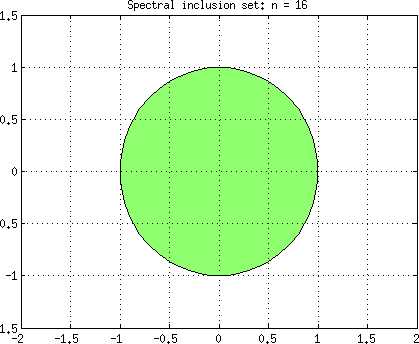} &
\includegraphics[width=22mm]{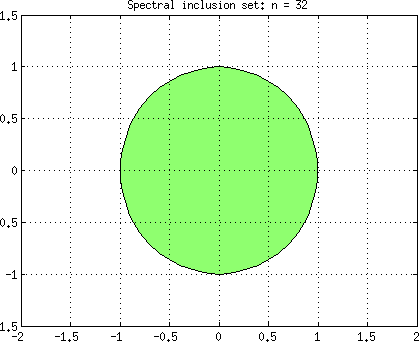}\\
$\stackrel{\begin{tabular}{c}\scriptsize \text{$\pi$ method,}\\\scriptsize\text{with $t=1$}\end{tabular}}{\rule{0mm}{6mm}}$ &
\includegraphics[width=22mm]{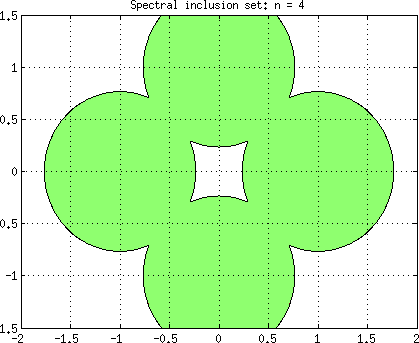} &
\includegraphics[width=22mm]{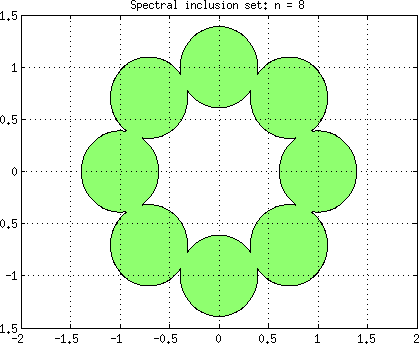} &
\includegraphics[width=22mm]{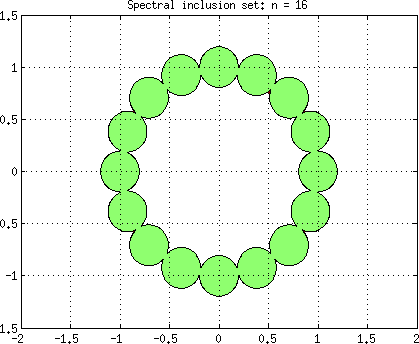} &
\includegraphics[width=22mm]{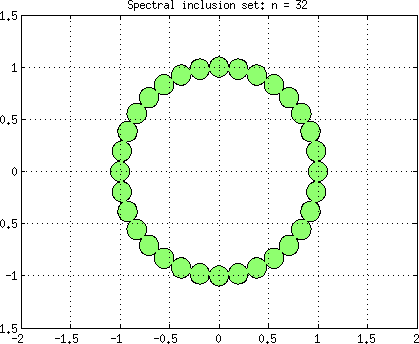}\\
$\stackrel{\text{$\tau_1$ method}}{\rule{0pt}{10mm}}$ &
\includegraphics[width=22mm]{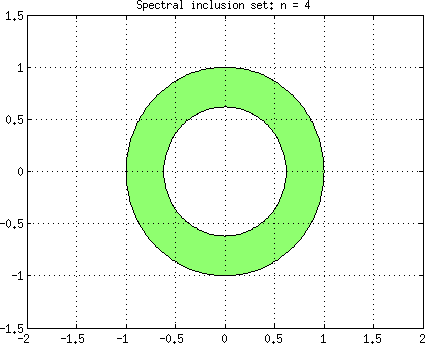} &
\includegraphics[width=22mm]{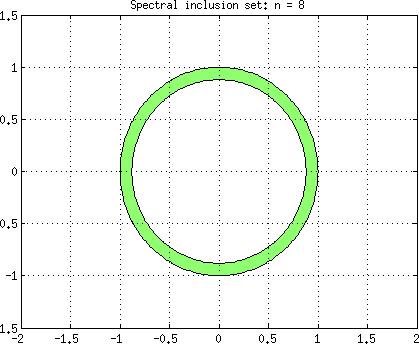} &
\includegraphics[width=22mm]{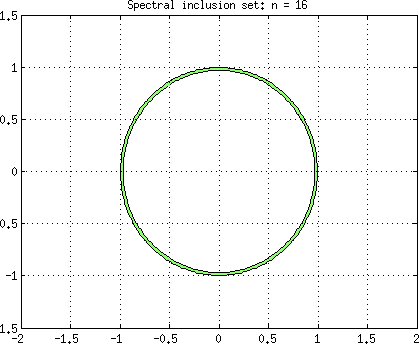} &
\includegraphics[width=22mm]{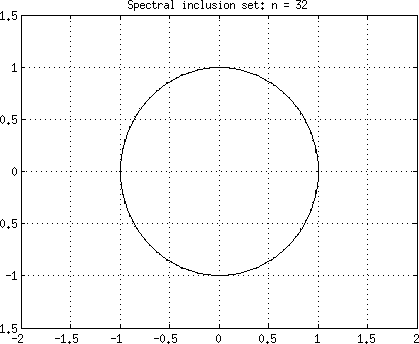}\\
\hline
\end{tabular}
\end{center}
\caption{The $\tau$, $\pi$, and $\tau_1$ inclusion sets, $\Sigma^n_{\eps_n}(V)$, $\Pi^{n,1}_{\eps'_n}(V)$, and $\Gamma^n_{\eps''_n}(V)$, respectively, for $n=4,\,8,\,16$, and $32$,
each an inclusion set for $\Spec V=\T$.} \label{fig:shift}
\end{figure}

The $\tau$ method and $\pi$ method inclusion sets for $\Spec V = \Specn_0 V$, given by \eqref{incl:met1} and
\eqref{incl:met1*}, respectively, reduce in this case to the sets
$$
\Sigma^n_{\eps_n}(V) = \Specn_{\eps_n}(V_n)
$$
and
$$
\Pi^{n, t}_{\eps'_n}(V) = \Specn_{\eps'_n}(V^{\per, t}_n), \quad t\in \T.
$$
Since $\|\alpha\|_\infty = 1$ and $\|\gamma\|_\infty=0$ for this example, we have by Corollary \ref{minimum_weighted_norm_two} and \eqref{eq:epsnd} that
$$
\eps_n = 2\sin\left(\frac{\pi}{4n+2}\right) \quad \mbox{and} \quad \eps'_n = 2\sin\left(\frac{\pi}{2n}\right).
$$
$\Specn_{\eps_n}(V_n)$ and $\Specn_{\eps'_n}(V^{\per, t}_n)$ are neighbourhoods of $\Spec V_n=\{0\}$ and $\Spec
V^{\per_n,t}=t^{1/n}\T_n$, respectively. By  \eqref{incl:met1} and
\eqref{incl:met1*} these neighbourhoods must be large enough to cover
$\Spec V=\T$.  As we will see in \S\ref{sec:shift}, for each $n\in \N$, $\Sigma^n_{\eps_n}(V)$ is precisely the closed unit disc, while (we recall this and other properties of pseudospectra below in \S\ref{sec:pseud}), since $V^{\per, t}_n$ is normal,  $\Specn_{\eps'_n}(V^{\per, t}_n)=\Pi^{n,t}_{\eps'_n}(V)$ is the closed $\eps'_n$ neighbourhood of $t^{1/n}\T_n$. Further, as we will see in \S\ref{sec:shift}, $\Gamma_{\eps''_n}^n(V)$, the $\tau_1$ method inclusion set given by \eqref{incl:met1*}, is the closed annulus with outer radius 1, inner radius $1-(\eps''_n)^2$, where
$$
\eps''_n = 2\sin\left(\frac{\pi}{2(n+1)}\right).
$$

Of course any intersection of these inclusion sets is also an inclusion set. In particular, see \S\ref{sec:shift},
$$
\bigcap_{t\in \T} \Pi^{n, t}_{\eps'_n}(V) = \bigcap_{t\in \T} \left(t^{1/n} \T_n + \eps'_n\, \overline{\D}\right)
$$
is also a closed annulus, with outer radius 1, inner radius $1-(\eps'_n)^2$.
 The $\tau$, $\pi$, and $\tau_1$ inclusion sets for $\spec V = \T$ 
are shown for a range of values of $n$, selecting $t=1$ for the $\pi$ method, in Figure \ref{fig:shift}.
\end{example}

This simple example illustrates the sharpness of the inclusions \eqref{incl:met1}, \eqref{incl:met1*}, and \eqref{incl:met2}. As we will see in \S\ref{sec:shift}, the inclusion sets, which are the closed
disc $\overline{\D}$ for the $\tau$ method, the union $t^{1/n}+\eps_n\overline{\D}$ of closed discs around shifted roots of unity  for the $\pi$ method,
and a closed annulus for the  $\tau_1$ method, do not cover $\spec V=\T$
if $\eps_n$, $\eps'_n$, and $\eps''_n$ are any smaller than the values given above. 

This example also illustrates Theorem \ref{thm:converge} on convergence of the $\tau_1$ method, that $\Gamma^n_{\eps''_n}(V)\Hto \spec V$ as $n\to\infty$.  Since also, for each $t\in \T$, $\Pi^{n, t}_{\eps_n}(V)\Hto\spec V$ as $n\to\infty$, this example also illustrates that the $\pi$ method can be convergent. (Generalising this example, we will show in \S\ref{sec:pseudoergodic}, by application of Theorem \ref{thm:converge3}, that the $\pi$ method is convergent for all tridiagonal $A\in L(E)$ with scalar entries that are pseudoergodic in the sense of Davies \cite{Davies2001:PseudoErg}, this class including all tridiagonal Laurent matrices, where the diagonals $\alpha$, $\beta$, and $\gamma$ are constant.) Of course, this is also an example illustrating that the $\tau$ method need not be convergent, an example where the $\tau$ method suffers from spectral pollution (see Definition \ref{def:specpol}), so that the conditions of Theorem \ref{thm:converge2} are not satisfied.

\subsection{Related work} \label{sec:other}
Let us give more detail about previous work related to this paper, on inclusions sets and approximation algorithms for the spectrum and pseudospectrum of bounded linear operators. See also the recent, overlapping review in \cite{ColHan2023}.

Given a Banach space $Y$, and $A\in L(Y)$, a trivial but important inclusion set for $\Spec A$ is the ball $\|A\|\overline{\D}=\{z\in \C:|z|\leq \|A\|\}$. In the case that $Y$ is a Hilbert space, a well-known, sharper inclusion set is $\overline{\mathrm{Num}\, A}$, the closure of $\mathrm{Num}\, A$, the numerical range of $A$. Important generalisations of these respective inclusion sets are the higher order hull and (in the case that $Y$ is  Hilbert space) the higher order numerical range, denoted $\mathrm{Hull}_n A $ and $\mathrm{Num}_n A$, respectively, and defined by (e.g., \cite[\S9.4]{Davies2007:Book})
\[
\begin{aligned}
\mathrm{Hull}_n A &:= \bigcap_{\mathrm{deg}(p)\leq n} \left\{z:|p(z)|\leq \|p(A)\|\right\} \quad \mbox{and}\\ 
\mathrm{Num}_n A&:= \bigcap_{\mathrm{deg}(p)\leq n} \left\{z:p(z)\in \overline{\mathrm{Num}(p(A))}\right\}, 
\end{aligned}
\]
for $n\in \N$. 
Note that $\mathrm{Num}_1 A=\overline{\mathrm{Num}\,A}$ and, rather surprisingly (see \cite[Thm.~9.4.5]{Davies2007:Book} and \cite{BurkeGreen06}), $\mathrm{Hull}_n A= \mathrm{Num}_n A$, for each $n$, in the case that $Y$ is a Hilbert space. Clearly, the higher order hulls and numerical ranges form decreasing sequences, i.e.\ $\mathrm{Hull}_{n+1} A\subset \mathrm{Hull}_n A$ and $\mathrm{Num}_{n+1} A\subset \mathrm{Num}_n A$, for $n\in \N$, so that
\[
\mathrm{Hull}_n A\Hto \mathrm{Hull}_\infty A := \bigcap_{n\in \N} \mathrm{Hull}_n A \quad \mbox{and} \quad \mathrm{Num}_n A\Hto \mathrm{Num}_\infty A := \bigcap_{n\in \N} \mathrm{Num}_n A,
\]
as $n\to\infty$. (Of course,  $\mathrm{Num}_\infty A$ is defined only in the Hilbert space case, in which case $\mathrm{Hull}_\infty A=\mathrm{Num}_\infty A$, since $\mathrm{Hull}_n A= \mathrm{Num}_n A$ for each $n$.)

These inclusion sets may be sharp or asymptotically sharp as $n\to\infty$.  In particular \cite[Lem.~9.4.4]{Davies2007:Book}, $\mathrm{Hull}_2(A)=\mathrm{Num}_2(A)=\Specn(A)$ in the Hilbert space if $A$ is self-adjoint. Further \cite[Thm.~2.10.3]{NevBook} $\mathrm{Hull}_\infty A=\widehat{\Spec}A$, the complement of the unbounded component of $\C\setminus \Spec A$, so that $\mathrm{Hull}_n A\Hto \Spec A$ as $n\to\infty$ if and only if $\C\setminus \Spec A$ has only one component. But, while these sequences of inclusion sets are of significant theoretical interest and converge in many cases to $\Spec A$, it is unclear, for general operators $A$ and particularly for larger $n$, how to realise these sets computationally. (However, see the related work of Frommer et al.\ \cite{Frommer} on computation of inclusion sets for pseudospectra via approximate computation of numerical ranges of resolvents of the operator, and work of B\"ogli \& Marletta \cite{BoMar} on inclusion sets for spectra via numerical ranges of related operator pencils.)

For finite matrices, as we have noted already in \S\ref{sec:ideas}, a standard inclusion  set for the spectrum (the set of eigenvalues) is provided by Gershgorin's theorem \cite{Gershgorin,Varga}. This result has been generalised, independently in Ostrowski \cite{Ostrowski}, Feingold \& Varga \cite{Feingold}, and Fiedler \& Pt\'ak \cite{Fiedler} (and see \cite[Chapter 6]{Varga}) to cases where the entries of the $N\times N$ matrix $A=[A_{ij}]$ are themselves submatrices\footnote{These papers are arguably the earliest occurrence of pseudospectra in the literature (cf.~\cite[{\S}I.6]{TrefEmbBook}). The Gershgorin-type  enclosure for the spectrum obtained in each of these papers (e.g., \cite[Equation (3.1)]{Feingold}), which is essentially the right-hand-side of \eqref{eq:salas}, is thought of in \cite{Feingold} as the  union of $N$ so-called Gershgorin sets $G_j$, $j=1,\ldots, N$. The $j$th set $G_j$ (though this language is not used), is precisely a (closed) pseudospectrum of $a_{j,j}$, the $j$th matrix on the diagonal.
}. 
A further extension to the case where each entry $A_{ij}$ is an operator between Banach spaces has been made by Salas \cite{Salas}. To make clear the connection with this paper, consider the case where each entry $A_{ij}\in L(X)$, for some Banach space $X$. The main result of \cite{Salas} in that case, expressed in the language of pseudospectra, is that
\begin{equation} \label{eq:salas}
\Spec A \subset \bigcup_{i=1}^N \Specn_{r_i} a_{i,i}, \qquad \mbox{where} \quad r_i := \sum_{j=1, \,j\neq i}^N \|a_{i,j}\|_{L(X)},
\end{equation}
which is close to \eqref{eq:incl} in the case $n=1$, when $\eps_n=r(A)$ and $r(A)$ is an upper bound for $\sup_{i\in \Z} r_i$. Of course, \eqref{eq:salas} is for finite rather than infinite matrices. Several authors have generalised the Gershgorin theorem to infinite matrices with scalar entries, but the focus has been on cases where the infinite matrix $A$ is an unbounded operator with a discrete spectrum; see \cite{Shiv,Farid1,Farid1} and the references therein. A simple generalisation of the Gershgorin theorem to the case where $A$ is a tridiagonal matrix of the form \eqref{eq:A}, with scalar entries, that is a bounded operator on $\ell^2(\Z)$ is proved as \cite[Theorem 2.50]{HengThesis}. This has \eqref{eq:incl} for $n=1$ as an immediate corollary.

Probably the most natural idea to approximate $\speps A$, where $A\in BO(E)$ or $BDO(E)$, with $E=\ell^2(\N)$ or $\ell^2(\Z)$, is to use the pseudospectra, $\speps A_n$, of the finite sections of $A$, i.e., $A_n:=[a_{i,j}]_{i,j=1}^n$ or $A_n:=[a_{i,j}]_{i,j=-n}^n$, respectively.
In some cases, including Toeplitz operators \cite{ReichelTref,BoeSi2} and semi-infinite random Jacobi operators
\cite{CWLi2016:Coburn}, $\speps A_n\Hto\speps A$ as $n\to\infty$. (In some instances where the finite section method is not effective, the pseudospectra of periodised finite sections (cf.~our $\pi$ method) can converge to $\speps A$; see \cite{Colb:PE_pBC} and Remark \ref{rem:Col}.) But in general, as noted already in \S\ref{sec:ideas}, the sequence $\speps A_n$ does not converge to $\speps A$; its cluster points typically contain $\speps A$ but also contain points outside $\Spec A$ (see, e.g., \cite{SeSi:FSBDO}).

Even in the self-adjoint case one observes this effect of so-called spectral pollution; see, e.g., \cite{DaviesPlum}. As we noted in \S\ref{sec:ideas}, Davies \& Plum \cite{DaviesPlum}, building on Davies \cite{Davies1998:Encl},  proposed, as a method to locate eigenvalues of self-adjoint operators in spectral gaps while avoiding spectral pollution, to compute $\nu(P_m(A-\lambda I)|_{E_n})$, with $m>n$, as an approximation to $\nu(A-\lambda I)$ (which coincides with $\mu(A-\lambda I)$, defined by \eqref{eq:invNorm}, in this self-adjoint setting). Here $(P_k)_{k\in \N}$ is any sequence of orthogonal projections onto finite-dimensional subspaces that is strongly convergent to the identity, and $E_n:=P_n(E)$ is the range of $P_n$ The significance of this proposal is that, in this self-adjoint setting: (a) $\Spec A=\{\lambda\in \R:\nu(A-\lambda I)=0\}$ (see \eqref{eq:invNorm}); (b) $\nu(P_m(A-\lambda I)|_{E_n})\to \nu(P_m(A-\lambda I)|_{E_n})$ as $m\to\infty$;  (c)  $\nu((A-\lambda I)|_{E_n})\to \nu(A-\lambda I)$ as $n\to\infty$, uniformly for $\lambda\in \R$; indeed \cite[Lemma 5]{DaviesPlum}, 
\begin{equation} \label{eq:Davies}
\nu(A-\lambda I)|_{E_n} \leq \nu(A-\lambda I) \leq \nu((A-\lambda I)|_{E_n}) + \eta_n, 
\end{equation}
for $\lambda \in \R$, where $\eta_n\to 0$ as $n\to\infty$. To relate this  to our results above, suppose that $E=\ell^2(\Z,X)$ and $P_n:= P_{2n+1,-n-1}$, for $n\in \N$, where $P_{n,k}$ is defined by \eqref{eq:Pdef}. Then, if $A\in L(E)$ is tridiagonal, and using the notations of \S\ref{sec:ideas},
$$
\nu(P_m(A-\lambda I)|_{E_n})=\nu((A-\lambda I)|_{E_n}) = \nu((A-\lambda I)^+_n),
$$
for $m\geq n+1$, so that the second inequality of \eqref{eq:Davies} bears a resemblance to Proposition \ref{prop:gam_main} (except that, importantly, \eqref{eq:Davies} can give no information on the dependence of $\eta_n$ on $A$). 

These ideas are moved to the non-self-adjoint case in   Hansen \cite{HansenJFA08,Hansen:nPseudo} (see also \cite{HeinPotLind08,SCI,ColbRomanHansen,ColHan2023}), where an  approximation $\cS^{n,m}_\eps(A)$, for $n,m\in \N$, $m\geq n$, to $\Speps A$ is proposed. This approximation, written in terms of lower norms using \eqref{eq:lowernorm1}, to match the work of Davies \& Plum \cite{DaviesPlum}, is given by
\begin{equation} \label{eq:Han}
\cS_\eps^{n,m}(A) := \left\{\lambda\in \C: \min\left(\nu\left(P_m(A-\lambda I)|_{E_n}\right),\nu\left(P_m(A-\lambda I)^*|_{E_n}\right)\right)\leq \eps \right\},
\end{equation}
where $E_n=P_n(E)$, $E$ is some separable Hilbert space, and $(P_n)_{n\in \N}$ is a sequence of projection operators converging strongly to the identity. Using \eqref{eq:lowernorm1}, in the case we consider where $E=\ell^2(\Z,X)$, and if we define $P_n:= P_{2n+1,-n-1}$, for $n\in \N$, and if $A\in L(E)$ is tridiagonal, then, if $m\geq n+1$, $\cS_\eps^{n,m}(A)$ is the specific approximation \eqref{eq:specfs}, which we have already contrasted with the new approximations  in this paper in \S\ref{sec:ideas}. Key features of the approximation \eqref{eq:Han} from the perspectives of the current work are: (i) that one can determine whether $\lambda \in \cS_\eps^{n,m}(A)$ in finitely many arithmetic operations (via the characterisations \eqref{eqln3} and \eqref{eq:ispd} in terms of positive definiteness and singular values of rectangular matrices, see \cite{Hansen:nPseudo}); and that: (ii) if $A$ band-dominated and $m=2n$,  then $\cS_\eps^{n,m}(A)\Hto \speps A$ as $n\to\infty$ (see \cite{SCI}).

Hansen \cite{HansenJFA08,Hansen:nPseudo}  extends \eqref{eq:Han}  to provide, in the Hilbert space context,  a sequence of approximations, based on computations with finite rectangular matrices, also to the so-called $(N,\eps)$-pseudospectra, $\specn_{N,\eps} A$, defined in \cite[Definition 1.2]{Hansen:nPseudo}.  These generalise the pseudospectrum (note that $\specn_{0,\eps}A=\speps A$), and are another sequence of inclusions sets: one has that $\Spec A + \eps \D\subset \specn_{N,\eps} A$, for $\eps >0$ and $N\in \N_0:=\N\cup\{0\}$; further, for every $\delta >\eps$, $\specn_{N,\eps} \subset \Spec A + \delta \D$, for all sufficiently large $N$. These approximation ideas are extended to the general Banach space case in Seidel \cite{Seidel:Neps}.

The Solvability Complexity Index (SCI), that we discuss in \S\ref{sec:SCI}, \S\ref{sec:scalar}, and \S\ref{sec:psInfinite}, is introduced in a specific context in Hansen  \cite{Hansen:nPseudo}, and in the general form that we use it in this paper in \cite{SCIshort,SCI}. Building on the work in \cite{Hansen:nPseudo}, 
Ben-Artzi el al.\ \cite{SCI} consider the computation of $\Specn_\eps A$, and hence, via \eqref{eq:speceps}, $\Spec A$, in the case when $A\in L(E)$ and $E$ is a separable Hilbert space. Via an orthonormal basis $(e_i)_{i\in \N}$ for $E$, we may identify $E$ with $\ell^2(\N)$ and $A$ with its matrix representation  $A=[a_{i,j}]_{i,j\in \N}$, where $a_{i,j} := (Ae_j,e_i)$.
 On the assumption that only the values $a_{i,j}$ are available as data, Ben-Artzi el al.\ \cite{SCI}, building on \cite{Hansen:nPseudo}, discuss the existence or otherwise of algorithms for computing $\Specn_\eps A$. In particular,  the authors construct in \cite{SCI}, for each $\eps>0$, the sequence of mappings $(\Gamma_{\eps,n})_{n\in \N}$, from $L(E)$ to $\C^C$, given by 
\begin{equation} \label{eq:grignew}
\Gamma_{\eps,n}(A) := \cS_\eps^{n,2n}(A) \cap \mathrm{Grid}(n),  \quad \mbox{where} \quad \mathrm{Grid}(n) := \frac{1}{n}(\Z+ \ri \Z).
\end{equation}
It is clear, from the above discussion,  that $\Gamma_{\eps,n}(A)$ can be computed for each $n$ from finitely many of the entries $a_{i,j}$ in finitely many operations, and it is not difficult, as a consequence of the properties of $\cS_\eps^{n,2n}(A)$ noted above, to see (cf.~Proposition \ref{prop:finite}) that
 $\Gamma_{\eps,n} (A)\Hto \Specn_\eps A$ as $n\to\infty$ whenever the matrix $A$ is band-dominated. Thus, where $(\eps_m)_{m\in \N}$ is any positive null sequence, $\Spec A$ can be computed as the iterated limit
\begin{equation} \label{eq:double}
\Spec A  = \lim_{m\to\infty} \Specn_{\eps_m} A= \lim_{m\to\infty} \lim_{n\to\infty} \Gamma_{\eps_m,n} (A).
\end{equation}

In the same paper \cite{SCI} (and see \cite{Matt}) they show moreover that this double limit is optimal; it cannot be replaced by a single limit. Precisely, they show that there exists no sequence of algorithms $(\gamma_{n})_{n\in \N}$, from $L(E)$ to $\C^C$, such that $\gamma_n(A)\Hto \Specn(A)$ for every band-dominated matrix $A$ and $\gamma_n(A)$ can be computed, for each $n$, with only finitely many of the matrix entries $a_{i,j}$ as input. 
In the language introduced in \cite{SCI}, the SCI of computing the spectrum of a band-dominated operator from its matrix entries is two.

\section{Pseudospectra and Globevnik's property} \label{sec:pseud}

We have introduced in \S\ref{sec:keynot} the definitions of the open and closed $\eps$-pseudospectra of $B\in L(Y)$, for a Banach space $Y$. A natural question is: what is the relationship between these two definitions? Clearly, for every  $\eps>0$, $\speps B\subset \Speps B$, so that also $\overline{\speps B} \subset \Speps B$ (since $\Speps B$ is closed). It is known  \cite{Globevnik74,Globevnik76,Shargo08} that  equality holds, i.e., $\Speps B = \overline{\speps B}$, if $Y$ has {\sl Globevnik's property}, that is: $Y$ is finite-dimensional or $Y$ or its dual $Y^*$ is a complex uniformly convex Banach space (see \cite[Definition 2.4 (ii)]{Shargo08}). Hilbert spaces have Globevnik's property but also $E=\ell^2(\Z,X)$ has Globevnik's property if $X$ has it \cite{Day}.  On the other hand, Banach spaces $Y$ and operators $B\in L(Y)$ are known \cite{Shargo08,Shargo09,ShargoShkarin09} such that $\overline{\speps B} \subsetneq\Speps B$.

Let us note a few further properties of pseudospectra that we will utilise. For a Banach space $Y$, $B\in L(Y)$, and $\eps>0$, let
\begin{equation} \label{eq:Speps0}
\speps^0 (B) := \{\lambda\in \C: \mbox{ there exists } x\in Y \mbox{ with }\|(A-\lambda I)x\| < \eps\|x\|\},
\end{equation}
and let $\Speps^0B$ denote the right hand side of the above equation with the $<$ replaced by $\leq$.
Then (this is \cite[p.31]{TrefEmbBook} plus \eqref{eq:invNorm}), we have the following useful characterisations of $\speps B$:
\begin{equation} \label{eq:spepspert}
\speps B\ =\ \bigcup_{\|T\|<\eps} \Spec (B+T) = \Spec B \cup \speps^0 B = \speps^0(B) \cup \speps^0(B^*), \quad \eps>0.
\end{equation}
For $\Speps B$ we have  (see \cite{Shargo09} and \eqref{eq:invNorm}) the weaker statement that 
\begin{equation} \label{eq:spepspert2}
\Speps B\ \supset\ \bigcup_{\|T\|\leq\eps} \Spec (B+T) \supset \Spec B \cup \Speps^0 B = \Speps^0(B) \cup \Speps^0(B^*), \quad \eps>0.
\end{equation}
It follows from \eqref{eq:spepspert} and \eqref{eq:spepspert2} that, for $\eps>0$,
$\Spec B+\eps\D\subset\speps B$ and $\Spec B + \eps \overline{\D}\subset \Speps B$. We note that equality holds in these inclusions  (see, e.g., \cite[p.~247]{Davies2007:Book}) if $Y$ is a Hilbert space and $B$ is normal. 
Further, it follows from  the first equality in \eqref{eq:spepspert} that, for $B,T\in L(Y)$ and $\eps, \delta>0$,
\begin{equation} \label{eq:PseudInc}
\speps(B+T) \subset \specn_{\eps+\delta} B \quad \mbox{if} \quad \|T\|\leq \delta.
\end{equation}

If $Y$ is finite dimensional or is a Hilbert space the first $\supset$ in \eqref{eq:spepspert2} can be replaced by $=$ (see \cite[Theorem 3.27]{HaRoSi2} and \cite{Shargo09}), so that we have also that
\begin{equation} \label{eq:PseudInc2}
\Speps(B+T) \subset \Specn_{\eps+\delta} B \quad \mbox{if} \quad \|T\|\leq \delta.
\end{equation}
 But the first $\supset$ in \eqref{eq:spepspert2} cannot be replaced by $=$ for every Banach space $Y$ and every $B\in L(Y)$, even if $Y$ has Globevnik's property, as shown by examples of Shargorodsky \cite{Shargo09}. The second $\supset$  in \eqref{eq:spepspert2} can be replaced by equality if $Y$ is finite dimensional, but not, in general, otherwise. As an example relevant to this paper let $A\in L(E)$ have the matrix representation \eqref{eq:A} with $\alpha=\gamma=0$ and $\beta_k = \tanh(k)$, $k\in \Z$. Then $E= \ell^2(\Z)$ is a Hilbert space so that the first $\supset$ in \eqref{eq:spepspert2}  is equality. However, $\Spec A = \overline{\{\beta_k:k\in\Z\}}$, so that $\pm 1\in \Spec A\subset [-1,1]$, which implies that $\lambda =1+\eps\in \Speps A\setminus \Spec A$. But this $\lambda \not\in \Speps^0 A$, for if $x\in E$ with $x\neq 0$ then
\begin{equation} \label{eq:low}
\|(A-\lambda I)x\|^2 = \sum_{k\in \Z} (1+\eps-\beta_k)^2|x_k|^2 > \eps^2\|x\|,
\end{equation}
since $\beta_k<1$ for all $k\in \Z$.

We have noted in \S\ref{sec:keynot} the simple identity \eqref{eq:speceps}, and the same identity with $\Speps$ replaced by $\speps$. Similarly, immediately from the definitions, for every Banach space $Y$ and $B\in L(Y)$, generalising \eqref{eq:speceps}, we have that
\begin{equation} \label{eq:oc}
\Speps B = \bigcap_{\eps'>\eps} \Specn_{\eps'}B = \bigcap_{\eps'>\eps} \specn_{\eps'}B, \qquad \eps\geq 0.
\end{equation}
Recalling the results on Hausdorff convergence of \S\ref{sec:keynot}, this implies, for every positive null sequence $(\eta_n)$ and every $\eps\geq 0$, that
\begin{equation} \label{eq:HDconv}
\Specn_{\eps+\eta_n} B\Hto \Speps B \qquad \mbox{and} \qquad \specn_{\eps+\eta_n} B\Hto \Speps B.
\end{equation}
Recalling the discussion of Hausdorff convergence in \S\ref{sec:keynot}, if $Y$ has Globevnik's property, so that $\overline{\speps B} = \Speps B$, then we can, if we prefer, write \eqref{eq:HDconv} with $\Speps B$ replaced by $\speps B$.

\section{The $\tau$ method: principal submatrices} \label{sec:tau}

In this section we prove the inclusion \eqref{incl:met1} and compute the optimal value of the
``penalty'' term $\eps_n$. The proof of
\eqref{incl:met1} is so central to this paper that we want to sketch
its main idea, captured in Proposition \ref{prop:tau_main} below, beforehand.

Recalling \eqref{eq:spepspert}, suppose that $\lambda\in\speps A\setminus \Spec A= \speps^0 A$ and $x=(x_j)_{j\in\Z}\in E$ is a
corresponding pseudomode, i.e.
\begin{equation} \label{eq:pseudomode}
\|(A-\lambda I)x\|\ <\ \eps\,\|x\|.
\end{equation}
Consider all finite subvectors of $x$, of some fixed length $n\in\N$, namely
\begin{equation} \label{eq:subvector}
x_{n,k}\ :=\ (x_{k+1},...,x_{k+n})\in X^n,\qquad k\in\Z,
\end{equation}
 and consider the corresponding principal
submatrices $A_{n,k}$ of $A$ introduced in \eqref{eq:Ank}. We claim
that, for at least one $k\in\Z$, it holds, for some $\eps_n>0$ to be specified, that
\begin{equation} \label{eq:pseudomode_nk}
\|(A_{n,k}-\lambda I_n)x_{n,k}\|\ <\ (\eps+\eps_n)\,\|x_{n,k}\|
\end{equation}
and hence
$\lambda\in\specn_{\eps+\eps_n} A_{n,k}$.

One way to  show this is to suppose that, conversely, for all
$k\in\Z$, the opposite of \eqref{eq:pseudomode_nk} holds, and then
to take all these (opposite) inequalities and add their squares up
over $k\in\Z$, which, after some computation, contradicts
\eqref{eq:pseudomode}. Although \eqref{eq:pseudomode_nk}
need not hold for every $k\in\Z$, it does hold ``in the quadratic
mean'' over all $k\in\Z$, and therefore for at least one $k\in\Z$.

This computation reveals that \eqref{eq:pseudomode_nk} holds with $\eps_n=O(n^{-1/2})$ as $n\to\infty$; see \eqref{eq:weights_equal}. However, one can do better: instead of the
``sharp cutoff'' \eqref{eq:subvector}, we can  introduce weights
$w_1,...,w_n>0$ and put
\[
x_{n,k}\ :=\ (w_1x_{k+1},...,w_nx_{k+n})\in X^n,\qquad k\in\Z.
\]
With this modification \eqref{eq:pseudomode_nk} holds for some $k\in\Z$
but now with $\eps_n$ dependent on the choice of $w_1,...,w_n$. By
varying the weight vector $w=(w_1,...,w_n)$ we can minimise
$\eps_n$; the minimal $\eps_n= O(n^{-1})$ as $n\to\infty$.


We structure these arguments as follows. In \S\ref{sec:weight} we prove, as Theorem \ref{weighted_norm_two_refined}, a version of \eqref{incl:met1} with a formula for $\eps_n$ dependent on the choice of $w_1,\ldots,w_n$. The main idea of the proof, outlined above, is captured in
Proposition \ref{prop:tau_main}. In \S\ref{sec:mini} we minimise $\eps_n$ over the choice of the weights $w_1,\ldots, w_n$; the full proof of \eqref{incl:met1}, establishing the minimal formula for $\eps_n$, is given as Theorem \ref{method0}. This formula for $\eps_n$, while amenable to numerical computation, is complex, requiring minimisation of a function of one variable on a finite interval, this function defined implicitly through the solution of a nonlinear equation. Corollary \ref{minimum_weighted_norm_two} states an explicit formula for $\eps_n$ in the case that $A$ is bidiagonal, and Corollaries \ref{CorolA} and \ref{CorolB} provide more explicit upper bounds for $\eps_n$ in the general tridiagonal case. These are of value as discussed in Remark \ref{rem:over}.

Before continuing with the above plan, let us note properties of the sets $\sigma^n_\eps(A)$ and $\Sigma^n_\eps(A)$, defined by \eqref{eq:sigdef}, that appear in \eqref{incl:met1}. Clearly, for all $n\in\N$ (and all tridiagonal $A\in L(E)$), $\sigma^n_\eps(A)$ is open, for $\eps>0$, and $\Sigma^n_\eps(A)$ closed, for $\eps\geq 0$. We have further that
\begin{equation} \label{eq:sigSig}
\overline{\sigma^n_\eps(A)} \subset \Sigma^n_\eps(A), \;\; \eps>0.
\end{equation}
The following proposition is a characterisation of  $\sigma^n_\eps(A)$ and $\Sigma^n_\eps(A)$ that is key to the proof of \eqref{incl:met1}. This uses the notations
\begin{equation} \label{eq:SigW}
\begin{aligned}
\widehat \sigma^n_\eps(A)\ &:=\ \{\lambda\in \C: \mu_n^\dag(A-\lambda I)< \eps\},\;\; \eps>0, \quad \mbox{and}\\ 
\widehat \Sigma^n_\eps(A)\ &:=\ \{\lambda\in \C: \mu_n^\dag(A-\lambda I)\leq \eps\},\;\; \eps\geq 0,
\end{aligned}
\end{equation}
for $n\in \N$, where $\mu_n^\dag$ is defined by \eqref{eq:mudag}. Note that \eqref{eq:lnPro2} and \eqref{eq:mudag} imply that
$$
|\mu_n^\dag(A)-\mu_n^\dag(B)|\leq \|A-B\|,
$$
so that $\widehat \sigma^n_\eps(A)$ is open and $\widehat \Sigma^n_\eps(A)$ is closed. Note also that 
\begin{equation} \label{eq:wideS}
\widehat \Sigma^n_\eps(A)\ =\ \bigcap_{\eps'>\eps}\widehat \sigma^n_{\eps'}(A),  \qquad \eps\geq 0.
\end{equation}

\begin{proposition} \label{prop:same}
For every tridiagonal $A\in L(E)$ and $n\in \N$, it holds that
$\sigma^n_\eps(A) = \widehat \sigma^n_\eps(A)$, for $\eps>0$. If $X$ has the Globevnik property, then also $\overline{\sigma^n_\eps(A)} =\Sigma^n_\eps(A) = \widehat \Sigma^n_\eps(A)$, for $\eps> 0$.
\end{proposition}
\begin{proof} That $\sigma^n_\eps(A)\ =\ \widehat \sigma^n_\eps(A)$, for $\eps>0$, follows easily from \eqref{eq:spepsnu}, recalling that, for every $\eps\in \R$, the infimum of a set of real numbers is $<\eps$ if and only if one of the numbers in the set  is $<\eps$.

Suppose that $\lambda \in \Sigma^n_\eps(A)$. Then there exists a sequence $(\lambda_j)\subset \C$ and a sequence $(k_j)\subset \Z$ such that $\lambda_j\to \lambda$ and $\lambda_j\in \Speps(A_{n,k_j})$, so that $\mu(A_{n,k_j}-\lambda_j I_n)\leq \eps$. By \eqref{eq:lnPro2} it follows that 
$$
\mu^\dag_n(A-\lambda I)\ =\ \inf_{k\in \Z} \mu(A_{n,k}-\lambda I_n)\  \leq\ \mu(A_{n,k_j}-\lambda I_n) \leq \eps + |\lambda-\lambda_j|,
$$
for each $j$, so that $\mu^\dag_n(A-\lambda I)\leq \eps$ and  $\lambda \in  \widehat \Sigma^n_\eps(A)$. Thus $\Sigma^n_\eps(A)\subset \widehat \Sigma^n_\eps(A)$.

Define $D\in L(E)$ by 
$D:=\Diag\{A_{n,k}:k\in\Z\}$, and note that, for all $\lambda\in \C$, 
\begin{equation} \label{eq:muD}
\mu(D-\lambda I)\ =\ \inf_{k\in \Z} \mu(A_{n,k}-\lambda I_n)\ =\ \mu^\dag_n(A-\lambda I),
\end{equation}
so that, by \eqref{eq:spepsnu}, $\widehat \sigma^n_\eps(A)=\speps D$ and $\widehat \Sigma^n_\eps(A)=\Speps D$. Noting also  \eqref{eq:sigSig}, we have shown that $\overline{\speps D} = \overline{\sigma^n_\eps(A)}\subset \Sigma^n_\eps(A)\subset \widehat \Sigma^n_\eps(A)=\Speps D$.
But if $X$ has the Globevnik property, then  $\overline{\speps D} = \Speps D$, completing the proof.
\end{proof}

\subsection{Proof of the inclusion \eqref{incl:met1}} \label{sec:weight}
We start our proof of \eqref{incl:met1} with a very simple lemma
that proves helpful throughout.

\begin{lemma}\label{Simon_theta_inequality}
For $a,b\in\R$ and $\theta > 0$ we have
$(a+b)^{2}\ \le\ a^{2}(1+\theta)\,+\,b^{2}(1+\theta^{-1})$,
with equality if and only if  $a\theta =b$.
\end{lemma}
\begin{proof} Clearly,
$(a+b)^2\, \le\, (a+b)^2+(a\theta^{\frac 12}-b\theta^{-\frac
12})^2\, =\, a^{2}(1+\theta)+b^{2}(1+\theta^{-1})$
with equality if and only if $a\theta^{\frac 12}-b\theta^{-\frac 12}=0$, i.e.
$a\theta =b$.
\end{proof}

The following results lead to Theorem \ref{weighted_norm_two_refined} that provides justification of \eqref{incl:met1}, initially with a formula for $\eps_n$ that depends on general weights $w_1,...,w_n$.

\begin{proposition}\label{prop:tau_main}
Let $\eps>0$ and $n\in\N$, suppose that $w_j\in\R$, for $j=1,...,n$, with
at least one $w_j$ non-zero, and that $A\in L(E)$ is tridiagonal and $\|Ax\|\leq \eps$, for some $x\in E$ with $\|x\|=1$.  Then, for some $k\in \Z$ it holds that
\begin{equation} \label{eq:tau_main}
\nu(A_{n,k}) \leq \eps+\eps_n,
\end{equation}
with
\begin{equation} \label{eq:epsn1}
\eps_n\ :=\
\left\|\alpha\right\|_{\infty}\sqrt{\frac{T_{n}^{-}}{S_{n}}}\ +\
\left\|\gamma\right\|_{\infty}\sqrt{\frac{T_{n}^{+}}{S_{n}}}\,,
\end{equation}
where
\begin{equation*} \label{Sndef}
S_{n}\ :=\ \sum_{j=1}^{n}w_{j}^{2}, \quad T_{n}^{-}\ :=\
\sum_{j=1}^{n}\left(w_{j-1}-w_{j}\right)^{2},\quad\textrm{and}\quad
T_{n}^{+}\ :=\ \sum_{j=1}^{n}\left(w_{j+1}-w_{j}\right)^{2},
\end{equation*}
with $w_0:= 0$ and $w_{n+1}:=0$.
\end{proposition}
\begin{proof}
 Let $y=Ax$,
so $\left\|y\right\|\leq \eps$.  
For $j=0,\ldots,n$, 
let
\begin{equation}\label{E1}
E_j := |w_{j+1}-w_j|.
\end{equation}
Note that 
\begin{equation} \label{eq:n2}
\sum_{j=1}^{n}E^2_j\ =\ T^{+}_{n} \quad \mbox{ and } \quad
\sum_{j=1}^{n}E_{j-1}^2\ =\ T^{-}_{n}.
\end{equation}
For $k\in\Z$, let
$a_k:=\left\|A_{n,k}\widetilde{x}_{n,k}\right\|$,
where 
$\widetilde{x}_{n,k}=(w_{1}x_{k+1},w_{2}x_{k+2},\ldots,w_{n}x_{k+n})^{T}$,
and let
$b_{k}:=
\left\|\widetilde{x}_{n,k}\right\|$.
We will prove that $a_{k}\leq (\eps +\eps_n)b_{k}$ for some $k\in\Z$ with $b_k\neq 0$,
which will show that $\nu(A_{n,k})\leq \eps+\eps_n$.

Note first that, using the notation (\ref{E1}),
\begin{eqnarray} \label{eq:aj}
a_k^2&=&\sum_{j=1}^{n}\left\|w_{j-1}\alpha_{j+k-1}x_{j+k-1}\ +\ w_j\beta_{j+k}x_{j+k}\ +\ w_{j+1}\gamma_{j+k+1}x_{j+k+1}\right\|_X^{2}\\ \nonumber
&=&\sum_{j=1}^{n}\left\|(w_{j-1}-w_j) \alpha_{j+k-1}x_{j+k-1}\ +\ w_jy_{j+k}\ +\ (w_{j+1}-w_j)\gamma_{j+k+1}x_{j+k+1}\right\|_X^{2}\\ \label{eq:ss}
&\leq& \sum_{j=1}^{n}\left(w_j\|y_{j+k}\|_X\ +\
E_{j-1}\|\alpha\|_\infty \|x_{j+k-1}\|_X+E_j\|\gamma\|_\infty \|x_{j+k+1}\|_X\right)^{2}.
\end{eqnarray}
So, for all $\theta >0$ and $\phi>0$, by Lemma \ref{Simon_theta_inequality}, 
\begin{eqnarray} \nonumber
a_{k}^{2}
&\leq&\sum_{j=1}^{n}\Big[(1+\theta)w_j^{2}\|y_{j+k}\|_X^2\ \\ \nonumber
& & \hspace{4ex} +\ (1+\theta^{-1})\big(E_{j-1}\|\alpha\|_\infty \|x_{j+k-1}\|_X+E_j\|\gamma\|_\infty \|x_{j+k+1}\|_X\big)^{2}\Big]\\ \nonumber
&\leq&\sum_{j=1}^{n}\Big[(1+\theta)w_j^{2}\|y_{j+k}\|_X^2  + (1+\theta^{-1}) \cdot \\ \label{eq:ss2}
& &
\hspace{1ex} \left((1+\phi)E^2_{j-1}\|\alpha\|^2_\infty \|x_{j+k-1}\|_X^2 + (1+\phi^{-1})E^2_j\|\gamma\|^2_\infty \|x_{j+k+1}\|_X^2\right)\Big]. \hspace*{3ex}
\end{eqnarray}
Thus, and recalling \eqref{eq:n2} and that $\|x\| =1$,
\begin{eqnarray*}
\sum_{k\in\Z}a_{k}^{2}
&\leq& \sum_{j=1}^n\Big[(1+\theta)\|y\|^2w_j^{2}   +
(1+\theta^{-1}) \cdot \\
& & \hspace{4ex} \left((1+\phi)E^2_{j-1}\|\alpha\|^2_\infty + (1+\phi^{-1})E^2_j\|\gamma\|^2_\infty \right)\Big]\\
&=&(1+\theta)\left\|y\right\|^{2}S_{n}\ +\ (1+\theta^{-1})\left[(1+\phi)\left\|\alpha\right\|_{\infty}^{2}T_{n}^{-}+(1+\phi^{-1})\left\|\gamma\right\|_{\infty}^{2}T_{n}^{+}\right].
\end{eqnarray*}
Similarly,
$\sum_{k\in\Z}b_{k}^{2}=S_{n}\left\|x\right\|^{2}=S_{n}$. Now, by
Lemma \ref{Simon_theta_inequality} and recalling  \eqref{eq:epsn1},
\begin{equation*} \label{fb}
\inf_{\phi>0}\left[(1+\phi)\left\|\alpha\right\|_{\infty}^{2}T_{n}^{-}+(1+\phi^{-1})\left\|\gamma\right\|_{\infty}^{2}T^{+}_{n}\right]\ \le \varepsilon_n^{2} S_n.
\end{equation*}
Thus
\begin{equation} \label{SB}
\sum_{k\in\Z}a_{k}^{2}\
 \leq\ \Big((1+\theta)\|y\|^{2}\ +\ (1+\theta^{-1})
\eps_n^{2}\Big)\sum_{k\in\Z}b_{k}^{2},
\end{equation}
for all $\theta>0$. Applying Lemma \ref{Simon_theta_inequality}
again, we see that
\begin{equation} \label{sb}
\inf_{\theta>0}\Big((1+\theta)\|y\|^{2}\ +\ (1+\theta^{-1})
\eps_n^{2}\Big)\ =\ (\|y\|\,+\,\eps_n)^{2}\leq (\eps\,+\,\eps_n)^2,
\end{equation}
so that
\[
\sum_{k\in\Z}a_{k}^{2}\ \leq\
\left(\eps+\eps_n\right)^{2}\sum_{k\in\Z}b_{k}^{2}.
\]
Thus, either $a_{k} > \left( \eps+\eps_n\right)b_{k}$, for some $k\in\Z$, in which case also $a_{k}< \left( \eps+\eps_n\right)b_{k}$, for some $k\in \Z$, which implies that $b_k>0$, or $a_{k} \leq  \left( \eps+\eps_n\right)b_{k}$, for all $k\in \Z$, so for some $k$ with $b_k\neq 0$.
\end{proof}

\begin{corollary} \label{cor:nubound}
Let $n\in\N$, suppose that $w_j\in\R$, for $j=1,...,n$, with
at least one $w_j$ non-zero, and that $A\in L(E)$ is tridiagonal.  Then
\begin{equation} \label{eq:infon}
\inf_{k\in \Z} \nu(A_{n,k})\ \leq\ \nu(A)+\eps_n,
\end{equation} 
where $\eps_n$ is given by \eqref{eq:epsn1}.
\end{corollary}
\begin{proof}
Let $\eta>0$. By definition of $\nu(A)$ there exists $x\in E$ with $\|x\|=1$ such that $\|Ax\| \leq \nu(A)+\eta$. By Proposition \ref{prop:tau_main}, $\nu(A_{n,k})\leq \nu(A)+\eps_n+\eta$, for some $k\in \Z$. Since this holds for all $\eta>0$ the result follows.
\end{proof}

It is not true, for all tridiagonal $A\in L(E)$, that $\nu(A_{n,k}) \leq \nu(A)+\eps_n$, for some $k\in \Z$ (a strengthened version of \eqref{eq:infon}). For consider $B=A-\lambda I$, where $A$ is as defined below \eqref{eq:PseudInc2}, so that $X=\C$, and choose $\lambda = 1+\eps$, for some $\eps\geq 0$. For $x=(x_k)_{k\in \Z} \in E = \ell^2(\Z)$, $(Bx)_k = (\beta_k-1-\eps)x_k$, $k\in \Z$, with $\beta_k=\tanh(k)$, so that $\nu(B)\leq 1+\eps-\beta_k$, for all $k\in \Z$, i.e., $\nu(B)\leq  \eps$. But (cf.~\eqref{eq:low}), for each $k\in \Z$, $\nu(B_{n,k})\geq 1+\eps-\beta_{k+n}>\eps=\eps+\eps_n$, since $\eps_n=0$ as $B$ is diagonal.

\begin{theorem}\label{weighted_norm_two_refined}
Suppose that $n\in \N$ and $w_j\in\R$, for $j=1,...,n$, with
at least one $w_j$ non-zero. Then
\begin{equation} \label{eq:Mtaud}
 \Speps A \ \subset\
\widehat\Sigma^n_{\eps+\eps_n}(A), \quad \eps \geq 0, \qquad\textrm{and}
\qquad \speps A\ \subset\ \sigma^n_{\eps+\eps_n}(A), \quad \eps>0,
\end{equation}
with $\eps_n$ given by \eqref{eq:epsn1}, in particular the second inclusion of  \eqref{incl:met1} holds with $\eps_n$ given by \eqref{eq:epsn1}. If $X$ has Globevnik's property, then also the first inclusion of \eqref{incl:met1} holds with $\eps_n$ given by \eqref{eq:epsn1}.
\end{theorem}
\begin{proof}
To see that the second inclusion in \eqref{eq:Mtaud}/\eqref{incl:met1} holds, suppose that $\lambda\in\speps A$. Then, by \eqref{eq:spepsnu} and the definition \eqref{eq:invNorm},  either $\nu(A-\lambda I)<\eps$ or  $\nu(A^*-\lambda I)<\eps$. By Corollary \ref{cor:nubound} it follows that either $\nu(A_{n,k}-\lambda I_n)<\eps+\eps_n$, for some $k\in \Z$, or  $\nu(A^*_{n,k}-\lambda I_n)<\eps+\eps_n$, for some $k\in \Z$. Thus, recalling the notation \eqref{eq:invNorm},  $\mu(A_{n,k}-\lambda I_n)<\eps+\eps_n$, for some $k\in \Z$, so that $\lambda \in \specn_{\eps+\eps_n} A_{n,k}$ by \eqref{eq:spepsnu}, so that $\lambda \in \sigma^n_{\eps+\eps_n}(A)$. The first inclusion in \eqref{eq:Mtaud} follows from the second inclusion by taking intersections, noting \eqref{eq:oc},  the first sentence of Proposition \ref{prop:same}, and \eqref{eq:wideS}. This implies, by the second sentence of Proposition \ref{prop:same}, that the first inclusion of \eqref{incl:met1} holds if $X$ has the Globevnik property.
\end{proof}


If we evaluate \eqref{eq:epsn1} in the case when all weights
$w_1,...,w_n$ are equal to $1$, we get that $S_n=n$ and
$T_n^-=T_n^+=1$, so that
\begin{equation} \label{eq:weights_equal}
\eps_n\ =\ \frac 1{\sqrt n}\
(\|\alpha\|_\infty\,+\,\|\gamma\|_\infty),
\end{equation}
and hence $\eps_n\to 0$ as $n\to\infty$.
In the next subsection we shall see that we can achieve much smaller values for $\eps_n$, smaller by a factor $<2\pi/\sqrt{n}$ for large $n$ (Corollary \ref{CorolA}), by minimising $\eps_n$ as a function of the weight vector $w$.

\subsection{Minimising $\eps_n$ via variation of the weight vector}
\label{sec:mini}
We will now minimise the ``penalty term'' $\eps_n$ as a function,
\eqref{eq:epsn1}, of the weight vector $w=(w_1,...,w_n)^T\in \R^n$. That is, we will compute
$$
\inf_{w\in \R^n,\, w\neq 0} \eps_n = \inf_{w\in \R^n,\, \|w\|=1} \eps_n,
$$
the equality of these two expressions following easily from \eqref{eq:epsn1}. It follows from the compactness of the set $\{w\in \R^n:\, \|w\|=1\}$ that this infimum is achieved by some weight vector $w=(w_1,...,w_n)^T\in \R^n$, so that the infimum is a minimum. This is important because it implies that Theorem \ref{weighted_norm_two_refined} applies with this minimal weight vector and $\eps_n$ taking its minimal value.

Let us abbreviate $\|\alpha\|_{\infty}=:r$
and $\|\gamma\|_{\infty}=:s$, so that
\begin{equation} \label{eq:epsncompl}
\eps_n^2\ =\ \left(r\sqrt{\frac{T_{n}^{-}}{S_{n}}}\ +\
s\sqrt{\frac{T_{n}^{+}}{S_{n}}}\right)^2 \ =\
\inf_{\tau>0}\left(r^{2}(1+\tau)\frac{T_{n}^{-}}{S_{n}}+s^{2}(1+\tau^{-1})\frac{T_{n}^{+}}{S_{n}}\right),
\end{equation}
by \eqref{eq:epsn1} and Lemma \ref{Simon_theta_inequality}. (This step \eqref{eq:epsncompl} may appear to be introducing additional complexity, but we shall see that it helpfully reduces a large part of the minimisation to an eigenvalue computation.) Now put
\[
B_{n}\ :=\ \left(\begin{array}{ccccc}
1   &      &     &    &   \\
 -1 & 1    &     &    &   \\
    & \ddots   & \ddots   &    &  \\
    &      & -1  &1   &   \\
    &      &     &-1  & 1 \end{array}\right)\ \in\ \R^{n\times n}.
\]
Then
\[
T_{n}^{-}\ =\ \|B_{n}w\|^{2}\ =\ (B_{n}w)^{T}(B_{n}w)\ =\
w^{T}B_{n}^{T}B_{n}w,
\]
while
\[
T_{n}^{+}\ =\ \|B_{n}^{T}w\|^{2}\ =\ (B_{n}^Tw)^{T}(B_{n}^Tw)\ =\
w^{T}B_{n}B_{n}^{T}w
\]
and
$S_{n}\ =\ \|w\|^{2}\ =\ w^{T}w$.
So
\begin{equation}\label{red_star}
r^{2}(1+\tau)\frac{T_{n}^{-}}{S_{n}}+s^{2}(1+\tau^{-1})\frac{T_{n}^{+}}{S_{n}}\
=\ \frac{w^{T}D_{n}w}{w^{T}w}
\end{equation}
with
$
D_{n}\ :=\ r^{2}(1+\tau)B_{n}^{T}B_{n}+s^{2}(1+\tau^{-1})B_{n}B_{n}^{T}.
$
It
is easy to see that $D_{n}$ is symmetric, real and positive
definite. In fact,
$
w^{T}D_{n}w\ \ge\ \left[
r^{2}(1+\tau)+s^{2}(1+\tau^{-1})\right]\lambda_{1},
$
where $\lambda_{1}>0$ is the smallest eigenvalue of $B_{n}^{T}B_{n}$
and $B_{n}B_{n}^{T}$. Our aim is to compute the sharpest lower bound on
\eqref{red_star}, i.e. the smallest eigenvalue of $D_n$, since
\begin{equation} \label{aim}
\inf_{w\in \R^n,\, w\neq 0} \eps^2_n = \inf_{\tau > 0} \inf_{w\in \R^n,\, w\neq 0} \frac{w^{T}D_{n}w}{w^{T}w}.
\end{equation}

To compute the infimum \eqref{aim},
 define $E_n(\phi)$, for $0\leq \phi \leq 1$, to be the real, symmetric, positive definite matrix
$
E_{n}(\phi)\ :=\ \phi B_{n}^{T}B_{n}+(1-\phi)B_{n}B_{n}^{T},
$
so that $E_1(\phi)=1$ while, for $n = 2,3,\dots$,
\[
E_{n}(\phi)=\begin{pmatrix}
1+\phi   & -1       &           &    &   \\
 -1 & 2        &    -1     &    &   \\
    & \ddots   & \ddots    & \ddots   &  \\
    &          & -1        &2    & -1  \\
    &          &           &-1   & 2-\phi
\end{pmatrix}\ \in\ \R^{n\times n}.
\]
Noting that
\begin{equation*}\label{DnEphi}
D_{n}\ =\ \left(r^{2}(1+\tau)+s^{2}(1+\tau^{-1})\right)
E_n\left(\frac{r^{2}(1+\tau)}{r^{2}(1+\tau)+s^{2}(1+\tau^{-1})}\right),
\end{equation*}
and where $\varrho_n(\phi):= \lambda_{\sf min}\left(E_{n}(\phi)\right)>0$ denotes
the smallest eigenvalue of $E_{n}(\phi)$, it follows from
(\ref{aim}) that
\begin{equation} \label{aim2}
\inf_{w\in \R^n,\, w\neq 0} \eps^2_n = \inf_{\tau > 0} \left( \left(r^{2}(1+\tau)+s^{2}(1+\tau^{-1})\right)
\varrho_n\left(\frac{r^{2}(1+\tau)}{r^{2}(1+\tau)+s^{2}(1+\tau^{-1})}\right) \right).
\end{equation}
Further, if $\phi:=\frac{r^{2}\left(1+\tau\right)}{r^{2}\left(1+\tau\right)+s^{2}\left(1+\tau^{-1}\right)}$, it follows that
$\tau=\frac{s^{2}\phi}{r^{2}(1-\phi)}$ and that $\{\phi:\tau>0\}=(0,1)$. Thus \eqref{aim2} can be rewritten as
\begin{equation} \label{aim3}
\inf_{w\in \R^n,\, w\neq 0} \eps^2_n\ =\
\inf_{0<\phi<1}\left[\left(\frac{r^{2}}{\phi}+\frac{s^{2}}{1-\phi}\right)\varrho_n(\phi)\right].
\end{equation}
The right hand side of equation (\ref{aim3}) already solves our minimisation problem. But we calculate now a more explicit expression for $\varrho_n(\phi)$, which will make (\ref{aim3}) more easily computable.

Clearly, $\varrho_1(\phi) = 1$. If $\lambda$ is an eigenvalue of $E_n(\phi)$ for $n\geq 2$ then, since $E_n(\phi)$ is positive definite, and applying the Gershgorin circle theorem, we see that $0<\lambda \leq 4$. It is rather straightforward to show that 4 is not an eigenvalue so that $0<\lambda<4$. But, in any case, for $n\geq 2$, $\lambda<4$ is an eigenvalue of $E_n(\phi)$, with eigenvector $v=(v_{n},\ldots,v_{1})^{T}\in \R^n$, if and only if 
\begin{equation} \label{lambda}
\lambda=2-2\cos\theta=4\sin^{2}\frac{\theta}{2}, \quad \mbox{ for some } \theta \in (0,\pi),
\end{equation}
and
\begin{eqnarray}
\label{March2009_2} \left[(1+\phi)-\lambda\right]v_{n}-v_{n-1}&=&0,\\
\label{March2009_3} -v_{j+1}+(2-\lambda) v_{j}-v_{j-1}&=&0,\qquad
\textrm{ for }j= 2,\ldots,n-1,\\
\label{March2009_4} -v_{2}+\left[(2-\phi)-\lambda\right]v_{1}&=&0.
\end{eqnarray}
Equation (\ref{March2009_3}) holds if and only if $v_j$ is a linear combination of $\cos((j-1)\theta)$ and $\sin((j-1)\theta)$. To within multiplication by a constant, the linear combination that also satisfies \eqref{March2009_4} is
\begin{eqnarray} \nonumber
v_{j}&=&\sin\theta\cos\left[(j-1)\theta\right]+\left(\cos\theta - \phi\right)\sin\left[(j-1)\theta\right]\\
&=&\sin(j\theta)-\phi\sin\left((j-1)\theta\right),\quad j=1,\ldots,
n.\;  \label{13March}
\end{eqnarray}
Substituting in
(\ref{March2009_2}), we see that $\lambda<4$ is an eigenvalue of
$E_{n}(\phi)$ for $n\geq 2$ if and only if \eqref{lambda} holds and $F(\theta)=0$, where
\[
F(t)\ :=\ (2\cos t -1)\sin(n t)-\sin((n-1)
t)\,+\,\phi(1-\phi)\sin((n-1) t),\qquad t\geq 0.
\]

Now, for $n\in\N$,
\begin{eqnarray}
F(t)
&=&-\sin(n t)+\sin\left(\left(n+1\right)t\right)+\phi(1-\phi)\sin\left(\left(n-1\right)t\right)\label{useful}\\
&=&\textstyle2\sin\frac{t}{2}\,\cos\left(\left(n+\frac{1}{2}\right)t\right)+\phi(1-\phi)\sin\left(\left(n-1\right)
t\right).\label{useful2}
\end{eqnarray}
Thus, noting that $\phi(1-\phi)\leq 1/4$ and setting $\sigma := \pi/(n+2)\leq \pi/3$, so that $1/2\leq \cos \sigma < 1$, we see from \eqref{useful} that
\begin{eqnarray*}
\textstyle F(\sigma)
\leq\textstyle -\sin 2\sigma + \sin \sigma +\frac{1}{4}\sin3\sigma
= \textstyle \frac{1}{4} \sin \sigma (3 + 4 \cos^2 \sigma - 8 \cos \sigma) \leq 0.
\end{eqnarray*}
Moreover, from \eqref{useful2}, while $F(0)=0$,
$\textstyle F(t)
>\phi(1-\phi)\sin((n-1)t)\ \ge\ 0,
$
for $0 < t < \frac{\pi}{2n+1}$.
Let $\theta_n(\phi)$ denote the smallest positive solution of the equation
$F(t)=0$. Then we have shown that
$\theta_n(\phi) \in [\frac{\pi}{2n+1}, \frac{\pi}{n+2}]\subset(0,\pi)$, and hence that the smallest eigenvalue of $E_n(\phi)$ is in $(0,4)$, precisely
\begin{equation} \label{eq:muphi}
\varrho_n(\phi)\ =\ \lambda_{\sf min}(E_n(\phi))\ =\
4\sin^{2}\left(\frac{\theta_n(\phi)}{2}\right).
\end{equation}
(We have shown this equation for $n\geq 2$, and it holds also for $n=1$ since $\varrho_1(\phi)=1$ and $\theta_1(\phi)=\pi/3$.)
Further, $\theta_n(\phi)$ is the unique solution of $F(t)=0$ in $[\frac{\pi}{2n+1}, \frac{\pi}{n+2}]$. To see this in the case $n\geq 2$,
note that, defining $\eta := \phi(1-\phi)\in [0,1/4]$, it follows from (\ref{useful2}) that
\begin{eqnarray*}
\textstyle F^\prime(t) & = & \textstyle \cos \frac{t}{2} \cos \left(\left(n+\frac{1}{2}\right)t\right) - (2n+1) \sin \frac{t}{2} \sin\left(\left(n+\frac{1}{2}\right)t\right)\\
& & \hspace{4ex}  +\eta(n-1)\cos((n-1)t) \\
&=& \textstyle \cos\left(\left(n+\frac{1}{2}\right)t\right){\mathcal A}_n(t) -\sin\left(\left(n+\frac{1}{2}\right)t\right){\mathcal B}_n(t),
\end{eqnarray*}
where ${\mathcal A}_n(t) := \cos \frac{t}{2} + \eta(n-1) \cos \frac{3t}{2}$ and ${\mathcal B}_n(t) := (2n+1)\sin \frac{t}{2} - \eta(n-1)\sin \frac{3t}{2}$.
Now, for
$n\geq 2$ and $t \in (\frac{\pi}{2n+1}, \frac{\pi}{n+2})$, it holds that $\cos\left(\left(n+\frac{1}{2}\right)t\right)< 0$, $\sin\left(\left(n+\frac{1}{2}\right)t\right)> 0$,  ${\mathcal A}_n(t)>0$, and ${\mathcal B}_n(t)> 0$ (the last inequality holding since $\eta(n-1)\sin \frac{3t}{2} \leq 3\eta (n-1)\sin\frac{t}{2}\leq \frac{3}{4} (n-1)\sin\frac{t}{2}$). Thus $F^\prime(t) < 0$ for $t \in (\frac{\pi}{2n+1}, \frac{\pi}{n+2})$, which implies that $F$ has at most one zero in $[\frac{\pi}{2n+1}, \frac{\pi}{n+2}]$.

So we have shown that $\theta_n(\phi)$ is the unique solution of $F(t)=0$ in $[\frac{\pi}{2n+1}, \frac{\pi}{n+2}]$.
Combining this with
\eqref{aim3} and \eqref{eq:muphi}, we see that we have completed our
aim of minimising $\eps^2_n$, proving the following
corollary of Theorem \ref{weighted_norm_two_refined}.

%
\begin{theorem}\label{method0}
For all $n\in\N$ the second inclusion in \eqref{incl:met1} holds with
\begin{equation} \label{f_n_method0}
\eps_n\ :=\ 2\inf_{0<\phi< 1}
\left(\sqrt{\frac{\|\alpha\|_\infty^{2}}{\phi}+\frac{\|\gamma\|_\infty^{2}}{1-\phi}}
\ \sin\frac{\theta_n(\phi)}{2}\right),
\end{equation}
where $\theta_n(\phi)$ is the unique solution in the range
$\left[\frac{\pi}{2n+1}\,,\,\frac{\pi}{n+2}\right]$
of the equation $F(t)=0$, i.e.\ of
\begin{equation} \label{eq:Ft=0}
2\sin\frac{t}{2}\cos\left(\left(n+\frac{1}{2}\right)t\right)\ +\ \phi(1-\phi)\sin\left(\left(n-1\right) t\right)\ =\ 0.
\end{equation}
If $X$ has the Globevnik property, then also the first inclusion in \eqref{incl:met1} holds with $\eps_n$ given by \eqref{f_n_method0}.
\end{theorem}

The following properties of  $\theta_n$ in Theorem \ref{method0} are
worth noting:
\begin{eqnarray} \label{prop3} \theta_1(\phi) = \frac{\pi}{3}, & \quad 0\le \phi\le 1;
\\\label{prop1}
\theta_n(\phi)=\theta_n(1-\phi), & \quad 0\le \phi\le 1, \;
n\in\N;\\  \label{prop2} \frac{\pi}{2n+1} = \theta_n(0) = \theta_n(1) \leq
\theta_n(\phi) \le \theta_n\left({\textstyle\frac{1}{2}}\right), &
\quad 0\le\phi\le1, \; n\in\N.
\end{eqnarray}
These properties are evident from the definition of $\theta_n$, except for the bound $\theta_n(\phi)\le
\theta_n(\frac{1}{2})$. One way to see this bound is via
\eqref{eq:muphi}, as it follows from $E_{n}(\frac{1}{2})=\frac
12(E_{n}(\phi)+E_{n}(1-\phi))$ that
\begin{eqnarray*}
\textstyle\varrho_n(\frac{1}{2})&=&\frac{1}{2}\min_{\|w\|=1}\left(\phi
w^{T}E_{n}(\phi)w+w^{T}E_{n}(1-\phi)w\right)\\
&\geq&\frac{1}{2}\left[\varrho_n(\phi)+\varrho_n((1-\phi))\right]\,
=\, \varrho_n(\phi),
\end{eqnarray*}
since $\varrho_n(\phi)=\varrho_n(1-\phi)$ as a consequence of \eqref{prop1} and \eqref{eq:muphi}.

We remark also that $\theta_n(\frac{1}{2})$ is the unique solution
of the equation
\begin{equation}\label{equi1s}
G(t):= \textstyle 2\cos\left(\left(\frac{n+1}{2}\right)t\right)-\cos\left(\left(\frac{n-1}{2}\right)t\right)\ =\ 0
\end{equation}
in the interval $(\frac{\pi}{n+3},\frac{\pi}{n+2}]$. To see this
claim it is enough to check that \eqref{equi1s} and \eqref{eq:Ft=0}
have the same solutions in $(0,\frac{\pi}{n+2}]$ when
$\phi=\frac{1}{2}$,  and to show that $\theta_n(\frac{1}{2})>
\frac{\pi}{n+3}$. To see the first statement multiply both sides of
\eqref{equi1s} by the positive term
$2\sin\frac{(n+1)t}{2}-\sin\frac{(n-1)t}{2}$, giving
\[
\textstyle
2\sin\left((n+1)t\right)-2\sin(nt)+\frac{1}{2}\sin\left((n-1)t\right)\
=\ 0.
\]
This is equivalent to \eqref{equi1s} with $\phi=\frac{1}{2}$, so that
$\theta_n(\frac{1}{2})$ is the unique solution of $G(t)=0$ in
$(0, \frac{\pi}{n+2}]$. To see the second statement, observe that
$
\textstyle
G(\frac{\pi}{n+3})\,=\,2\sin\frac{\pi}{n+3}-\sin\frac{2\pi}{n+3}\
>\ 0,
$
whereas, where $\sigma = \frac{\pi}{n+2}$,
$
\textstyle
G(\sigma)\, =\,2\sin\frac{\sigma}{2}-\sin\frac{3\sigma}{2}\, =\,\sin\frac{\sigma}{2}\left(4\sin^2\frac{\sigma}{2}-1\right)
\leq\ 0,
$
since $0<\frac{\sigma}{2} \leq \frac{\pi}{6}$. Thus
$\theta_n(\frac{1}{2})> \frac{\pi}{n+3}$.

In general, it appears not possible to explicitly compute the value
of $\phi$ for which the bracket in \eqref{f_n_method0} is minimised.
But it follows easily from the bound \eqref{prop2} that, in the case
$r = \|\alpha\|_\infty = 0$ the infimum is attained by setting
$\phi=0$, and this infimum is $\eps_n\ =\
2\|\gamma\|_{\infty}\sin\frac{\pi}{4n+2}$. Likewise, when $s =
\|\gamma\|_\infty = 0$ we see that the infimum is attained by
setting $\phi=1$, and this infimum is $\eps_n\ =\
2\|\alpha\|_{\infty}\sin\frac{\pi}{4n+2}$. Thus we have computed the
mimimal value of $\eps_n$ completely explicitly in the case when $A$
is bidiagonal, giving the following corollary.

\begin{corollary}\label{minimum_weighted_norm_two}
If $n\in\N$ and $\gamma=0$ or $\alpha=0$, then
\eqref{f_n_method0} reduces to
\begin{equation} \label{eq:epsn_method1}
\eps_n\ =\
2\left(\|\alpha\|_{\infty}+\|\gamma\|_\infty\right)\sin\left(\frac{\pi}{4n+2}\right).
\end{equation}
\end{corollary}
The example of the shift operator (see  Example \ref{ex:shift} and
\S\ref{sec:shift} below) shows that the value of $\eps_n$ in
Corollary~\ref{minimum_weighted_norm_two} is the best possible, in
the sense that, for every $n\in\N$, there exists a bidiagonal $A\in L(E)$ for
which \eqref{incl:met1} is not true for any smaller value for
$\eps_n$.

In the case that $r= \|\alpha\|_\infty\neq 0$ and $s=
\|\gamma\|_\infty\neq 0$ it is not clear what the infimum in
\eqref{f_n_method0} is explicitly when $n\ge 2$. However, since
$\theta_n(\phi)$ depends continuously on $\phi$ and is bounded as a
function of $\phi$ on $(0,1)$, it is clear that the infimum is
attained at some $\phi\in (0,1)$. Further, as a consequence of
\eqref{prop1}, it is easy to see that the infimum is attained
for some $\phi$ in the range $(0,\frac{1}{2})$ if $r<s$, for some
$\phi$ in the range $(\frac{1}{2},1)$ if $r>s$.  A simple choice of
$\phi$ which has these properties and which reduces to the optimal
values of $\phi$, $\phi=0$ and $\phi=1$, respectively, when $r=0$
and $s=0$, is
$\phi=\frac{r}{r+s}=\frac{\|\alpha\|_\infty}{\|\alpha\|_\infty+\|\gamma\|_\infty}$.
This choice is further motivated by the fact that this is the unique
$\phi$ that attains the infimum in \eqref{f_n_method0} in the case
$n=1$ (which is an easy calculation in view of \eqref{prop3}).
 If we evaluate the bracket in \eqref{f_n_method0} for this choice
we obtain the following corollary which includes Corollary \ref{minimum_weighted_norm_two} as a special case.
\begin{corollary}\label{CorolA}
For all $n\in\N$, $\eps_n$ given by \eqref{f_n_method0} satisfies
\[
\eps_n \leq \eta_n := 2\left(\left\|\alpha\right\|_{\infty}+\left\|\gamma\right\|_{\infty}\right)\sin\frac{\theta}{2},
\]
where $\theta$ is the unique solution in the range
$\left[\frac{\pi}{2n+1}\,,\,\frac{\pi}{n+2}\right]$
of the equation
\[
2\sin\Big(\frac{t}{2}\Big)\cos\Big(\Big(n+\frac{1}{2}\Big)t\Big)\ +\ \frac{\|\alpha\|_\infty\|\gamma\|_\infty}{(\|\alpha\|_\infty+\|\gamma\|_\infty)^2} \sin\left(\left(n-1\right) t\right)\ =\ 0.
\]
In particular, $\eta_1=\eps_1 = \|\alpha\|_\infty + \|\gamma_\infty\|_\infty$.
\end{corollary}
The following even more explicit result is obtained if we evaluate
the bracket in \eqref{f_n_method0} for $\phi=\frac 12$, noting the
equivalent characterisation of $\theta_n(\frac{1}{2})$ above, that
it is the unique solution in
$\left(\frac{\pi}{n+3}\,,\,\frac{\pi}{n+2}\right]$ of
\eqref{equi1s}.
\begin{corollary}\label{CorolB}
For all $n\in\N$, $\eps_n$ given by \eqref{f_n_method0} satisfies
\[
\eps_n\ \leq\ \eta'_n\ :=\
2\sqrt{2}\sqrt{\left\|\alpha\right\|^{2}_{\infty}+\left\|\gamma\right\|^{2}_{\infty}}\sin\frac{\theta}{2},
\]
where $\theta$ is the unique solution in the range
$\left(\frac{\pi}{n+3}\,,\,\frac{\pi}{n+2}\right]$ of \eqref{equi1s}.
\end{corollary}

Note that Corollaries \ref{CorolA} and \ref{CorolB} give the same upper bound for $\eps_n$ when $\|\alpha\|_\infty= \|\gamma\|_\infty$, but when
$\|\alpha\|_\infty\neq  \|\gamma\|_\infty$, as a consequence of
\eqref{prop2}, Corollary \ref{CorolA} is sharper ($\eta_n< \eta'_n$).
\section{The $\pi$ method: periodised principal submatrices} \label{sec:piproof}

In this section we prove the inclusion \eqref{incl:met1*} that is the
equivalent of \eqref{incl:met1} but with periodised submatrices
\eqref{eq:Ankper} instead of \eqref{eq:Ank}.
We start with Proposition \ref{prop:pi_main}, the equivalent of Proposition~\ref{prop:tau_main} for the $\pi$ method. The matrix referenced in  this proposition is a generalisation of the $\pi$-method matrix $A_{n,k}^{\per,t}$ of \eqref{eq:Ankper}, defined, for any $\mathfrak{a},\mathfrak{c}\in \ell^\infty(\Z,L(X))$, by
$$
A_{n,k}^{\mathfrak{a},\mathfrak{c}}\ :=\  A_{n,k}+B_{n,k}^{\mathfrak{a},\mathfrak{c}},  \qquad k\in \Z, \;\; n\in \N,
$$
where $B_{n,k}^{\mathfrak{a},\mathfrak{c}}$ is the $n\times n$ matrix whose entry in row $i$, column $j$ is $\delta_{i,1}\delta_{j,n}  \mathfrak{a}_k+ \delta_{i,n}\delta_{j,1}\mathfrak{c}_k$. Thus, for $n\geq 3$,
$A^{\mathfrak{a},\mathfrak{c}}_{n,k}$ takes the form \eqref{eq:Ankper}, but with the top right entry $t\alpha_k$ replaced by $\mathfrak{a}_k$, and the bottom left entry $\bar t \gamma_{k+n+1}$ replaced by $\mathfrak{c}_k$. Importantly,
\begin{equation} \label{eq:special}
A_{n,k}^{\per,t} = A_{n,k}^{\mathfrak{a},\mathfrak{c}} \qquad \mbox{if} \qquad \mathfrak{a}_k:=t\alpha _k \mbox{ and }\mathfrak{c}_k:=\bar t \gamma_{k+n+1}, \quad k\in \Z.
\end{equation}
 (Of course, also $A_{n,k}^{\mathfrak{a},\mathfrak{c}}=A_{n,k}$, if $\mathfrak{a}=\mathfrak{c}=0$.)

\begin{proposition}\label{prop:pi_main}
Let $\eps>0$ and $n\in\N$, suppose that $w_j\in\R$, for $j=1,...,n$, with
at least one $w_j$ non-zero, that $\mathfrak{a},\mathfrak{c}\in \ell^\infty(\Z,L(X))$, with $\|\mathfrak{a}\|_\infty \leq \|\alpha\|_\infty$ and $\|\mathfrak{c}\|_\infty \leq \|\gamma\|_\infty$, and that $A\in L(E)$ is tridiagonal and $\|Ax\|\leq \eps$, for some $x\in E$ with $\|x\|=1$.  Then, for some $k\in \Z$ it holds that
\begin{equation} \label{eq:tau_main}
\nu(A^{\mathfrak{a},\mathfrak{c}}_{n,k}) \leq \eps+\eps'_n,
\end{equation}
with
\begin{equation} \label{eq:epsn2}
\eps'_n\ :=\ \left(
\left\|\alpha\right\|_{\infty}+\left\|\gamma\right\|_{\infty}\right)\sqrt{\frac{T_{n}}{S_{n}}},
\end{equation}
where
\[
S_{n}\ =\ \sum_{j=1}^{n}w_{j}^{2} \qquad\textrm{and}\qquad T_{n}\ =\
(w_{1}+w_{n})^{2}+\sum_{j=1}^{n-1}(w_{j+1}-w_{j})^{2},
\]
with $w_0:= 0$ and $w_{n+1}:=0$.
\end{proposition}
\begin{proof}
 Let $y=Ax$,
so $\left\|y\right\|\leq \eps$. We use the notation $E_j$, for $j=0,\ldots,n$, defined in \eqref{E1}
and, where
$\tilde{x}_{n,k}:=(w_{1}x_{k+1},w_{2}x_{k+2},\ldots,w_{n}x_{k+n})^{T}$, put
$a_{k}:=\|A^{\mathfrak{a},\mathfrak{c}}_{n,k}\tilde x_{n,k}\|$ and $b_k := \|\tilde x_{n,k}\|$, for
$k\in\Z$.
As in the proof of Proposition \ref{prop:tau_main}, we have that $\sum_{k\in \Z} b_k^2 = S_n$. Further, for $n\geq 2$ and  $k\in \Z$, using (\ref{E1}) and the Kronecker delta notation,
\begin{eqnarray*}
a_{k}^{2}&=& \left\|w_{1}y_{k+1}+w_{n}\mathfrak{a}_{k} x_{k+n}-w_{1}\alpha_k x_{k}+(w_{2}-w_{1})\gamma_{k+2}x_{k+2}\right\|_X^{2}\\
&&+\sum_{j=2}^{n-1}\left\|w_{j}y_{k+j}+(w_{j-1}-w_{j})\alpha_{k+j-1}x_{k+j-1}+(w_{j+1}-w_{j})\gamma_{k+j+1}x_{k+j+1}\right\|_X^{2}\\
&&+\left\|w_{n}y_{k+n}+(w_{n-1}-w_{n})\alpha_{k+n-1}x_{k+n-1}+w_{1}\mathfrak{c}_k  x_{k+1}-w_{n}\gamma_{k+n+1}x_{k+n+1}\right\|_X^{2}\\
&\leq& \sum_{j=1}^{n}\left(w_j\left\|y_{k+j}\right\|_X+ \|\alpha\|_\infty c_{j,k}
+\|\gamma\|_\infty d_{j,k}\right)^{2},
\end{eqnarray*}
(cf.~\eqref{eq:ss}) where, for $j=1,\ldots,n$ and $k\in \Z$,
\begin{equation}\label{method1_AB}
c_{j,k} := \delta_{j,1}w_n\|x_{k+n}\|_X  + E_{j-1}\|x_{k+j-1}\|_X \quad \mbox{and} \quad d_{j,k} := \delta_{j,n}w_1\|x_{k+1}\|_X + E_{j}\|x_{k+j+1}\|_X.
\end{equation}
After applying Lemma \ref{Simon_theta_inequality} twice, we get (cf.~\eqref{eq:ss2}) that,
for all $\theta
>0$ and $\phi>0$, 
\begin{eqnarray*}
a_{k}^{2} &\leq&\sum_{j=1}^{n}\Big[(1+\theta)w^2_{j}\left\|y_{k+j}\right\|_X^2 +(1+\theta^{-1}) \cdot \\
& & \hspace{4ex} \big( (1+\phi) \|\alpha\|^{2}_\infty c_{j,k}^2
 +
(1+\phi^{-1})\|\gamma\|^{2}_\infty d_{j,k}^2\big)\Big].
\end{eqnarray*}
Now,
recalling that $\|x\|=1$,
\begin{eqnarray*}
\sum_{k\in \Z} c_{1,k}^2 &=& \sum_{k\in \Z} \left(w_n \|x_{k+n}\|_X+w_1\|x_{k}\|_X\right)^{2}\\
&\leq&
w_{n}^{2}+2w_{1}w_{n}\sum_{k\in\Z}\left\|x_{k+n}x_{k}\right\|_X+w_{1}^{2}\,\leq\,
(w_{1}+w_{n})^{2},
\end{eqnarray*}
so that
$$
\sum_{j=1}^n\sum_{k\in \Z} c_{j,k}^2 \leq (w_1+w_n)^2 +\sum_{j=2}^n\sum_{k\in \Z} E_{j-1}^2 |x_{k+j-1}|^2 = T_n.
$$
Similarly,
$$
\sum_{k\in \Z} d_{n,k}^2 \leq
(w_{1}+w_{n})^{2}\quad \mbox{so that} \quad \sum_{j=1}^n\sum_{k\in \Z} d_{j,k}^2 \leq  T_n.
$$
Thus, applying Lemma \ref{Simon_theta_inequality} (cf.~\eqref{SB} and \eqref{sb}),
\begin{eqnarray} \nonumber
\sum_{k\in\Z}a_{k}^{2}
&\leq& \inf_{\theta,\phi>0} \Big[
(1+\theta)S_n \|y\|^2
+(1+\theta^{-1})T_n \cdot \\  \label{eq:eq1}
& & \hspace{4ex} \big[(1+\phi)\|\alpha\|_{\infty}^{2}
+(1+\phi^{-1})\|\gamma\|_{\infty}^{2}\big]\Big]\\ \nonumber
&=&\inf_{\theta>0} \left[(1+\theta)\|y\|^{2}S_{n}+(1+\theta^{-1})T_n\left(\|\alpha\|_{\infty}+\|\gamma\|_{\infty}\right)^{2}\right]\\ \nonumber
&\leq&\inf_{\theta>0} \Bigg[(1+\theta)\eps^{2}+(1+\theta^{-1})
\Bigg((\|\alpha\|_{\infty}+\|\gamma\|_{\infty})\sqrt{\frac{T_n}{S_n}}\Bigg)^{2}\Bigg]S_{n}\\ \nonumber
&=&\inf_{\theta>0} \Big[\big((1+\theta)\eps^{2}+(1+\theta^{-1})
{\eps'_n}^2\big)\Big]\ S_n\\ \label{eq:eq2}
&=&\big(\eps+\eps'_n\big)^{2} \sum_{k\in\Z}b_k^2.
\end{eqnarray}
Arguing as at the end of Proposition \ref{prop:tau_main} we deduce that 
 $a_{k}\leq (
\eps+\eps'_n)b_k$, for some $k\in\Z$ with $b_k\neq 0$.
We reach the same conclusion,  via  the same bounds \eqref{eq:eq1} and \eqref{eq:eq2}, also in the case $n=1$ when $S_n=w_1^2$, $T_n=4w_1^2$, and
$$
a_k^2 = w_1^2\left\|y_{k+1} + \mathfrak{a}_k x_{k+1}-\alpha_k x_k + \mathfrak{c}_{k}x_{k+1}-\gamma_{k+2}x_{k+2}\right\|_X^2, \qquad k\in \Z.
$$
Thus $\nu(A^{\mathfrak{a},\mathfrak{c}}_{n,k}) \leq \eps+\eps'_n$.
\end{proof}

This is the $\pi$ version of Corollary \ref{cor:nubound}, which has an identical proof.
\begin{corollary}\label{cor:nb2}
Let $n\in\N$, suppose that $w_j\in\R$, for $j=1,...,n$, with
at least one $w_j$ non-zero, that $\mathfrak{a},\mathfrak{c}\in \ell^\infty(\Z,L(X))$, with $\|\mathfrak{a}\|_\infty \leq \|\alpha\|_\infty$ and $\|\mathfrak{c}\|_\infty \leq \|\gamma\|_\infty$, and that $A\in L(E)$ is tridiagonal.  Then
\begin{equation} \label{eq:tau_main}
\inf_{k\in \Z}\, \nu(A^{\mathfrak{a},\mathfrak{c}}_{n,k}) \leq \nu(A)+\eps'_n,
\end{equation}
where $\eps'_n$ is given by \eqref{eq:epsn2}.
\end{corollary}

To state the main theorem of this section we introduce a generalisation of the notation \eqref{eq:pidef}. For $n\in \N$ and $\mathfrak{a},\mathfrak{c}\in \ell^\infty(\Z,L(X))$, let
\begin{equation} \label{eq:pidefgen}
\pi^{n,\mathfrak{a},\mathfrak{c}}_{\eps}(A)\ :=\ \bigcup_{k\in\Z} \speps A^{\mathfrak{a},\mathfrak{c}}_{n,k}, \quad \eps>0,
\quad\textrm{and}\quad \Pi^{n,\mathfrak{a},\mathfrak{c}}_{\eps}(A)\ :=\ \overline{\bigcup_{k\in\Z}
\Speps A^{\mathfrak{a},\mathfrak{c}}_{n,k}}, \quad \eps\geq 0.
\end{equation}
Clearly, for $n\in \N$ and $t\in \T$,
\begin{equation} \label{eq:pisp}
\pi^{n,t}_{\eps}(A)\ =\ \pi^{n,\mathfrak{a},\mathfrak{c}}_{\eps}(A), \quad \eps>0,
\qquad\textrm{and}\qquad \Pi^{n,t}_{\eps}(A)\ =\ \Pi^{n,\mathfrak{a},\mathfrak{c}}_{\eps}(A), \quad \eps\geq 0,
\end{equation}
if $\mathfrak{a}$ and $\mathfrak{c}$ are given by \eqref{eq:special}. Clearly, for all $n\in \N$ and $\mathfrak{a},\mathfrak{c}\in \ell^\infty(\Z,L(X))$, $\pi^{n,\mathfrak{a},\mathfrak{c}}_{\eps}(A)$ is open and  $\Pi^{n,\mathfrak{a},\mathfrak{c}}_{\eps}(A)$ is closed. Further, defining
\begin{equation} \label{eq:muac}
\mu^{\mathfrak{a},\mathfrak{c}}_n(B)\  :=\ \inf_{k\in \Z} \, \mu(A^{\mathfrak{a},\mathfrak{c}}_{n,k})  = \inf_{k\in \Z} \, \min\Big(\nu\left(A^{\mathfrak{a},\mathfrak{c}}_{n,k}\right),\,\nu\left((A^{\mathfrak{a},\mathfrak{c}}_{n,k})^*\right)\Big),
\end{equation}
it holds, for $n\in \N$ and $\mathfrak{a},\mathfrak{c}\in \ell^\infty(\Z,L(X))$, by arguing as in the proof of Proposition \ref{prop:same}, that
\begin{equation} \label{eq:piAAlt}
\pi^{n,\mathfrak{a},\mathfrak{c}}_{\eps}(A)\ =\  \{\lambda\in \C:\mu^{\mathfrak{a},\mathfrak{c}}_n(A-\lambda I)<\eps\}, \qquad \eps>0, 
\end{equation} 
and that, if $X$ has the Globevnik property, also
\begin{equation} \label{eq:piAAlt2}
\overline{\pi^{n,\mathfrak{a},\mathfrak{c}}_{\eps}(A)}\ =\  \Pi^{n,\mathfrak{a},\mathfrak{c}}_{\eps}(A)\ =\  \{\lambda\in \C:\mu^{\mathfrak{a},\mathfrak{c}}_n(A-\lambda I)\leq \eps\}, \qquad \eps> 0.
\end{equation} 
Thus, if $X$ has the Globevnik property,
\begin{equation} \label{eq:piPi}
\Pi^{n,\mathfrak{a},\mathfrak{c}}_{\eps}(A)\ =\ \bigcap_{\eps'>\eps}\pi^{n,\mathfrak{a},\mathfrak{c}}_{\eps'}(A), \qquad \eps>0.
\end{equation}
In relation to the definition \eqref{eq:muac}, note that
\begin{equation} \label{eq:swap}
(A^{\mathfrak{a},\mathfrak{c}}_{n,k})^* = (A^*)^{\mathfrak{c},\mathfrak{a}}_{n,k}.
\end{equation}

\begin{theorem}\label{thm:Mpi}
Let $n\in\N$, and suppose that $w_j\geq 0$, for $j=1,\ldots,n$, with at least one $w_j$ non-zero, and that $\mathfrak{a},\mathfrak{c}\in \ell^\infty(\Z,L(X))$, with $\|\mathfrak{a}\|_\infty \leq \|\alpha\|_\infty$ and $\|\mathfrak{c}\|_\infty \leq \|\gamma\|_\infty$,
Then, where  $\eps'_n$ is given by \eqref{eq:epsn2},
\begin{equation} \label{eq:pgen}
\speps A\ \subset\ \pi^{n,\mathfrak{a},\mathfrak{c}}_{\eps+\eps'_n}(A), \quad \eps>0.  \qquad\textrm{Moreover,}
\qquad \Speps A \ \subset\ \Pi^{n,\mathfrak{a},\mathfrak{c}}_{\eps+\eps'_n}(A), \quad \eps\geq 0,
\end{equation}
provided $X$ has the Globevnik property.
\end{theorem}
\begin{proof}
To see that the first inclusion in \eqref{eq:pgen} holds, suppose that $\lambda\in\speps A$. Then, by \eqref{eq:spepspert},  either $\nu(A-\lambda I)<\eps$ or  $\nu(A^*-\lambda I)<\eps$. By applications of Corollary \ref{cor:nb2}, it follows, in the case $\nu(A-\lambda I)<\eps$, that $\nu(A^{\mathfrak{a},\mathfrak{c}}_{n,k}-\lambda I_n)<\eps+\eps_n$, for some $k\in \Z$, and, in the case $\nu(A^*-\lambda I)<\eps$, recalling \eqref{eq:swap}, that  $\nu((A^{\mathfrak{a},\mathfrak{c}}_{n,k})^*-\lambda I_n)=\nu((A^*)^{\mathfrak{c},\mathfrak{a}}_{n,k}-\lambda I_n)<\eps+\eps_n$, for some $k\in \Z$. Thus, recalling the notation \eqref{eq:invNorm},  $\mu(A^{\mathfrak{a},\mathfrak{c}}_{n,k}-\lambda I_n)<\eps+\eps_n$, for some $k\in \Z$, so that $\lambda \in \specn_{\eps+\eps_n} A^{\mathfrak{a},\mathfrak{c}}_{n,k}$ by \eqref{eq:spepsnu}, so that $\lambda \in \pi^{n,\mathfrak{a},\mathfrak{c}}_{\eps+\eps_n}(A)$. The second inclusion in \eqref{eq:pgen} follows from the first inclusion by taking intersections, noting \eqref{eq:oc} and \eqref{eq:piPi}.
\end{proof}

In particular, Theorem \ref{thm:Mpi} holds with $\mathfrak{a}$ and $\mathfrak{c}$ given by \eqref{eq:special}, giving the following corollary.

\begin{corollary}\label{weighted_norm_two_refined_circulant}
Let $n\in\N$, and suppose that $w_j\geq 0$, for $j=1,\ldots,n$, with at least one $w_j$ non-zero, and that $t\in \T$.
Then the second inclusion of \eqref{incl:met1*} holds with $\eps'_n$ given by \eqref{eq:epsn2}. If $X$ has the Globevnik property, then also the first  inclusion of \eqref{incl:met1*} holds with $\eps'_n$ given by \eqref{eq:epsn2}.
\end{corollary}


As we did above for the $\tau$ method, we will now minimize the penalty term $\eps'_n$ as a
function, \eqref{eq:epsn2}, of the weight vector
$w=(w_1,...,w_n)^T$. In the cases $n=1$ and $2$ it is easy to see that $\eps_n$ is minimized by taking $w_1=\cdots=w_n=1$, giving $\eps_1=2(\|\alpha\|_\infty+\|\gamma\|_\infty)$ and $\eps_2=\sqrt{2}(\|\alpha\|_\infty+\|\gamma\|_\infty)$. For $n\geq 3$, from the definitions of $S_{n}$ and $T_{n}$ in
Theorem \ref{weighted_norm_two_refined_circulant}, we know that
$S_{n} = \|w\|_{2}^{2}$ and $T_{n} = \|B_{n}w\|_{2}^{2}$, where
\[
B_{n}=\begin{pmatrix}
1   & -1     &            &          &   \\
    & 1      & -1         &          &   \\
    &        & \ddots     & \ddots   &  \\
    &        &            & 1        & -1   \\
1   &        &            &          & 1
\end{pmatrix}
\quad\textrm{so that}\quad B_{n}^TB_{n}=\begin{pmatrix}
2   & -1       &           &    & 1  \\
 -1 & 2        &    -1     &    &   \\
    & \ddots   & \ddots    & \ddots   &  \\
    &          & -1        &2    & -1  \\
1   &          &           &-1   & 2
\end{pmatrix}.
\]
Let us seek
\begin{equation} \label{eq:infepsn}
\inf_{w\in \R^n, \,w\neq 0}\sqrt{\frac{T_{n}}{S_{n}}}\ =\ \inf_{w\in \R^n, \,w\neq
0}\sqrt{\frac{\|B_{n}w\|_{2}^2}{\|w\|_{2}^2}}\ =\ \min_{w\in \R^n, \,\|w\|=1}\sqrt{\frac{w^{T}B_n^{T}B_nw}{w^{T}w}}\ =\ \sqrt{\lambda_{\sf
min}(B^{T}_{n}B_{n})},
\end{equation}
noting that a vector $w$ achieves this minimum if and only if it is an eigenvector of $B_n^TB_n$ corresponding to the eigenvalue $\lambda_{\sf
min}(B^{T}_{n}B_{n})$. It is clear that this minimum is strictly positive and that (taking $w_1=w_2=\cdots = w_n$) it is no larger than $2/\sqrt{n}$. Further, $\lambda\in (0,4)$ is an eigenvalue of $B_n^TB_n$ with
eigenvector $v=(v_1,...,v_n)^T$ if and only if
\begin{equation} \label{eq:lambda1}
\lambda \ =\ 2(1-\cos\theta)\ =\ 4\sin^{2}\frac{\theta}{2}, \quad \mbox{for some } \theta\in (0,\pi),
\end{equation}
and
\begin{eqnarray}
\label{Feb2009_4} v_{n}+(2-\lambda)v_{1}-v_{2}&=&0,\\
\label{Feb2009_3} -v_{i+1}+(2-\lambda)v_{i}-v_{i-1}&=&0,\qquad
\textrm{ for }i= 2,\ldots,n-1,\\
\label{Feb2009_2} -v_{n-1}+(2-\lambda)v_{n}+v_{1}&=&0.
\end{eqnarray}
Note that if $(\lambda,v)$ satisfies \eqref{Feb2009_4}-\eqref{Feb2009_2} then so does $(\lambda,\tilde v)$, where $\tilde v:= (v_n,\ldots,v_1)$. Further, $(\lambda,v^\dag)$, where $v^\dag:=v+\tilde v$ satisfies $v^\dag_j= v^\dag_{n+1-j}$, $j=1,...,n$.  Thus, to compute $\lambda$ and a corresponding eigenvector $v$ from \eqref{Feb2009_4}-\eqref{Feb2009_2}, we may assume that
\begin{equation} \label{eq:symm}
v_j= v_{n+1-j}, \quad j=1,...,n.
\end{equation}
 As discussed around \eqref{March2009_3}, (\ref{Feb2009_3}) holds if and only if $v_j$ is a linear combination of $\cos((j-1)\theta)$ and $\sin((j-1)\theta)$. To within multiplication by a constant, the linear combination that satisfies \eqref{eq:symm} is
\begin{equation} \label{eq:vsymm}
v_{j}\ =\ \sin((j-1)\theta) + \sin((n-j)\theta),\qquad
j=1,2,\ldots,n,
\end{equation}
and this satisfies \eqref{Feb2009_4} and \eqref{Feb2009_2} if and only if
$F_\pi(\theta)=0$ where, for $t\in \R$,
\begin{eqnarray*}
F_\pi(t) &:=& \sin((n-1)t) + 2\cos(t)\sin((n-1)t)-\sin(t)-\sin((n-2)t)\\
& =& \sin((n-1)t)+\sin(nt)-\sin(t) = 2\sin((n-2)t/2)\cos(nt/2)+\sin(nt).
\end{eqnarray*}
From this last expression for $F_\pi(t)$ it is clear that $F_\pi(\pi/n)=0$ and that $F_\pi(t)>0$ for $0<t<\pi/n$. Thus $\lambda_{\sf min}(B^{T}_{n}B_{n})$ is given by \eqref{eq:lambda1} with $\theta=\frac{\pi}{n}$. Since the corresponding eigenvector given by \eqref{eq:vsymm} satisfies $v_j>0$, $j=1,\ldots, n$, the following corollary  follows from
Theorem \ref{weighted_norm_two_refined_circulant} and
\eqref{eq:infepsn} and the observations above about the cases $n=1$ and $2$.

\begin{corollary}\label{fn_circulant_corollary}
The conclusions of Theorem \ref{thm:Mpi} hold with 
\begin{equation} \label{eq:epsn_method1*}
\eps'_n\ :=\
2(\|\alpha\|_{\infty}+\|\gamma\|_{\infty})\sin\frac{\pi}{2n}.
\end{equation}
In particular,  for $n\in \N$, the second inclusion of \eqref{incl:met1*}
holds with $\eps'_n$ given by \eqref{eq:epsn_method1*}, as does the first inclusion of \eqref{incl:met1*} if $X$ has Globevnik's property.
\end{corollary}

Again, as in Corollary~\ref{minimum_weighted_norm_two}, the example
of the shift operator (see  Example \ref{ex:shift} and \S\ref{sec:shift} below) shows that this value of $\eps'_n$ is the best
possible.

\section{The $\tau_1$ method: one-sided truncations} \label{sec:tau1}
\subsection{Proof of the inclusions \eqref{incl:met2}} \label{sec:tau1a}
In this section, which follows the patterns of \S\ref{sec:tau} and \S\ref{sec:piproof}, we establish the inclusions \eqref{incl:met2}. 

\begin{proposition} \label{prop:gam_main}
Let $\eps>0$ and $n\in \N$, suppose that $w_j\in \R$, for $j=1,\ldots,n$, with at least one $w_j$ non-zero, and that $A\in L(E)$ is tridiagonal and $\|Ax\|\leq \eps$, for some $x\in E$ with $\|x\|=1$. Then, for some $k\in \Z$, where $A^+_{n,k}$ is defined by \eqref{eq:An+}, it holds that
\begin{equation} \label{eq:gam_main}
\nu(A^+_{n,k}) \leq \eps+\eps''_n,
\end{equation} 
where
\begin{equation} \label{eq:eps''}
\eps''_n\ :=\ \left(
\left\|\alpha\right\|_{\infty}+\left\|\gamma\right\|_{\infty}\right)\sqrt{\displaystyle\frac{T_{n}}{S_{n}}}
\end{equation}
with
\[
S_{n}\ =\ \sum_{j=1}^{n}w_{j}^{2}\qquad\textrm{and}\qquad T_{n}\ =\
w_{1}^{2}+w_{n}^{2}+\sum_{j=1}^{n-1}\left(w_{i+1}-w_{i}\right)^{2}.
\]
\end{proposition}
\begin{proof}
Let $y=Ax$, so $\|y\|\leq \eps$. We argue as in the proof of Proposition \ref{prop:tau_main}, using the notation $E_j$, defined in \eqref{E1},  for $j=-1,0,\ldots,n,n+1$, with $w_{-1}:=w_0:=w_{n+1}:= w_{n+2}:= 0$.
Where
$\tilde{x}_{n,k}:=(w_{1}x_{k+1},w_{2}x_{k+2},\ldots,w_{n}x_{k+n})^{T}$, put
$a_{k}:=\|A^{+}_{n,k}\tilde x_{n,k}\|$ and $b_k := \|\tilde x_{n,k}\|$, for
$k\in\Z$.
As in the proof of Proposition \ref{prop:tau_main}, we have that $\sum_{k\in \Z} b_k^2 = S_n$. Further, cf.~\eqref{eq:aj} and \eqref{eq:ss},
\begin{eqnarray*}
a_k^2&=&  \|w_1\gamma_{k+1}x_{k+1}\|_X^2+ \|w_n\alpha_{k+n}x_{k+n}\|_X^2\\
& & ~\rule{5mm}{0pt} + \sum_{j=1}^{n}\|\alpha_{j+k-1}w_{j-1}x_{j+k-1}\ +\ \beta_{j+k}w_jx_{j+k}\ +\ \gamma_{j+k+1}w_{j+1}x_{j+k+1}\|_X^{2}
\\
&\leq& \sum_{j=0}^{n+1}\left(w_j\|y_{j+k}\|_X +
E_{j-1}\|\alpha\|_\infty \|x_{j+k-1}\|_X+E_j\|\gamma\|_\infty \|x_{j+k+1}\|_X\right)^{2} .
\end{eqnarray*}
So, for all $\theta >0$ and $\phi>0$, using Lemma
\ref{Simon_theta_inequality} twice,
\begin{eqnarray*}
a_{k}^{2}
&\leq&\sum_{j=0}^{n+1}\Big[(1+\theta)\big( w_j\|y_{j+k}\|_X\big)^{2}+(1+\theta^{-1})\cdot \\
& & \hspace{1ex} \Big( (1+\phi) \big|\alpha_{i-1}\big|^{2} \big(E_{j-1}\|\alpha\|_\infty \|x_{j+k-1}\|_X\big)^{2} + (1+\phi^{-1})\big(E_j\|\gamma\|_\infty \|x_{j+k+1}\|_X\big)^2\Big)\Big]
\end{eqnarray*}
Thus, for all $\theta,\phi>0$,
\[
\sum_{k\in\Z}a_{k}^{2}\ \leq\
(1+\theta)S_n\|y\|^{2}
+(1+\theta^{-1})\left[(1+\phi)T_n\|\alpha\|_{\infty}^{2}
+(1+\phi^{-1})T_n\|\gamma\|_{\infty}^{2}\right].
\]
Arguing as in the proof of Proposition \ref{prop:tau_main} (cf.~\eqref{SB} and \eqref{sb}) it follows that
\[
\sum_{k\in\Z}a_{k}^{2}
\leq\left(\eps+\eps''_n\right)^{2}\sum_{k\in\Z}b_{k}^{2},
\]
so that $a_{k}\leq (\eps+\eps''_n)b_{k}$, for some $k\in\Z$ with $b_k\neq 0$, which implies that $\nu(A_{n,k}^+)\leq \eps+\eps''_n$.
\end{proof}

The above result has the following straightforward corollary (cf.~Corollary \ref{cor:nubound}).

\begin{corollary} \label{cor:Gbound}
Let $n\in\N$, suppose that $w_j\in\R$, for $j=1,...,n$, with
at least one $w_j$ non-zero, and that $A\in L(E)$ is tridiagonal.  Then
\begin{equation} \label{eq:infonG}
\inf_{k\in \Z} \nu(A^+_{n,k})\ \leq\ \nu(A)+\eps''_n,
\end{equation} 
where $\eps''_n$ is given by \eqref{eq:eps''}.
\end{corollary}

The  proof of the following main theorem is  very similar to those of Theorems
\ref{weighted_norm_two_refined} and
\ref{thm:Mpi}.
\begin{theorem}\label{thm_method2}
Suppose that $n\in\N$ and $w_j\in \R$, for $j=1,\ldots,n$, with at least one $w_j$ non-zero. Then \eqref{incl:met2} holds with $\eps''_n$ given by \eqref{eq:eps''}.
\end{theorem}
\begin{proof} We have justified already, below \eqref{eq:epsn_method2}, the inclusions from the left in \eqref{incl:met2}.
To see that the second inclusion from the right in \eqref{incl:met2} holds, 
let
$\lambda\in\speps (A)$. Then either $\nu(A-\lambda I) <\eps$ or $\nu(A^*-\lambda I) < \eps$. By Corollary \ref{cor:Gbound} it follows that either $\nu(A^+_{n,k}-\lambda I_n^+) <\eps+\eps''_n$, for some $k\in \Z$, or $\nu((A^*)^+_{n,k}-\lambda I_n^+) <\eps+\eps''_n$, for some $k\in \Z$. Thus, by \eqref{eq:munmat}, $\mu_n(A-\lambda I) < \eps+\eps_n''$, so that $\lambda \in \gamma^n_{\eps+\eps''_n}(A)$. The first inclusion from the right in \eqref{incl:met2} follows by taking intersections, recalling \eqref{eq:oc}.
\end{proof}

Following the pattern of \S\ref{sec:mini} and \S\ref{sec:piproof}, we now minimise $\eps''_n$ as a function of the weight vector $w=(w_1,\ldots,w_n)^T\in \R^n$ over $\R^n\setminus \{0\}$, or equivalently over $\{w\in \R^n:\|w\|=1\}$. In order to minimise
$\eps''_n$, given by \eqref{eq:eps''}, 
we need to minimise $T_{n}/S_{n}$, where
\begin{eqnarray*}
S_{n}\ =\ \|w\|^2 \quad \mbox{and} \quad T_{n}\ =\ w_{1}^{2}+(w_{2}-w_{1})^{2}+\cdots+(w_{n}-w_{n-1})^{2}+w_{n}^{2}\ =\ \|Bw\|^2,
\end{eqnarray*}
with
\[
B=\begin{pmatrix}
1   &         &            &          &   \\
-1  & 1       &            &          &   \\
    &  \ddots & \ddots     &          &  \\
    &        &           -1& 1        &    \\
    &        &            &         -1 & 1 \\
    &        &             &          & -1
\end{pmatrix}_{(n+1)\times n}
\textrm{so that}\quad B^{T}B=\begin{pmatrix}
2   &  -1       &            &          &   \\
-1  &2        &  -1          &          &   \\
    &  \ddots & \ddots     & \ddots         &  \\
    &        &           -1& 2        &  -1  \\
    &        &            &  -1        & 2
 \end{pmatrix}_{n\times n}.
\]
Clearly,
\[
\inf_{\|w\|\neq 0}\frac{T_{n}}{S_{n}}\ =\ \inf_{\|w\|\neq
0}\frac{\|Bw\|^{2}}{\|w\|^{2}}\ =\ \inf_{\|w\|\neq
0}\frac{w^T B^T Bw}{w^Tw}\ =\ \lambda_{\sf min}(B^{T}B),
\]
the smallest eigenvalue of $B^T B$. One could now compute the eigenvalues of $B^{T}B$ as in the previous
sections \S\ref{sec:mini} and \S\ref{sec:piproof}. In this simple case ($-B^TB$ is the discrete Laplacian),
this is a standard result (e.g. \cite{BoeGru}). We have that
\begin{equation} \label{eq:spec_laplace}
\spec(B^TB)\ =\
\left\{\lambda_j:=2-2\cos\frac{j\pi}{(n+1)}=4\sin^{2}\frac{j\pi}{2(n+1)}\
:\ j\in\{1,\ldots, n\}\right\},
\end{equation}
so that $\lambda_{\sf min}(B^{T}B)=\lambda_1$. Hence the minimal
value for $\eps''_n$ is
\[
\eps''_n\ :=\
(\|\alpha\|_{\infty}+\|\gamma\|_{\infty})\sqrt{\frac{T_{n}}{S_{n}}}\
=\ (\|\alpha\|_{\infty}+\|\gamma\|_{\infty})\sqrt{\lambda_{1}}\ =\
2(\|\alpha\|_{\infty}+\|\gamma\|_{\infty})\sin\frac{\pi}{2(n+1)}.
\]
This minimum is realised \cite{BoeGru} by the choice 
$w=(\sin\frac{j\pi}{n+1})_{j=1}^{n}$ for the weight vector in
Theorem \ref{thm_method2}.

\begin{corollary}\label{cor:one_sided_truncation}
For all $n\in\N$, the inclusions \eqref{incl:met2} hold
with
\begin{equation}\label{eq:one_sided_truncation}
\eps''_n\ :=\
2(\|\alpha\|_{\infty}+\|\gamma\|_{\infty})\sin\frac{\pi}{2(n+1)}.
\end{equation}
\end{corollary}
The example of the shift operator  (Example \ref{ex:shift} and \S\ref{sec:shift}) shows that the above formula is sharp; \eqref{incl:met2} does not hold for all tridiagonal $A\in L(E)$ if  $\eps''_n$ is any smaller than the above value.

\subsection{Does an analogue of Proposition \ref{prop:same} hold for the $\tau_1$ method?} \label{sec:rect}
For operators $B\in L(X)$  (in particular when $X=\C^n$ and  $B$ is a finite square matrix), the resolvent norm
\[
\lambda\ \mapsto\ \|(B-\lambda I)^{-1}\|
\]
on $\rho(B) := \C\setminus \Spec B$ is a subharmonic function, subject to a maximum principle; it  cannot have a local maximum in $\rho(B)$. Likewise, its reciprocal, which (recall \eqref{eq:invNorm} and \eqref{eq:lnPro2}) is the Lipschitz continuous function
\[
f_B:\ \lambda\ \mapsto\ \mu(B-\lambda I),
\]
cannot have a local minimum on $\rho(B)$, indeed cannot have a local minimum on $\C$, except that it takes the minimum value zero on $\Spec B$.

The proof of the identity $\overline{\sigma^n_\eps(A)}=\Sigma_\eps^n(A)=\widehat \Sigma_\eps^n(A)$ for the $\tau$ method in Proposition \ref{prop:same}  rests indirectly on this property, and the stronger result that $f_B$ cannot be locally constant on any open subset of $\rho(B)$, which holds if $X$ has the Globevnik property. The same is true for the proof of the corresponding property \eqref{eq:piAAlt2} for the $\pi$ method. But the corresponding relationship for the $\tau_1$ method, relating the two $\tau_1$ sets defined in \eqref{eq:gamdef}, that $\overline{\gamma^n_\eps(A)}=\Gamma_\eps^n(A)$, does not hold for every $n\in \N$, $\eps>0$, and every tridiagonal $A$, as the example we give below shows.

The issue is that the appropriate version of the mapping $f_B$ can have local (non-zero) minima, in the case that $B$ is a finite rectangular matrix, as discussed in \cite[{\S}X.46]{TrefEmbBook} and in \cite{TorgePhD}. An example that illustrates this is the 
 $4\times 2$ matrix
$$
B:=\begin{pmatrix}\delta&0\\0&0\\1&1\\0&\delta\end{pmatrix}\quad\text{with}\qquad \delta\in(0,\textstyle \frac 12).
$$
The relevant version of the mapping $f_B$,  namely $f_B(\lambda):=\nu(B-\lambda I_2^+)$, where $I_2^+$ is as defined above \eqref{eq:munmat} (with $X=\C$),
has global minima (over all $\lambda\in\mathbb C$) at $0$ and $1$ (see the calculations in Lemma \ref{lem:B2dominant} below), with the values
$$
f(0)=f(1)=\nu(B)=\delta>0.
$$

To obtain our counterexample to $\overline{\gamma^n_\eps(A)}=\Gamma_\eps^n(A)$ we build a tridiagonal bi-infinite matrix around $B$, given by\footnote{This is an example of a so-called {\em paired Laurent operator}, a class for which the spectrum can be computed explicitly. See, e.g., \cite[\S 4.4.1]{HaRoSi2}, \cite[\S3.7.3]{LiBook}.}
\begin{equation} \label{eq:exA}
A:=\left(
\begin{array}{cccccccc}
\ddots&\ddots & & & \\, 
\ddots&0&\delta & & \\
      &1&0&\cellcolor{black!10}\delta&\cellcolor{black!10} \\
      & &1&\cellcolor{black!10}0&\cellcolor{black!10} \\
      & & &\cellcolor{black!10}1&\cellcolor{black!10}1\\
      & & &\cellcolor{black!10} &\cellcolor{black!10}\delta&1\\
      & & & &  &\delta &1\\
      & & & &  & &\ddots &\ddots
\end{array}
\right).
\end{equation}
The sets $\gamma_\eps^n(A)$ and $\Gamma_\eps^n(A)$ are defined by \eqref{eq:gamdef} in terms of $\mu_n(A-\lambda I)$, which in turn is expressed in terms of lower norms of rectangular matrices in \eqref{eq:munmat}. For the case $n=2$ there are only three distinct matrices $A^+_{n,k}$ as $k$ ranges over $\Z$, namely
\[
B_1:=\begin{pmatrix}\delta&\\0&\delta\\1&0\\&1\end{pmatrix},\quad
B_2:=B = \begin{pmatrix}\delta&\\0&0\\1&1\\&\delta\end{pmatrix},\quad\text{and}\quad
B_3:=\begin{pmatrix}0&\\1&0\\\delta&1\\&\delta\end{pmatrix},
\]
and only  five different matrices $(A^*)^+_{n,k}$ with $n=2$ and $k\in\mathbb Z$, namely
\[
C_1=\begin{pmatrix}1&\\0&1\\\delta&0\\&\delta\end{pmatrix},\quad
C_2=\begin{pmatrix}1&\\0&1\\\delta&0\\&0\end{pmatrix},\quad
C_3=\begin{pmatrix}1&\\0&1\\0&1\\&0\end{pmatrix},\quad
C_4=\begin{pmatrix}1&\\1&\delta\\0&1\\&0\end{pmatrix},\quad\text{and}\quad
C_5=\begin{pmatrix}\delta&\\1&\delta\\0&1\\&0\end{pmatrix}.
\]
Putting, for $j\in\{1,2,3\}$ and $k\in\{1,\dots,5\}$,
\[
f_j(\lambda)\ :=\ \nu(B_j-\lambda I_2^+)\qquad\text{and}\qquad g_k(\lambda)\ :=\ \nu(C_k-\lambda I_2^+)
,\qquad \lambda\in\C,
\]
we have
\begin{equation} \label{eq:def_h}
\mu_n(A-\lambda I)\ =\ \min\left\{\,f_j(\lambda)\,,\,g_k(\lambda): j\in \{1,2,3\}, \, k\in \{1,\ldots,5\}\right\}\ =:\ h(\lambda).
\end{equation}
\begin{lemma} \label{lem:B2dominant}
Let $\delta\in (0,\frac 12)$. Then:
\begin{enumerate}[label=\bf\alph*),topsep=0pt]
\item $f_2(\lambda)\ge \delta$ for all $\lambda\in\C$;
\item $0<\delta=f_2(0)<\frac 12< f_3(0)=g_5(0)< g_2(0)=g_3(0)=g_4(0)=1 < f_1(0)=g_1(0)$.
\end{enumerate}
As a consequence, the function $f_2$, and so also the function $h(\lambda)=\mu_n(A-\lambda I)$, have positive local minima at zero, namely
\[
\mu_n(A)\ =\ h(0)\ =\ f_2(0)\ =\ \delta.
\]
\end{lemma}
\begin{proof}

{\bf a) } For all $x={x_1\choose x_2}\in\C^2$, we have
\[
\|(B_2-\lambda I)x\|
\ =\ \left\|\begin{pmatrix}\delta x_1\\-\lambda x_1\\x_1+(1-\lambda)x_2\\\delta x_2\end{pmatrix}\right\|
\ \ge\ \left\|\begin{pmatrix}\delta x_1\\0\\0\\\delta x_2\end{pmatrix}\right\|\ =\ \delta\|x\|,
\]
so that $f_2(\lambda)=\nu(B_2-\lambda I)\ge\delta$ for all $\lambda\in\C$.

{\bf b) } We start with the computation of $f_j(0)$ for $j\in\{1,2,3\}$. Firstly,
$f_1(0)=\nu(B_1)=\sqrt{1+\delta^2}>1$, since $\|B_1x\|^2=(1+\delta^2)\|x\|^2$ for all $x\in\C^2$. Secondly,
$f_2(0)=\nu(B_2)\ge\delta$ holds by a). To see equality, observe that $\|B_2x\|=\delta\|x\|$ for $x={1\choose -1}$. Thirdly,
$f_3(0)=\nu(B_3)\in [1-\delta,1)$, since
$\|B_3x-\delta(0,0,x_1,x_2)^\top\|=\|x\|$, so that $\|B_3x\|\ge \|x\|-\delta\|x\|$ for all $x={x_1\choose x_2}\in\C^2$, whence $\nu(B_3)\ge 1-\delta$. Computation of $\|B_3x\|$ with $x={1\choose -\delta}$ shows that $\nu(B_3)<1$.

Carrying on to the $g_k(0)$ with $k\in\{1,\dots,5\}$, again writing $x\in\C^2$ as $x={x_1\choose x_2}$, we get:
$g_2(0)=\nu(C_2)=1$, since $\|C_2x\|\ge\|(x_1,x_2,0,0)^\top\|=\|x\|$ for all $x\in\C^2$, with equality for $x={0\choose 1}$;
$g_3(0)=\nu(C_3)=1$, since $\|C_3x\|\ge\|(x_1,x_2,0,0)^\top\|=\|x\|$ for all $x\in\C^2$, with equality for $x={1\choose 0}$;
$g_4(0)=\nu(C_4)=1$, since $\|C_4x\|\ge\|(x_1,0,x_2,0)^\top\|=\|x\|$ for all $x\in\C^2$, with equality for $x={\delta\choose -1}$.
Finally, denoting the flip isometry $(x_1,\dots,x_m)\mapsto (x_m,\dots, x_1)$ on $\C^m$ by $J_m$, we have $C_1=J_4B_1J_2$ and $C_5=J_4B_3J_2$, showing that $g_1(0)=f_1(0)$ and $g_5(0)=f_3(0)$.
\end{proof}

By Lemma \ref{lem:B2dominant}, the tridiagonal matrix \eqref{eq:exA} has the properties that, for $n=2$ and $0<\delta<1/2$,
$0\in \Gamma_\delta^n(A)=\{\lambda\in\C:\mu_n(A-\lambda I)\le\delta\}$ but $\gamma_\delta^n(A)=\{\lambda\in\C:\mu_n(A-\lambda I)<\delta\}$ does not intersect $\eps\D$, for some $\eps>0$. Thus $0\not \in \overline{\gamma_\delta^n(A)}$, and 
\begin{equation} \label{eq:neq}
\overline{\gamma_\delta^n(A)} \subsetneq  \Gamma^n_{\delta}(A).
\end{equation}

\section{Proof of Theorem \ref{thm:fin} and a band-dominated generalisation} \label{sec:thmgen}

The following result is a generalisation of Theorem \ref{thm:fin} to the band-dominated case. It reduces to Theorem \ref{thm:fin} if $A$ is tridiagonal (in which case $A^{(n)}=A$, for each $n$, so that $w_n=1$ and also $A_n=A$ and $\delta_n=0$), provided also $X$ is finite-dimensional or a Hilbert space (in which case $\eta_n=0$). Note that the requirement that $\{A_{ij}:i,j\in \Z\}\subset L(X)$ be relatively compact is satisfied automatically if $X$ is finite-dimensional.
\begin{theorem} \label{thm:fin2}
Suppose that $A\in BDO(E)$ and that $\{a_{i,j}:i,j\in \Z\}\subset L(X)$ is relatively compact. Let $A^{(n)}\in BO(E)$ be defined, for $n\in \N$, by $A^{(n)} := A$ if $A$ if tridiagonal, otherwise by \eqref{eq:Andef} for some $\mathfrak{p}\geq 0$ such that
$\delta_n:= \|A-A^{(n)}\|\to 0$ as $n\to\infty$, and let $A_n:= \cI_{w_n}A^{(n)}\cI^{-1}_{w_n}$, where $w_n$ is the band-width of $A^{(n)}$ and $\cI_b$ is defined as above \eqref{eq:norm}. In the case that $X$ is finite-dimensional or a Hilbert space set $\eta_n:=0$, for $n\in \N$, otherwise let $(\eta_n)_{n\in \N}$ denote any positive null sequence. Then, for every $n\in \N$ there exists a finite set $K_n\subset \Z$ such that:
\begin{equation} \label{eq:comp2}
\begin{aligned}
& \forall k\in \Z \;\; \exists j\in K_n \mbox{ such that }\\
& \|(A_n)^+_{n,k}-(A_n)^+_{n,j}\|\leq 1/n \mbox{ and } \|((A_n)^*)^+_{n,k}-((A_n)^*)^+_{n,j}\|\leq 1/n.
\end{aligned}
\end{equation}
Further, for $\eps\geq 0$ and $n\in \N$, using the notations \eqref{eq:findef} and \eqref{eq:epsn''2},
\begin{equation} \label{eq:genn}
\Gamma^{n,\mathrm{fin}}_{\eps}(A_n)\ \subset\ \Specn_{\eps+\delta_n+\eta_n} A \quad \mbox{and} \quad \Speps A\ \subset \Gamma^{n,\mathrm{fin}}_{\eps+\delta_n+\eta_n+\eps''_n(A_n) + 1/n}(A_n),
\end{equation}
where
\begin{equation} \label{eq:findef2}
\Gamma^{n,\mathrm{fin}}_{\eps}(A_n) =  \left\{\lambda\in \C : \mu_n^{\mathrm{fin}}(A_n-\lambda I)\leq \eps\right\} 
\end{equation}
with
\begin{equation} \label{eq:mufin2}
\mu_n^{\mathrm{fin}}(A_n-\lambda I)= \min_{j\in K_n} \min\Big(\nu\left((A_n)^+_{n,j}-\lambda I^+_n\right),\,\nu\left(((A_n)^*)^+_{n,j}-\lambda I^+_n\right)\Big), \quad \mbox{for }\; \lambda\in \C.
\end{equation}
Moreover, for $\eps\geq 0$, $\Gamma^{n,\mathrm{fin}}_{\eps+2\delta_n+2\eta_n+\eps''_n(A_n) + 1/n}(A_n)\Hto \Speps A$ as $n\to\infty$, in particular\\ 
$\Gamma^{n,\mathrm{fin}}_{\eps''_n(A_n) + 2\delta_n+2\eta_n+1/n}(A_n)\Hto \Spec A$.
\end{theorem}
\begin{proof} Fix $n\in \N$. To see \eqref{eq:comp2} note first that the relative compactness of the matrix entries of $A$ implies, by the definition of $A_n$, relative compactness also of the entries of $A_n$. Further, since the set of matrix entries of $A_n$ is relatively compact, so that every sequence has a convergent subsequence, the same is true for 
\begin{equation} \label{eq:Sdef}
\cS_n := \{((A_n)^+_{n,k},((A_n)^*)^+_{n,k}):k\in \Z\}\subset ((X^{w_n})^{(n+2)\times n})^2. 
\end{equation}
Thus there exists $K_n\subset \Z$ such that \eqref{eq:comp2} holds since relatively compact sets are totally bounded. 

To see \eqref{eq:genn}, note that it follows from \eqref{eq:finGG}, \eqref{incl:met2}, and \eqref{eq:PseudInc} (if $\eta_n>0$) or \eqref{eq:PseudInc2} (if $\eta_n=0$) that $\Gamma^{n,\mathrm{fin}}_{\eps}(A_n) \subset \Speps A_n \subset \Specn_{\eps+\delta_n+\eta_n} A$. Again applying \eqref{eq:PseudInc}/\eqref{eq:PseudInc2}, 
$$
\Specn_{\eps} A\ \subset\ \Specn_{\eps+\delta_n+\eta_n} A_n\ \subset\ \Gamma^{n}_{\eps+\delta_n+\eta_n+\eps''_n(A_n)}(A_n)\ \subset\ \Gamma^{n,\mathrm{fin}}_{\eps+\delta_n+\eta_n+\eps''_n(A_n)+1/n}(A_n),
$$
by \eqref{incl:met2} and \eqref{eq:finGG}.
Thus, for $n\in \N$ and $\eps\geq 0$,
$$
\Speps A\ \subset \Gamma^{n,\mathrm{fin}}_{\eps+\delta_n+\eta_n+\eps''_n(A_n) + 1/n}(A_n)\ \subset\ \Specn_{\eps+2\delta_n+2\eta_n+\eps''_n(A_n)+1/n} A.
$$
Arguing as above Theorem \ref{thm:bdo}, we have that $\eps''_n(A_n) \to 0$ as $n\to\infty$. This implies that  $\Gamma^{n,\mathrm{fin}}_{\eps+2\delta_n+2\eta_n+\eps''_n(A_n) + 1/n}(A_n)$ $\Hto \Speps A$ as $n\to\infty$, by \eqref{eq:HDconv}.
\end{proof}

\begin{remark} \label{rem:remalt} In the theorem above we note that there exists a finite set $K_n\subset \Z$ such that \eqref{eq:comp2} holds. An equivalent statement is to say that there exists a finite set $S_n\subset \cS_n$, where $\cS_n$ is defined by \eqref{eq:Sdef}, such that 
\begin{equation} \label{eq:comp3}
\forall k\in \Z \;\; \exists (B_1,B_2)\in S_n \mbox{ such that }\|(A_n)^+_{n,k}-B_1\|\leq 1/n \mbox{ and } \|((A_n)^*)^+_{n,k}-B_2\|\leq 1/n.
\end{equation}
The above theorem thus remains true if we replace \eqref{eq:comp2} with \eqref{eq:comp3} and replace \eqref{eq:mufin2} with the formula
\begin{equation} \label{eq:mufin3}
\mu_n^{\mathrm{fin}}(A_n-\lambda I)= \min_{(B_1,B_2)\in S_n} \min\Big(\nu\left(B_1-\lambda I^+_n\right),\,\nu\left(B_2-\lambda I^+_n\right)\Big), \quad \mbox{for }\; \lambda\in \C.
\end{equation}
Importantly, with these changes, the proof still holds if, instead of requiring that $S_n\subset \cS_n$, we make the weaker requirement  that $S_n\subset \overline{\cS_n}$, the closure of $\cS_n$ in  $((X^{w_n})^{(n+2)\times n})^2$: the elements of $S_n$ need not be in $\cS_n$ only in $\overline{\cS_n}$. This observation is helpful because there are instances where it is much easier to identify $\overline{\cS_n}$ than $\cS_n$.
\end{remark}

Theorem \ref{thm:fin} and the above result are both  based on the $\tau_1$ method. Versions of these results hold also for the $\tau$ and $\pi$ methods, at least in the case that $A$ is tridiagonal, provided (this is a substantial assumption) these methods do not suffer from spectral pollution for the particular operator $A$. Here is a version (cf.~Theorem \ref{thm:fin}) for the case that $A$ is tridiagonal (recall from \S\ref{sec:main} that every banded $A$ can be written in tridiagonal form), written in the way suggested by the above remark. We use in this theorem the notation that, for tridiagonal $B\in L(E)$, 
\begin{equation} \label{eq:gennot}
B^{(M)}_{n,k} := \left\{\begin{array}{ll} B_{n,k}, & \mbox{if $M=\tau$},\\ B^{\per,t}_{n,k}, & \mbox{if $M=\pi$},\end{array}\right.
\end{equation}
where  $t$ in \eqref{eq:gennot} is some fixed value with $t\in \T$, and $B_{n,k}$ and $B_{n,k}^{\per,t}$ are as defined in \eqref{eq:Ank} and \eqref{eq:Ankper}. 

\begin{theorem} \label{thm:fin3}
Suppose that $X$ satisfies Globevnik's property, that $A\in L(E)$ is tridiagonal, and that $\{a_{i,j}:i,j\in \Z\}\subset L(X)$ is relatively compact,  and, where $M=\tau$ or $\pi$, suppose that the $M$ method does not suffer from spectral pollution for $A$ in the sense of Definition \ref{def:specpol}. Let $(\eta_n)_{n\in \N}$ be the null sequence from that definition, and $(\eta'_n)_{n\in \N}$ be any other positive null sequence. Then, for every $n\in \N$ there exists a finite set $S_n\subset \overline{\cS_n^{(M)}}$, where
\begin{equation} \label{eq:cS}
\cS_n^{(M)}\ :=\ \left\{A^{(M)}_{n,k}:k\in \Z\right\}\subset X^{n\times n},
\end{equation}
such that:
\begin{equation} \label{eq:comp4}
\forall k\in \Z \;\; \exists B\in S_n \mbox{ such that }\|A^{(M)}_{n,k}-B\|\leq 1/n. 
\end{equation}
Further, for $\eps\geq 0$ and $n\in \N$,
\begin{equation} \label{eq:genn2}
\Delta^{n,\mathrm{fin},M}_{\eps}(A)\ \subset\ \specn_{\eps+\eta_n+\eta'_n} A \quad \mbox{and} \quad \Speps A\ \subset \Delta^{n,\mathrm{fin},M}_{\eps+\eps^{(M)}_n + 1/n}(A),
\end{equation}
where $\eps^{(M)}_n:= \eps_n$ for $M=\tau$, $:= \eps'_n$ for $M=\pi$, with $\eps_n$ and $\eps'_n$ defined as in \eqref{method0} and \eqref{eq:epsnd}, and
\begin{equation} \label{eq:findef4}
\Delta^{n,\mathrm{fin},M}_{\eps}(A) :=  \left\{\lambda\in \C : \min_{B\in S_n} \mu\left(B-\lambda I_n\right)\leq \eps\right\}.
\end{equation}
Moreover, for $\eps\geq 0$, $\Delta^{n,\mathrm{fin},M}_{\eps+\eps^{(M)}_n + 1/n}(A)\Hto \Speps A$ as $n\to\infty$, in particular $\Delta^{n,\mathrm{fin},M}_{\eps^{(M)}_n+1/n}(A)$ $\Hto \Spec A$.
\end{theorem}
\begin{proof} That \eqref{eq:comp4} holds follows from the relative compactness of $\{a_{i,j}:i,j\in \Z\}\subset L(X)$, arguing as in the proof of Theorem \ref{thm:fin} (and see Remark \ref{rem:remalt}). Now, by  \eqref{eq:sigdef} and \eqref{eq:spepsnu},
$$
\Sigma_\eps^n(A) = \left\{\lambda\in \C: \inf_{B\in S^{(\tau)}} \mu\left(B-\lambda I_n\right)\leq \eps\right\}= \left\{\lambda\in \C: \inf_{B\in \overline{S^{(\tau)}}} \mu\left(B-\lambda I_n\right)\leq \eps\right\},
$$
by \eqref{eq:lnPro2}, so that $\Delta^{n,\mathrm{fin},\tau}_{\eps}(A) \subset \Sigma_\eps^n(A)$. Similarly, $\Delta^{n,\mathrm{fin},\pi}_{\eps}(A) \subset \Pi_\eps^{n,t}(A)$ by \eqref{eq:pidef}, \eqref{eq:spepsnu}, and \eqref{eq:lnPro2}. Since $\Sigma_\eps^n(A)\subset \sigma_{\eps+\eta_n'}^n(A)$ and $\Pi_\eps^n(A)\subset \pi_{\eps+\eta_n'}^n(A)$, the first inclusion in \eqref{eq:genn2} follows from the assumption of the absence of spectral pollution. Note that (cf.~\eqref{eq:finGG}), if \eqref{eq:comp4} holds, then, for $\eps\geq0$,
$$
\Sigma_\eps^n(A)\ \subset \Delta^{n,\mathrm{fin},\tau}_{\eps+1/n}(A) \quad \mbox{and} \quad \Pi_\eps^{n,t}(A)\ \subset \Delta^{n,\mathrm{fin},\pi}_{\eps+1/n}(A),
$$
as a consequence of  \eqref{eq:lnPro2} and since
$$
\sigma_\eps^n(A) = \left\{\lambda\in \C: \inf_{B\in S^{(\tau)}} \mu\left(B-\lambda I_n\right)< \eps\right\},
$$
and the same representation holds for $\pi^{n,t}_\eps(A)$, with $S^{(\tau)}$ replaced by $S^{(\pi)}$.
These inclusions and \eqref{incl:met1} and  \eqref{incl:met1*} imply the second inclusion in \eqref{eq:genn2}. 
Thus, for $n\in \N$ and $\eps\geq 0$,
$$
\Speps A\ \subset \Delta^{n,\mathrm{fin},M}_{\eps+\eps^{(M)}_n + 1/n}(A)\ \subset\ \specn_{\eps+\eps^{(M)}_n+\eta_n+\eta'_n+1/n} A,
$$
so that $\Delta^{n,\mathrm{fin},M}_{\eps+\eps^{(M)}_n+ 1/n}(A)\Hto \Speps A$ as $n\to\infty$, by \eqref{eq:HDconv}.
\end{proof}

\section{The scalar case: computational aspects and the solvability complexity index} \label{sec:scalar}
In this section we focus on the case that, for some $p\in \N$, $X = \C^p$, which we equip with the Euclidean norm (the $2$-norm), so that $X$ and  $E=\ell^2(\Z,X)=\ell^2(\Z,\C^p)$ are Hilbert spaces, and the entries of the matrix representation $[a_{i,j}]$ of $A\in L(E)$ are $p\times p$ complex matrices, i.e., $L(X)=\C^{p\times p}$. This includes the special case that $p=1$, so that $E=\ell^2(\Z)$ and the entries $a_{i,j}$ are just complex numbers\footnote{This is a special case, but recall, as noted at the top of the paper, that, given a separable Hilbert space $Y$, our results for this special case apply to any $A\in L(Y)$ that, with respect to some orthonormal basis $(e_i)_{i\in\Z}\subset Y$, has a matrix representation $a_{i,j}=(Ae_j,e_i)\in \C$ that is banded or  band-dominated.}. 

Our goals in this section are to provide computational details regarding the determination of membership   of our $\tau$, $\pi$, and $\tau_1$ inclusion sets  in this case, to prove Proposition \ref{prop:finite} and an extension of this result to the band-dominated case, and to note implications regarding the solvability complex index. In particular, a key conclusion in this section is that the computational problem of determining the spectrum of a band-dominated operator, given the inputs that we assume, has solvability complexity index SCI$_A=1$, indeed is in the class $\Pi_1^A$ (our notations are those of \cite{SCI,ColHan2023}).


In the above set up, in which $X$ and $E$ are Hilbert spaces, standard formulae are available for computing the lower norms that our methods require. For any Hilbert spaces $Y$ and $Z$ and any $B\in L(Y,Z)$ we have (see, e.g., \cite[\S1.3]{Kurbatov} or \cite[\S2.4]{LiBook})
\begin{equation} \label{eq:lowernorm1}
(\nu(B))^2 = \nu(B'B) = \min \Spec(B'B) = \left(s_{\min}(B)\right)^2,
\end{equation}
where $B'\in L(Z,Y)$ denotes the Hilbert space adjoint of $B$ and its smallest singular value is denoted by $s_{\min}(B):=(\min \Spec(B'B))^{1/2}$, so that (recall our notation \eqref{eq:invNorm})
\begin{equation} \label{eq:lowernorm2}
\mu(B) = \left(\min(\nu(B'B),\nu(BB'))\right)^{1/2} = \min(s_{\min}(B),s_{\min}(B')).
\end{equation}
In particular, the above holds when $Y=\C^M$ and $Z=\C^N$, for some $M,N\in \N$, in which case $B$ is an $N\times M$ matrix and $B'$ is the conjugate transpose of $B$. In particular (see the discussion below \eqref{eq:invNorm}), if $M=N$, so that $Y=Z$ and $B$ is a square matrix, then
\begin{equation} \label{eqln3}
\mu(B)=\nu(B)=s_{\min}(B)= (\min \Spec(B'B))^{1/2}=(\nu(B'B))^{1/2}.
\end{equation}

To decide whether a particular $\lambda\in \C$ is contained in one of our inclusion sets it is not necessary to actually compute any lower norms, it is enough  to decide whether or not $\nu(B)\geq \eps$ or is $>\eps$, for some non-negative threshold $\eps$, and for $B\in \B$, where $\B$ is some set of $N\times M$ matrices, depending on which inclusion set method we choose ($\tau$, $\pi$, or $\tau_1$). By \eqref{eq:lowernorm1}, this is equivalent to deciding whether the  $M\times M$ Hermitian matrix $B'B-\eps^2 I_M$ is positive semi-definite or positive definite, in particular
\begin{equation} \label{eq:ispd}
\nu(B) \geq \eps \quad \Leftrightarrow \quad B'B-\eps^2 I_M \mbox{ is positive semi-definite}.
\end{equation}
 This can be tested (cf., \cite[Supplementary Materials]{ColbRomanHansen}) by Gaussian elimination or by attempting an $LDL^T$ version of the Cholesky decomposition \cite[Prop.10.1]{SCI}. These algorithms only require finitely many arithmetic operations. Precisely, $O(M^3)$ operations are sufficient, indeed  $O(w^2M)$ operations suffice  if $B'B$ has band-width $w$ (see, e.g., \cite[Algorithm 4.3.5]{GolubVanLoan}).

Let us emphasise that in this case we are considering where $X$ is a finite-dimensional Hilbert space, $X$ (and so also $E$) satisfy Globevnik's property (see \S\ref{sec:pseud}), so that all the inclusions \eqref{incl:met1}, \eqref{incl:met1*}, \eqref{incl:met2} apply. In this finite-dimensional case the expressions for the key inclusion sets for the $\tau$ and $\pi$ methods simplify somewhat: we have  by Proposition \ref{prop:same} that $\sigma^n_\eps(A)$ and $\Sigma^n_\eps(A)$ are given by \eqref{eq:sigmaAAlt} and, noting \eqref{eqln3}, the expression \eqref{eq:mudag} for $\mu^\dag(A-\lambda I)$ simplifies to 
\begin{equation} \label{eq:sigdef*}
\mu_n^\dag(A-\lambda I) = \inf_{k\in \Z} \nu(A_{n,k}-\lambda I_n), \quad n\in\N, \;\; \lambda\in \C.
\end{equation}
Similarly, recalling from \eqref{eq:pisp} that $\pi^{n,t}_\eps(A) = \pi_\eps^{n,\mathfrak{a},\mathfrak{c}}(A)$ and $\Pi^{n,t}_\eps(A) = \Pi_\eps^{n,\mathfrak{a},\mathfrak{c}}(A)$, when $\mathfrak{a}$ and $\mathfrak{c}$ are given by \eqref{eq:special}, and letting $\mu_n^{\pi,t}$ denote the function $\mu_n^{\mathfrak{a},\mathfrak{c}}$ defined by \eqref{eq:muac} in the case that $\mathfrak{a}$ and $\mathfrak{c}$ are given by \eqref{eq:special}, we have by
 \eqref{eq:piAAlt}, \eqref{eq:piAAlt2}, and \eqref{eqln3} that,  for all $n\in \N$ and $t\in \T$,
\begin{equation} \label{eq:pidef*}
\begin{aligned}
\pi^{n,t}_\eps(A) &=& \{\lambda \in \C: \mu_n^{\pi,t}(A-\lambda I) < \eps\}, \quad \eps>0,\\
\Pi^{n,t}_\eps(A) &=& \{\lambda \in \C:\mu^{\pi,t}_n(A-\lambda I)\leq \eps\}, \quad \eps\geq 0,
\end{aligned}
\end{equation}
where
\begin{equation} \label{eq:pidef*2}
\mu_n^{\pi,t}(A-\lambda I) = \inf_{k\in \Z} \nu(A^{\per,t}_{n,k}-\lambda I_n), \quad n\in\N, \;\; \lambda\in \C,\;\; t\in \T.
\end{equation}
The corresponding representations for the $\tau_1$ method inclusion sets are \eqref{eq:gamdef} and \eqref{eq:munmat}.

The above representations, and the discussion around \eqref{eq:ispd}, make clear that, in the case that $\{A_{n,k}:k\in \Z\}$ is finite, which, as we have noted in \S\ref{sec:comput}, holds if and only if the matrix representation of $A$ has only finitely many distinct entries, we can determine whether a given $\lambda\in \C$ is in one of the $\tau$-method-related sets, $\sigma^n_\eps(A)$ and $\Sigma^n_\eps(A)$, in only finitely many arithmetic operations (assuming that we have already determined the finite set $\{A_{n,k}:k\in \Z\}$, or at least its cardinality). Similar remarks apply to the $\pi$ and $\tau_1$ methods. In the remainder of this section we consider questions of computation in finitely many operations in cases where $A$ may have infinitely many distinct entries, proving Proposition \ref{prop:finite} and an extension to the band-dominated case.\\

\begin{proofof}{Proposition \ref{prop:finite}} 
Suppose $n\in \N$.
Given the inputs specified in the theorem, to compute the finite set $\Gamma^n_{\mathrm{fin}}(A)$, given by \eqref{eq:finite}, one only needs to calculate $R:= \alpha_{\max}+\beta_{\max}+\gamma_{\max}$ and  check whether or not each of the finitely many points in $\mathrm{Grid}(n,R)$, given by \eqref{eq:grid}, are in $\Gamma^{n,\mathrm{fin}}_{\eps^*_n+2/n}(A)$, where $\Gamma^{n,\mathrm{fin}}_\eps(A)$ is given by \eqref{eq:findef} and \eqref{eq:mufin}. But this is just a question of deciding whether or not $\nu(B)\leq \eps^*_n+2/n$ for finitely many matrices $B$, where $\eps^*_n$ is given by \eqref{eq:eps*}, and, as noted above, this can be done in finitely many arithmetic operations.

Recall that $K_n=\A_p(A,n)$ so that \eqref{eq:comp} holds. To see that $\Spec A \subset \widehat \Gamma^n_{\mathrm{fin}}(A)$ it is enough, by Theorem \ref{thm:fin}, and since $\Spec A \subset \|A\|\overline{\D}$, to show that $\Gamma^{n,\mathrm{fin}}_{\eps''_n+1/n}(A)\cap \|A\|\overline{D}  \subset \widehat    \Gamma^n_{\mathrm{fin}}(A)$. So suppose that $\lambda'\in \Gamma^{n,\mathrm{fin}}_{\eps''_n+1/n}(A)$ and $|\lambda'|\leq \|A\|$. Then, since $R\geq \|A\|$, it is easy to see that there exists $\lambda \in \mathrm{Grid}(n,R)$ with $|\lambda-\lambda'|\leq \sqrt{2}/n$. Since $\lambda' \in \Gamma^{n,\mathrm{fin}}_{\eps''_n+1/n}(A)$ and recalling \eqref{eq:lnProp}, we see that there exists $k\in K_n$ such that either $\nu(A_{n,k}^+-\lambda I_n^+) \leq \eps''_n + (1+\sqrt{2})/n$ or $\nu((A^*)_{n,k}^+-\lambda I_n^+) \leq \eps''_n + (1+\sqrt{2})/n< \eps''_n + 3/n$, so that $\lambda\in \Gamma^{n,\mathrm{fin}}_{\eps''_n+3/n}(A)\subset \Gamma^{n,\mathrm{fin}}_{\eps^*_n+3/n}(A)$, i.e., $\lambda \in \Gamma^n_{\mathrm{fin}}(A)$. This implies, since $|\lambda-\lambda'| \leq \sqrt{2}/n < 2/n$, that $\lambda'\in \widehat    \Gamma^n_{\mathrm{fin}}(A)$.
Thus $\Spec A \subset \widehat \Gamma^n_{\mathrm{fin}}(A)$, but also, by the definition \eqref{eq:finite} and Theorem \ref{thm:fin},
$\Gamma^n_{\mathrm{fin}}(A)\subset \Gamma^{n,\mathrm{fin}}_{\eps^*_n+3/n}(A) \subset \Specn_{\eps^*_n+3/n}(A)$, so that
$$
\Spec A\ \subset\ \widehat \Gamma^n_{\mathrm{fin}}(A)\ \subset\  \Specn_{\eps^*_n+3/n}(A) + \frac{2}{n}\overline{\D}.
$$
It follows from \eqref{eq:HDconv} that $\Specn_{\eps^*_n+3/n}(A) + \frac{2}{n}\overline{\D}\Hto \Spec A$ as $n\to\infty$, so that also $\widehat \Gamma^n_{\mathrm{fin}}(A)\to \Spec A$. Since $d_H(\Gamma^n_{\mathrm{fin}}(A),\widehat \Gamma^n_{\mathrm{fin}}(A))\leq 2/n$, we have also that $\Gamma^n_{\mathrm{fin}}(A)\to \Spec A$.
\end{proofof}

Let us, as promised, extend Proposition \ref{prop:finite} and the discussion of \S\ref{sec:SCI} to the band-dominated case.
Let $\Omega=BDO(E)$, with $X=\C$, so that $E=\ell^2(\Z)$ and the matrix entries of $A\in \Omega$ are complex numbers. The mappings we will need, to provide us with the inputs needed to compute a sequence of  approximations to $\Spec A$, for $A\in \Omega$,  are $\A_p$, $\B_p$, $\C_p$, for $p\in \N$,  as defined in \S\ref{sec:SCI}, plus the mapping
\begin{enumerate}
\item[] $\mathcal{D}=(\cD_1,\cD_2):\Omega\times \N\to \Omega_T^n\times \R$, $(\A,n)\mapsto (A_n,\eta_n)$, where $A_n:= \cI_n A^{(n)}\cI_n^{-1}$, $\cI_n$ is defined below \eqref{eq:norm}, $A^{(n)}$ is defined by \eqref{eq:Andef} with $\mathfrak{p}=1$, and $\eta_n \geq 0$ is such that $\eta_n\geq \|A-A^{(n)}\|$ and $\eta_n\to 0$ as $n\to\infty$.
\end{enumerate}
The sequence of approximations to $\Spec A$ is $(\Gamma^n_{\mathrm{fin}}(A))_{n\in \N}$, given by
\begin{equation} \label{eq:finitebdo}
\Gamma^n_{\mathrm{fin}}(A) := \Gamma^{n,\mathrm{fin}}_{\eps^*_n(A_n)+ 3/n + \eta_n}(A_n) \cap \mathrm{Grid}(n,R_n), \qquad n\in \N,
\end{equation}
where $(A_n,\eta_n) := \cD(A,n)$, $(\alpha_{\max},\beta_{\max},\gamma_{\max}) := \B_n(A_n)$, $R_n:= \alpha_{\max}+\beta_{\max}+\gamma_{\max}+\eta_n$ (so that $R_n$ is an upper bound for $\|A\|$), $K_n\subset \Z$, in the definition \eqref{eq:findef} and \eqref{eq:mufin}, is given by $K_n:= \A_n(A_n,n)$, and (cf.~\eqref{eq:epsn''2} and \eqref{eq:eps*})
\begin{equation} \label{eq:eps*bdo}
\eps_n^*(A_n)\ :=\ (\alpha_{\max}+\gamma_{\max})\,\frac{22}{7(n+1)}\ \geq\ 2(\alpha_{\max}+\gamma_{\max})\,\sin\frac{\pi}{2(n+1)}\geq \eps''_n(A_n),
\end{equation}
where $\eps''_n(A_n)$ is defined by \eqref{eq:epsn''2}. As we note in the following theorem, each element of this sequence can be computed in finitely many arithmetical operations, given finitely many evaluations of the input maps $\A_p$, $\B_p$, $\cC_p$, and $\cD$. 

\begin{theorem} \label{prop:finite2}
For $A\in \Omega$ and $n\in \N$, $\Gamma^n_{\mathrm{fin}}(A)$, as defined in \eqref{eq:finitebdo}, can be computed in finitely many arithmetic operations, given $\eta_n := \cD_2(A,n)$, $K_n:= \A_n(A_n,n)$, where  $A_n := \cD_1(A,n)$, $(\alpha_{\max},\beta_{\max},\gamma_{\max}) := \B_n(A_n)$, and, for $k\in K_n$, $((A_n)_{n,k}^+,((A_n)A^*)^+_{n,k})$ $:=$ $\cC_n(A_n,k,n)$. Further, $\Gamma^n_{\mathrm{fin}}(A)\Hto \Spec A$ as $n\to \infty$, and also
$$
\widehat \Gamma^n_{\mathrm{fin}}(A)\ :=\ \Gamma^n_{\mathrm{fin}}(A) + \frac{2}{n}\overline{\D}\ \Hto\ \Spec A
$$
as $n\to \infty$, with $\Spec A \subset \widehat \Gamma^n_{\mathrm{fin}}(A)$ for each $n\in \N$.
\end{theorem}

We will prove Theorem \ref{prop:finite2} shortly, but let us  interpret this theorem  as a result relating to the solvability complexity index (SCI) of \cite{SCIshort,SCI,ColHan2023}, as we did for Proposition \ref{prop:finite} in \S\ref{sec:SCI}.  Inspecting the definition of $\Gamma^n_{\mathrm{fin}}(A)$, note that it is not necessary  to explicitly compute the action of both components of $\cD$,  it is enough to evaluate $\cD_2$, the second component of $\cD$, to obtain $\eta_n=\cD_2(A,n)$, and then to evaluate compositions of $\cD_1$ with $\A_n$, $\B_n$, and $\cC_n$, to obtain
$K_n := (\A_n(\cdot,n)\circ \cD_1)(A)$,  $(\alpha_{\max},\beta_{\max},\gamma_{\max}) := (\B_n\circ \cD_1)(A)$, and $((A_n)_{n,k}^+,((A_n)^*)^+_{n,k}):= (\cC_n(\cdot,k,n)\circ \cD_1)(A)$.
 Equipping $\C^C$ with the Hausdorff metric as in \S\ref{sec:keynot},
the mappings 
$$
\Omega\to \C^C, \qquad A\mapsto \Gamma^n_{\mathrm{fin}}(A) \quad \mbox{and} \quad A\mapsto \widehat \Gamma^n_{\mathrm{fin}}(A),
$$
are, for each $n\in \N$,  general algorithms in the sense of \cite{SCI,ColHan2023},  with evaluation set (in the sense of  \cite{SCI,ColHan2023})
$$
\Lambda:= \{\A_n(\cdot,n)\circ \cD_1,\B_n\circ\cD_1,\cC_n(\cdot,k,n)\circ\cD_1, \cD_2:k\in \Z, \, n\in \N\};
$$ 
each function in $\Lambda$ has domain $\Omega$, and each can be expressed in terms of finitely many complex-valued functions; see the footnote below \eqref{eq:Lambda}.
Further, $\widehat \Gamma^n_{\mathrm{fin}}(A)$ can be computed in finitely many arithmetic operations and specified using finitely many complex numbers (the elements of $\Gamma^n_{\mathrm{fin}}(A)$ and the value of $n$). Thus, where $\Xi:\Omega\to \C^C$ is the mapping given by $\Xi(A) := \Spec A$, for $A\in \Omega$, the computational problem $\{\Xi,\Omega, \C^C,\Lambda\}$ has  arithmetic SCI, in the sense of \cite{SCI,ColHan2023}, equal to one; more precisely, since also $\Spec A \subset \widehat \Gamma^n_{\mathrm{fin}}(A)$, for each $n\in \N$ and $A\in \Omega$, this computational problem is in the class $\Pi_1^A$, as defined in \cite{SCI,ColHan2023}. \\

\begin{proofof}{Theorem \ref{prop:finite2}} 
Suppose $n\in \N$.
Note first that $R:=\sup \{R_n:n\in \N\}<\infty$  as (cf.~\eqref{eq:normbound}) $R_n\leq 3\|A_n\|+\eta_n$, $(\eta_n)$ is a null sequence, and $\|A_n\|=\|A^{(n)}\|\leq \|A\|+\eta_n$. Thus, arguing as in the proof of Proposition \ref{prop:finite}, $\Gamma^n_{\mathrm{fin}}(A)$ can be computed in finitely many arithmetic operations.

Since $K_n=\A_n(A_n,n)$, \eqref{eq:comp} holds with $A$ replaced by $A_n$. To see that $\Spec A \subset \widehat \Gamma^n_{\mathrm{fin}}(A)$, note that, by the first inclusion in  \eqref{eq:spepspert2}, $\Spec A\subset \Spec_{\eta_n} A^{(n)}=\Spec_{\eta_n} A_n$. Further, by Theorem \ref{thm:fin}, $\Spec_{\eta_n} A_n \subset \Gamma^{n,\mathrm{fin}}_{\eta_n+\eps''_n(A_n) + 1/n}$. Since also $\Spec A \subset \|A\|\overline{\D}$, it follows that $\Spec A\subset \Gamma^{n,\mathrm{fin}}_{\eta_n+\eps''_n(A_n) + 1/n}\cap \|A\| \overline{D}$. Since $R_n\geq \|A\|$, it follows by arguing as in the proof of Proposition \ref{prop:finite} that 
 $\Spec A \subset \widehat \Gamma^n_{\mathrm{fin}}(A_n)$. 
Also, by definition and  Theorem \ref{thm:fin},
\[
\Gamma^n_{\mathrm{fin}}(A_n)\subset \Gamma^{n,\mathrm{fin}}_{\eps^*_n(A_n)+3/n+\eta_n}(A_n) \subset \Specn_{\eps^*_n(A_n)+3/n+\eta_n}(A_n),
\]
and
$\Specn_{\eps^*_n(A_n)+3/n+\eta_n}(A_n)\subset \Specn_{\eps^*_n(A_n)+3/n+2\eta_n}(A)$,
by \eqref{eq:PseudInc2}. Thus
\[
\Spec A\ \subset\ \widehat \Gamma^n_{\mathrm{fin}}(A)\ \subset\  \Specn_{\eps^*_n(A_n)+3/n+2\eta_n}(A) + \frac{2}{n}\overline{\D},
\]
so that, by \eqref{eq:HDconv}, $\widehat \Gamma^n_{\mathrm{fin}}(A)\to \Spec A$. Since $d_H(\Gamma^n_{\mathrm{fin}}(A),\widehat \Gamma^n_{\mathrm{fin}}(A))\leq 2/n$, we have also that $\Gamma^n_{\mathrm{fin}}(A)\to \Spec A$.
\end{proofof}

\section{Pseudoergodic operators and other examples} \label{sec:pseudop}
In this section, as in the previous section, our focus is the case that, for some $p\in \N$, $X = \C^p$, which we equip with the Euclidean norm, so that $X$ and  $E=\ell^2(\Z,X)=\ell^2(\Z,\C^p)$ are Hilbert spaces, and the  matrix entries $a_{i,j}$  are $p\times p$ complex matrices. This includes the special case $p=1$, so that $E=\ell^2(\Z)$ and the entries $a_{i,j}\in \C$. Our aim is to 
 illustrate various of the results in earlier sections in the case that $A$ is the tridiagonal matrix \eqref{eq:A}. We focus particularly on the study and computation of $\Sigma_{\eps_n}^n(V)$, $\Pi_{\eps'_n}^{n,t}(V)$, and  $\Gamma_{\eps''_n}^n(V)$, defined in \eqref{eq:sigdef}, \eqref{eq:pidef}, and \eqref{eq:gamdef}, respectively, that, by \eqref{incl:met1}, \eqref{incl:met1*}, and \eqref{incl:met2}, are inclusion sets for $\Spec A$, and the convergence of these sequences of inclusion sets as $n\to \infty$.

\subsection{The shift operator} \label{sec:shift}
As a first example, we consider the shift operator $V$ from Example \ref{ex:shift}
in more detail. As we noted in Example \ref{ex:shift}, $\Spec V$ is the unit circle
$\T$ for  this operator. Moreover, we can obtain exact analytical descriptions for the  inclusion sets
for $\Spec V$ provided by the $\tau$, $\pi$, and $\tau_1$ methods in \eqref{incl:met1}, \eqref{incl:met1*},
and \eqref{incl:met2}. These are $\Sigma_{\eps_n}^n(V)$, $\Pi_{\eps'_n}^{n,t}(V)$, and  $\Gamma_{\eps''_n}^n(V)$, respectively, defined in \eqref{eq:sigdef}, \eqref{eq:pidef}, and \eqref{eq:gamdef}. In this case, as we noted in Example \ref{ex:shift}, $\|\alpha\|_\infty = 1$ and $\|\gamma\|_\infty=0$ so that, by Corollary \ref{minimum_weighted_norm_two}, \eqref{eq:epsnd}, and \eqref{eq:epsn_method2}, $\eps_n=2\sin\frac{\pi}{4n+2}$,
$\eps'_n=2\sin\frac{\pi}{2n}$, and
$\eps''_n=2\sin\frac{\pi}{2n+2}$.

In general, computation of these inclusion sets for an operator $A$ requires consideration of infinitely many submatrices of $A$, indexed by $k\in \Z$. But, as we have noted already in Example \ref{ex:shift}, in the case $A=V$  the relevant $n\times n$ submatrices for the $\tau$ and $\pi$ methods are all the same: we have (see \eqref{eq:Vn}) that $A_{n,k} = V_n$ and $A^{\per,t}_{n,k} = V^{\per,t}_n$, for $k\in \Z$, where $V_n$ and $V_n^{\per,t}$ are defined in \eqref{eq:Vn}, so that, for $n\in \N$,
\begin{equation} \label{eq:sigV}
\Sigma_{\eps_n}^n(V) = \Specn_{\eps_n}(V_n) \quad \mbox{and} \quad \Pi_{\eps'_n}^{n,t}(V) = \Specn_{\eps'_n}(V^{\per,t}_n), \quad t\in \T. 
\end{equation}
Similarly, recalling the definitions \eqref{eq:sigdef} and \eqref{eq:pidef},
$$
\sigma_{\eps_n}^n(V) = \specn_{\eps_n}(V_n) \quad \mbox{and} \quad \pi_{\eps'_n}^{n,t}(V) = \specn_{\eps'_n}(V^{\per,t}_n), \quad t\in \T. 
$$

For the $\tau_1$ method the relevant submatrices for computing $\gamma_{\eps''_n}^n(A)$ and $\Gamma_{\eps''_n}^n(A)$, given by \eqref{eq:gamdef}, are $A^+_{n,k}$ and $(A^*)^+_{n,k}$, for $k\in \Z$, this notation defined in \eqref{eq:An+}. In the case $A=V$ we have, for $n\in \N$ and $k\in \Z$, that
$$
A^+_{n,k}\ =\ V_n^+\ :=\ \begin{pmatrix}
0 & 0&\cdots&0& 0\\\hline
& & V_n\\\hline
0& 0 &\cdots&0 & 1
\end{pmatrix} \quad \mbox{and} \quad (A^*)^+_{n,k}\ =\ (V^*)_n^+\ :=\ \begin{pmatrix}
1 & 0&\cdots&0& 0\\\hline
& & V^T_n\\\hline
0& 0 &\cdots&0 & 0
\end{pmatrix}.
$$
Thus, by \eqref{eq:gamdef}, \eqref{eq:munmat}, and \eqref{eq:lowernorm1},
\begin{equation} \label{eq:GamV}
\gamma_{\eps''_n}^n(V)\ =\ \left\{\lambda\in \C: v(\lambda)< \eps''_n\right\} \quad \mbox{and} \quad
\Gamma_{\eps''_n}^n(V)\ =\ \left\{\lambda\in \C: v(\lambda) \leq \eps''_n\right\},
\end{equation}
where
\begin{eqnarray*}
v(\lambda) &:=& \min\left(\nu(V_n^+-\lambda I_n^+), \nu((V^*)_n^+-\lambda I_n^+)\right)\\ 
&=& \min\left(s_{\min}(V_n^+-\lambda I_n^+), s_{\min}((V^*)_n^+-\lambda I_n^+)\right).
\end{eqnarray*}

 With $\D$ and $\T_n$ as defined in \S\ref{sec:keynot}, it follows from \eqref{eq:sigV} and \eqref{eq:GamV} that
\begin{eqnarray}
\label{eq:shift1}\Sigma_{\eps_n}^n(V) &=& \overline{\D},\\
\label{eq:shift2}\Pi_{\eps'_n}^{n,t}(V) &=& t^{1/n}\T_n+\eps'_n\overline{\D}, \quad \mbox{for $t\in \T$,} \;\; \mbox{ and that}\\
\label{eq:shift3}\Gamma_{\eps''_n}^n(V) &=&
\left\{\begin{array}{ll}\textstyle\overline{\D}\,\setminus\,
\left(1-(\eps''_n)^2\right)\D, & n\geq 3,\\
\overline{\D}, & n=1,2.
\end{array}\right.  
\end{eqnarray}
We have shown \eqref{eq:shift2} already in Example \ref{ex:shift}, and we will show the other identities above in a moment. Clearly, as predicted by \eqref{incl:met1}, \eqref{incl:met1*}
and \eqref{incl:met2}, these identities show that
$$
\Spec V\ \subset\ \Sigma_{\eps_n}^n(V) \cap \Pi_{\eps'_n}^{n,t}(V) \cap \Gamma_{\eps''_n}^n(V), \quad t\in \T,
$$
but, as we will see in the calculations below, for $n\in \N$ and $t\in \T$,
\begin{equation} \label{eq:shift_nobound}
\Spec V\ \not\subset\ \sigma_{\eps_n}^n(V),\qquad \Spec V\
\not\subset\ \pi_{\eps'_n}^{n,t}(V),\qquad\textrm{and}\qquad \Spec V\
\not\subset\ \gamma_{\eps''_n}^n(V),
\end{equation}
since
\begin{equation} \label{eq:shiftnew}
\sigma_{\eps_n}^n(V) = \D, \;\; \pi_{\eps'_n}^{n,t}(V) = t^{1/n}\T_n+\eps'_n\D, \;\; 
\gamma_{\eps''_n}^n(V) =
\left\{\begin{array}{ll}\D\,\setminus\,
\left(1-(\eps''_n)^2\right)\overline{\D}, & n\geq 2,\\
\D, & n=1.
\end{array}\right.
\end{equation}
This makes clear that $\eps_n$, $\eps'_n$ and $\eps''_n$ are
just as large as necessary for \eqref{incl:met1}, \eqref{incl:met1*},
and \eqref{incl:met2} to hold. 
Observe also that the above expressions imply that
$$
d_H(\Spec V,\Gamma_{\eps''_n}^n(V)) = 4\sin^{2}\frac{\pi}{2(n+1)} \quad \mbox{and} \quad d_H(\Spec V,\Pi_{\eps'_n}^n(V)) = 2\sin\frac{\pi}{2n},
$$
so that $\Gamma_{\eps''_n}^n(V) \Hto \Spec V$ as $n\to\infty$, as predicted by Theorem \ref{thm:converge}, and also $\Pi_{\eps'_n}^n(V)\Hto$ $\Spec V$ (this is a special case of Theorem \ref{thm:periodic} below). However, $d_H(\Spec V,\Gamma_{\eps''_n}^n(V))$ $=1$ for all $n\in \N$.\\

\begin{proofof}{formulas \eqref{eq:shift1}--\eqref{eq:shift_nobound}}
Fix $n\in\N$ and put $R_n(\lambda):=\|(V_n-\lambda I_n)^{-1}\|$ for
$\lambda\in\C$ (with $R_n(0):=\infty$), where $V_n$ is as in
\eqref{eq:Vn}. Firstly, $R_n$ is invariant under rotation around
$0$. Indeed, let $\ph\in [0,2\pi)$ and put
$D_n:=\diag(e^{\ri\ph},e^{\ri 2\ph},...,e^{\ri n\ph})$. Then $D_n$
is an isometry and $V_nD_n=e^{-\ri\ph}D_nV_n$, so that
\begin{eqnarray*}
R_n(e^{\ri\ph}\lambda)&=&\|(V_n-e^{\ri\ph}\lambda I_n)^{-1}\|\ =\
\|e^{-\ri\ph}(e^{-\ri\ph}V_n-\lambda I_n)^{-1}\|\\
&=&\|(D_n^{-1}V_nD_n-\lambda I_n)^{-1}\|\ =\ \|D_n^{-1}(V_n-\lambda
I_n)^{-1}D_n\|\ =\ R_n(\lambda).
\end{eqnarray*}
Secondly, $R_n$ is strictly monotonously decreasing on the half-axis
$(0,+\infty)$. Indeed, suppose $0<\lambda_1<\lambda_2<\infty$ and
$R_n(\lambda_1)\le R_n(\lambda_2)$. Since
$\lim_{\lambda\to\infty}R_n(\lambda)=0$, there is a
$\lambda_3\ge\lambda_2$ with $R_n(\lambda_3)=R_n(\lambda_1)$. Thus
$R_n$ attains its supremum $M:=\max_{\lambda\in
[\lambda_1,\lambda_3]}R_n(\lambda)\ge R_n(\lambda_2)\ge
R_n(\lambda_1)=R_n(\lambda_3)$ in the open disk
$U:=\frac{\lambda_1+\lambda_3}2 + \frac{\lambda_3-\lambda_1}2\D$,
which contradicts the maximum principle of \cite{Globevnik76}, noted in  \cite[Theorem 2.2]{Shargo08}.
Finally, recalling \eqref{eq:lowernorm1} and the notations $B_n$, $E_n(\phi)$, and $\varrho_n(\phi)$ of
\S\ref{sec:mini}, $R_n(1)=\|B_n^{-1}\|$ and $1/\|B_n^{-1}\|^2=(\nu(B_n))^2 = \min \Spec(B_n^TB_n)=\min \Spec(E_n(1))=\varrho_n(1)$. Further, by \eqref{eq:muphi} and \eqref{prop2}, $\varrho_n(1)=\eps_n^2$. Thus $R_n(1)=1/\eps_n$.
 Taking all of this together, we get that
$\Sigma_{\eps_n}^n(V)=\{\lambda\in\C:R_n(\lambda)>1/\eps_n\}=\overline{\D}$ and $\sigma_{\eps_n}^n(V) = \D$.

We have shown  in Example \ref{ex:shift} that
$V^{\per,t}_n$ is normal and $\Spec
V^{\per,t}_n = t^{1/n}\T_n$, so that
\begin{equation} \label{eq:shiftper}
\Pi_{\eps}^{n,t}(V)\ =\ \Specn_{\eps} V^{\per,t}_n\ =\ \Spec
V^{\per,t}_n+\eps\overline{\D}\ =\ t^{1/n}\T_n+\eps\overline{\D},\qquad \eps>0.
\end{equation}
Clearly, for each $t\in \T$, $\Pi_{\eps}^{n,t}(V)$ covers $\Spec V=\T$ if and only if it covers
$t^{1/n} \T_{2n}$. The latter is the case if and only if $\eps$ is greater than or
equal to the Euclidean distance of two adjacent points of $\T_{2n}$,
which is $2\sin\frac{\pi}{2n}=\eps'_n$. Thus  $\Spec V\subset \Pi_{\eps'_n}^{n,t}(V)$, but $\Spec V \not\subset \pi_{\eps'_n}^{n,t}(V)=\ t^{1/n}\T_n+\eps'_n \D$, since the points $t^{1/n}\mathfrak{S}_n$, where $\mathfrak{S}_n:=\T_{2n}\setminus \T_n$, which are points of intersection of the discs of radius $\eps'_n$ centred on $\T_n$, are not contained in $\pi_{\eps'_n}^{n,t}(V)$. Nor does $\pi_{\eps'_n}^{n,t}(V)$ contain the other points of intersection of these discs, the points  $rt^{1/n}\mathfrak{S}_n$, with $r=1-(\eps'_n)^2$, or any of the points $rt^{1/n}\mathfrak{S}_n$, with $0\leq r\leq 1-(\eps'_n)^2$ or with $r\geq 1$. On the other hand, by elementary geometry, $\pi_{\eps'_n}^{n,t}(V)$ contains the annulus $\D\setminus (1-(\eps'_n)^2)\overline{\D}$. Thus
$$
\bigcap_{t\in \T} \pi^{n, t}_{\eps'_n}(V)\ =\ \bigcap_{t\in \T} \left(t^{1/n} \T_n + \eps'_n\, \D\right)\ =\ \D\setminus \left(1-(\eps'_n)^2\right)\overline{\D},
$$
and, by \eqref{eq:piAAlt2}, also $\bigcap_{t\in \T} \,\Pi^{n, t}_{\eps'_n}(V)=\overline{\D}\setminus (1-(\eps'_n)^2)\D$.

To see that \eqref{eq:shift3} holds, note that, for $n\in\N$, recalling \eqref{eq:lowernorm1}, and where $D_n:=I_n$ if $\lambda=0$, $D_n:=\diag(w^1,...,w^n)$ with
$w:=|\lambda|/\overline{\lambda}$ if $\lambda\neq 0$,
\begin{eqnarray*}
\left(\nu(V_n^+-\lambda I_n^+)\right)^2 & = & \left(\nu((V-\lambda I)|_{E_{n,0}})\right)^2\ =\ \min \Specn\left(P_{n,0}(V-\lambda I)'(V-\lambda I)|_{E_{n,0}}\right)\\
& = & \min \Specn\left((1+|\lambda|^2)I_n-\lambda V^T_n-\overline\lambda V_n\right)\\
& = &\min \Specn\left(D_n^{-1}\Big((1+|\lambda|^2)I_n-|\lambda| (V^T_n+ V_n)\Big)D_n\right)\\
& = &\min \Specn\left((1+|\lambda|^2)I_n-|\lambda| (V^T_n+ V_n)\right)\\
& = & 1+|\lambda|^2-2|\lambda|c_n,
\end{eqnarray*}
by \eqref{eq:spec_laplace}, where $c_n:=\cos\frac{\pi}{n+1}=1-\frac{1}{2}(\eps''_n)^2$. Similarly,
\begin{eqnarray*}
\left(\nu((V^*)_n^+-\lambda I_n^+)\right)^2 & = & \min \Specn\left(P_{n,0}(V'-\lambda I)'(V'-\lambda I)|_{E_{n,0}}\right)\\
& = & \min \Specn\left((1+|\lambda|^2)I_n-\lambda V_n-\overline\lambda V^T_n\right)\\
& = &\min \Specn\left((1+|\lambda|^2)I_n-|\lambda| (V^T_n+ V_n)\right)\\
& = & 1+|\lambda|^2-2|\lambda|c_n,
\end{eqnarray*}
so that
$$
v(\lambda) = \left(1+|\lambda|^2-2|\lambda|c_n\right)^{1/2}, \qquad \lambda\in \C.
$$
Now $\lambda\in\Gamma_{\eps_n''}^n(V)$ if and only if $v(\lambda)\leq \eps''_n$, so if and only if
\begin{equation} \label{eq:modl}
(1-|\lambda|)^2 \leq (\eps''_n)^2(1-|\lambda|).
\end{equation}
If $\eps''_n \geq 1$, which is the case for $n=1$ and $2$, then \eqref{eq:modl} holds if and only if $|\lambda|\leq 1$. If $\eps''_n < 1$, which is the case for $n\geq 3$, then \eqref{eq:modl} holds if and only if $1-(\eps''_n)^2\leq |\lambda|\leq 1$. Similarly, $\lambda\in\gamma_{\eps_n''}^n(V)$ if and only if $v(\lambda)< \eps''_n$, which holds if and only if $|\lambda|<1$ for $n=1$, if and only if $1-(\eps''_n)^2 <|\lambda|< 1$, for $n\geq 2$ (note that $\eps''_2=1$).
\end{proofof}

\subsection{Block-Laurent operators} \label{sec:bL}

Suppose now, generalising the shift operator example of \S\ref{sec:shift}, that $X=\C^p$, for some $p\in \N$, and that $A$ is given by \eqref{eq:A} with each diagonal constant, so that, for some $\widehat \alpha$, $\widehat \beta$, $\widehat \gamma \in L(X)=\C^{p\times p}$, $\alpha_{j}=\widehat \alpha$, $\beta_{j}=\widehat \beta$, and $\gamma_{j}=\widehat \gamma$, for $j\in \Z$, i.e., $A$ is a {\em block-Laurent} operator (a {\em Laurent} operator in the case $p=1$). Generalising the notation of \S\ref{sec:shift}, let $V\in L(E)$ denote the forward shift operator, so that $(Vx)_j=x_{j+1}$, for $x\in E=\ell^2(\Z,X)$ and $j\in \Z$. Then\footnote{In the following and similar expressions $\widehat \alpha$, $\widehat \beta$, and $\widehat \gamma$ are to be understood as operating pointwise  on the elements of a given $x=(x_i)_{i\in \Z}\in E$, so that, e.g., $(\widehat \alpha x)_i := \widehat \alpha x_i$, for $i\in \Z$.}
\begin{equation} \nonumber 
A = \widehat \alpha V^1 + \widehat \beta V^0 + \widehat \gamma V^{-1} = a(V),
\end{equation}
where $a$ is the {\em Laurent polynomial}
$$
a(z) := \widehat \alpha z^1 + \widehat \beta z^0 + \widehat \gamma z^{-1}, \qquad z\in \T.
$$
It is well-known that $A$ is unitarily equivalent to multiplication by $a$, the {\em symbol} of $A$. Precisely (see, e.g., Theorem 4.4.9 and its proof in \cite{Davies2007:Book}),  the Fourier operator $\cF:E=\ell^2(\Z,X)\to L^2(\T,X)$, defined  by
$$
(\cF x)(z) = \frac{1}{\sqrt{2\pi}}\sum_{j\in \Z} x_j z^{-j}, \qquad z\in \T,
$$
is unitary, and 
$$
A = \cF^{-1}M_a\cF,
$$
where $M_a:L^2(\T,X)\to L^2(\T,X)$ is the operation of multiplication by $a\in C(\T,L(X))$, i.e. $(M_a\phi)(z)=a(z)\phi(z)$, for $\phi\in L^2(\T,X)$, $z\in \T$. Thus \cite[Theorem 4.4.9]{Davies2007:Book}
\begin{equation} \label{eq:specform}
\Spec A = \Spec M_a = \bigcup_{z\in \T} \Spec a(z),
\end{equation}
and, similarly,
\begin{equation} \label{eq:psform}
\begin{aligned}
\speps A &= \speps M_a = \bigcup_{z\in \T} \speps a(z) \quad \mbox{and} \\
\Speps A &= \Speps M_a = \bigcup_{z\in \T} \Speps a(z), \quad \eps>0.
\end{aligned}
\end{equation}

In this block-Laurent case similar spectral computations hold for the matrices $A_{n,k}^{\per, t}$, given by \eqref{eq:Ankper}, that arise in the $\pi$ method. For $n\in \N$, $t\in \T$, and all $k\in \Z$,
\[
A_{n,k}^{\per,t}\ =\  \widehat\alpha V_{n,t}^{1}+\widehat\beta V_{n,t}^0+\widehat\gamma V_{n,t}^{-1}\ =\ a(V_{n,t})
\]
where
\[
V_{n,t}\ :=\ \begin{pmatrix}
0&&&tI_X\\
I_X&0\\
&\ddots&\ddots\\
&&I_X&0
\end{pmatrix}_{\!\! n\times n}.
\]
A simple computation shows that, for any $s\in\T$ with $s^n=t$ (cf.\ the calculation in Example \ref{ex:shift}),
\[
V_{n,t}\ =\ s D_n^{-1}V_{n,1} D_n\qquad\text{with}\qquad D_n :=\diag(s^1I_X,s^2I_X,\dots, s^nI_X),
\]
so that
\begin{equation} \label{eq:Ank-sim}
A_{n,k}^{\per,t}\ =\ a(V_{n,t})\ =\ a(s D_n^{-1} V_{n,1} D_n)\ =\ D_n^{-1} a(sV_{n,1})D_n,
\end{equation}
so that $A_{n,k}^{\per,t}$ is unitarily equivalent to $a(sV_{n,1})$. Moreover (see, e.g., \cite{Davis}, \cite[{\S}II.7]{TrefEmbBook}),
 $V_{n,1}$ can be diagonalized as $V_{n,1}=\cF^{-1}_n M_n \cF_n$, where $\cF_n:X^n\to X^n$ is the discrete Fourier transform 
\[
\cF_n=\frac 1{\sqrt n}\left(\omega^{jk}I_X\right)_{j,k=1}^n
\quad\text{and}\quad 
M_n:=\diag(\omega^1I_X,\omega^2I_X,\dots,\omega^nI_X)
\]
with
$\omega:=\exp(2\pi i/n)$. Combining this observation with \eqref{eq:Ank-sim}, we see that
\begin{equation} \label{eq:Ank_sim}
A_{n,k}^{\per,t}\ =\ D_n^{-1}\cF^{-1}_n\ a(sM_n)\ \cF_nD_n \quad \mbox{with} \quad  a(sM_n)\ =\ \diag( a(s\,\omega^1),\ldots,a(s\,\omega^n)),
\end{equation}
in particular that $A_{n,k}^{\per,t}$ is unitarily equivalent to $a(sM_n)$.
Since 
$\{\omega^1,\omega^2,\dots,\omega^n\}=\{z\in\C:z^n=1\}=\T_n$, we see, similarly to \eqref{eq:specform} and \eqref{eq:psform}, that
\begin{equation} \label{eq:specform2}
\Spec A_{n,k}^{\per,t} = \Spec a(sM_n) = \bigcup_{z\in s\T_n} \Spec a(z),
\end{equation}
and that, for $\eps>0$,
\begin{equation} \label{eq:psform2}
\speps A_{n,k}^{\per,t} = \bigcup_{z\in s\T_n} \speps a(z) \quad \mbox{and} \quad \Speps A_{n,k}^{\per,t} = \bigcup_{z\in s\T_n} \Speps a(z).
\end{equation}
Note that the set $s\T_n$ is independent of which root $s$ we choose as the solution of $s^n=t$.

The above observations lead, applying Theorem \ref{thm:converge3}, to the following result about convergence of the $\pi$ method in the case that $A$ is block-Laurent. Of course, we already have, as a consequence of Theorem \ref{thm:converge}, that the $\tau_1$ method is convergent (recall that $X$ has Globevnik's property as it is finite-dimensional), and we know, from the example in \S\ref{sec:shift}, that the $\tau$ method is not convergent for tridiagonal block--Laurent operators in general, even in the special case that $X=\C$.

\begin{theorem} \label{thm:periodic}
If $X=\C^p$, equipped with the Euclidean norm, and $A\in L(E)$ is a tridiagonal block-Laurent operator, then, for every $k\in \Z$, $n\in \N$, and $t\in \T$, $A_{n,k}^{\per,t}=A_{n,0}^{\per,t}$,
\begin{equation} \label{eq:spper}
\Pi^{n,t}_\eps(A)\ =\ \Speps A_{n,0}^{\per,t}, \quad \mbox{for } \eps\geq 0, \quad \mbox{and} \quad \pi^{n,t}_\eps(A)\ =\ \speps A_{n,0}^{\per,t}, \quad \mbox{for } \eps> 0.
\end{equation}
Further,
 the $\pi$ method does not suffer from spectral pollution for $A$ for all $t\in \T$; indeed, for all $n\in \N$ and $\eps>0$,
\begin{equation} \label{eq:specunion}
\begin{aligned}
\Spec A= \bigcup_{t\in \T}\Spec A_{n,0}^{\per,t}, \quad
\speps A = \bigcup_{t\in \T}\speps A_{n,0}^{\per,t}, \quad \mbox{and}\\
\Speps A = \bigcup_{t\in \T}\Speps A_{n,0}^{\per,t},
\end{aligned}
\end{equation}
so that, as $n\to\infty$,
$$
\Pi^{n,t}_{\eps+\eps'_n}(A)\ =\ \Specn_{\eps+\eps'_n} A_{n,0}^{\per,t}\ \Hto\ \Speps A, \;\; \mbox{ for } \eps\geq 0, 
$$
in particular, $\Pi^{n,t}_{\eps'_n}(A) = \Specn_{\eps'_n} A_{n,0}^{\per,t}\Hto \Spec A$, and
$$
\pi^{n,t}_{\eps+\eps'_n}(A)\ =\ \specn_{\eps+\eps'_n} A_{n,0}^{\per,t}\ \Hto\ \speps A, \;\; \mbox{ for } \eps> 0.
$$
\end{theorem}
\begin{proof} That $A_{n,k}^{\per,t}=A_{n,0}^{\per,t}$, for $k\in \Z$, and the representations \eqref{eq:spper},  are clear from  the definitions \eqref{eq:Ankper} and \eqref{eq:pidef} of $A_{n,k}^{\per,t}$, $\Pi^{n,t}_\eps(A)$, and $\pi^{n,t}_\eps(A)$. That \eqref{eq:specunion} holds follows by comparing \eqref{eq:specform} and \eqref{eq:psform} with \eqref{eq:specform2} and \eqref{eq:psform2}, noting that $\cup_{t\in \T} (s\T_n) =\T$. The second of the equalities in \eqref{eq:specunion} implies, by Definition \ref{def:specpol}, that, for every $t\in \T$, the $\pi$ method does not suffer from spectral pollution for $A$. The remaining results follow from \eqref{eq:spper} and Theorem \ref{thm:converge3}, noting that $X$ satisfies Globevnik's property (see \S\ref{sec:pseud}) since $X$ is finite-dimensional.
\end{proof}

\subsection{Periodic tridiagonal matrices}
\label{sec:perturb_periodic} Suppose $A$ is given by \eqref{eq:A} with $X=\C$ so that $E=\ell^2(\Z)$,
with all three diagonals $p$-periodic for some $p\in\N$. Then $A$
can be understood as a block-Laurent operator, so as to make use of the results from the previous section. Precisely, as in our discussion around \eqref{eq:norm}, define $E_p:= \ell^2(\Z,X^p)$, equipping $X^p=\C^p$ with the Euclidean norm \eqref{eq:norm}, and let  $A_p:=\cI_p^{-1}V^{k}A V^{-k}\cI_p$, for some $k\in \Z$, where $V$ is the shift operator as defined in \S\ref{sec:shift}. Then $A_p\in L(E_p)$ is a tridiagonal block-Laurent operator with constant entries on the diagonals, $\widehat \alpha$, $\widehat \beta$, $\widehat \gamma \in L(X) = \C^{p\times p}$, where
\[
(\widehat\alpha\,|\,\widehat\beta\,|\,\widehat\gamma)\ =\
\left(\begin{array}{ccc|cccc|ccc}
~~&~~&\alpha_{p+k}&\beta_{1+k}&\gamma_{2+k}&&&~~&~~&~~\\
~~&~~&~~&\alpha_{1+k}&\beta_{2+k}&\ddots& &~~&~~&~~\\
~~&~~&~~&&\ddots&\ddots&\gamma_{p+k}&~~&~~&~~\\
~~&~~&~~&&&\alpha_{p+k-1}&\beta_{p+k}&\gamma_{1+k}&~~&~~
\end{array}\right)_{p\times 3p}.
\]
Noting that the $n\times n$ matrix $(A_p)^{\per,t}_{n,0}$ with entries in $\C^{p\times p}$ is identical to the $(np)\times (np)$ matrix $A^{\per,t}_{np,k}$, we have the following corollary of Theorem \ref{thm:periodic}. The first equation of \eqref{eq:specunion2} is well-known  (see \cite{BoeSi,BoeSi2,Davies2007:Book}). In the case $n=1$ it expresses $\Spec A$ as the union over $t$ of the eigenvalues of the $p\times p$ matrix $\Spec A_{p,k}^{\per,t}$.

\begin{corollary} \label{cor:periodic}
Suppose that $X=\C$ and $A\in L(E)$ is tridiagonal with, for some $p\in \N$, $\alpha_{j+p}=\alpha_j$, $\beta_{j+p}=\beta_j$, $\gamma_{j+p}=\gamma_j$, for $j\in \Z$. Then, for every $k\in \Z$, $n\in \N$,  and $\eps>0$,
\begin{equation} \label{eq:specunion2}
\begin{aligned}
\Spec A = \bigcup_{t\in \T}\Spec A_{pn,k}^{\per,t}, \quad \speps A = \bigcup_{t\in \T}\speps A_{pn,k}^{\per,t}, \quad \mbox{and}\\
\Speps A\ =\ \bigcup_{t\in \T}\Speps A_{pn,k}^{\per,t}.
\end{aligned}
\end{equation}
Further,  as $n\to\infty$, if $k\in \Z$ and $t\in \T$, then
$$
\Specn_{\eps+\eps'_n} A_{pn,k}^{\per,t}\ \Hto\ \Speps A, \;\; \mbox{ for } \eps\geq 0, \quad \mbox{in particular} \quad \Specn_{\eps'_n} A_{pn,k}^{\per,t}\ \Hto\ \Spec A,
$$
 and
$$
\specn_{\eps+\eps'_n} A_{pn,k}^{\per,t}\ \Hto\ \speps A, \;\; \mbox{ for } \eps> 0.
$$
\end{corollary}

The above result implies that, for $\eps\geq 0$ and $t\in \T$,
$$
\Pi^{pn,t}_{\eps+\eps'_n}(A)\ =\ \bigcup_{k=1}^p\Specn_{\eps+\eps'_n} A_{pn,k}^{\per,t}\ \Hto\ \Speps A, 
$$
as $n\to\infty$, and similarly that $\pi^{pn,t}_{\eps+\eps'_n}(A)  \Hto \speps A$, for $\eps>0$. A natural question is: does $\Pi^{n,t}_{\eps+\eps'_n}(A) \Hto \Speps A$,  as $n\to\infty$, for $\eps\geq 0$? I.e., does  the whole sequence converge and not just the subsequence $\Pi^{pn,t}_{\eps+\eps'_n}(A)$? The answer is not necessarily!

As an example where the whole sequence does not converge, suppose that $A$ has $p$-periodic diagonals with $p=2$, and that the entries of $A$ are chosen so that the block-Laurent operator $A_p=\cI_p^{-1}A\cI_p$ has $\widehat \alpha = \widehat \gamma=0$ and
$$
\widehat \beta = \left(\begin{array}{cc} 0 & 1\\ 1 & 0\end{array}\right),
$$
so that $A_2$ is block diagonal and $A_2$ and $A$ are self-adjoint. Then $\Spec A=\Spec A_2 = \Spec \widehat \beta =\{-1,1\}$; indeed, since $A$  is normal, we have further that $\Speps A = \Spec A + \eps\overline{\D} = \{-1,1\}+\eps\overline{\D}$, for $\eps\geq 0$. But, if $n$ is odd, then, for $t\in \T$, the last column of $A_{n,0}^{\per, t}$ is zero, so that $0\in \Spec A_{n,0}^{\per,t}$, so that (see \S\ref{sec:pseud}) $\eps \overline{\D}\subset \Speps A_{n,0}^{\per,t}$, for $\eps\geq 0$. Thus $(\eps+\eps'_n) \overline{\D}\subset \Pi^{n,t}_{\eps+\eps'_n}(A)$, for $\eps\geq 0$, in particular $(\eps+\eps'_n) \ri\in  \Pi^{n,t}_{\eps+\eps'_n}(A)$. Thus, for $\eps\geq 0$,
\begin{eqnarray*}
d_H(\Pi^{n,t}_{\eps+\eps'_n}(A),\Speps A) &\geq& \dist((\eps+\eps'_n) \ri,\Speps A)\\
&=& \sqrt{1+(\eps+\eps'_n)^2}-\eps\ \geq\ \sqrt{1+\eps^2}-\eps,
\end{eqnarray*}
so $\Pi^{n,t}_{\eps+\eps'_n}(A)$ does not converge to $\Speps A$ as $n\to\infty$.

On the other hand, the numerical results in Figures \ref{fig:3per1} and \ref{fig:3per2}, both examples with $p=3$, suggest that there are periodic operators for which the whole sequence does converge. It appears in both figures that $\Pi^{n,t}_{\eps'_n}(A)\Hto \Spec A$ as $n\to\infty$ (note that the values of $n$ we choose are not multiples of 3). Note also, as predicted by Theorem \ref{thm:converge}, that $\Gamma^n_{\eps''_n} \Hto \Spec A$ as $n\to\infty$. Further, it appears in these two examples (and also in Example \ref{ex:shift}) that $\Sigma^n_{\eps_n}(A)$, the $\tau$ method inclusion set,  converges to $\widehat{\Spec} A$, as $n\to\infty$, where $\widehat \Spec A$ denotes the complement of the unbounded component of $\C\setminus \Spec A$.

\begin{figure}[h]
\begin{center}
\begin{tabular}{|c|cccc|}
\hline
& $n=16$ & $n=32$ & $n=64$ & $n=128$\\
\hline
&&&&\\[-1em]
$\stackrel{\text{$\tau$ method}}{\rule{0pt}{10mm}}$ &
\includegraphics[width=22mm]{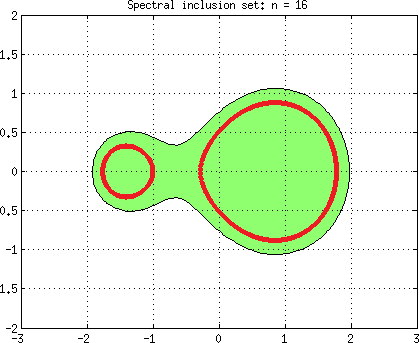} &
\includegraphics[width=22mm]{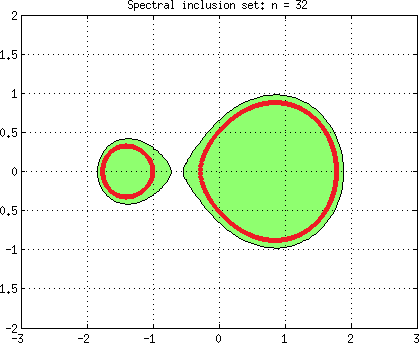} &
\includegraphics[width=22mm]{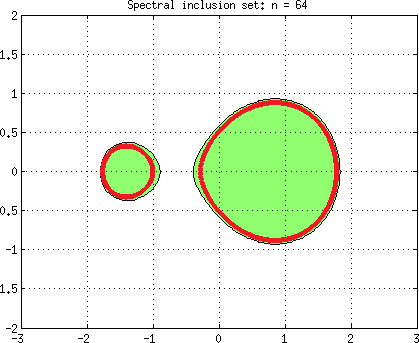} &
\includegraphics[width=22mm]{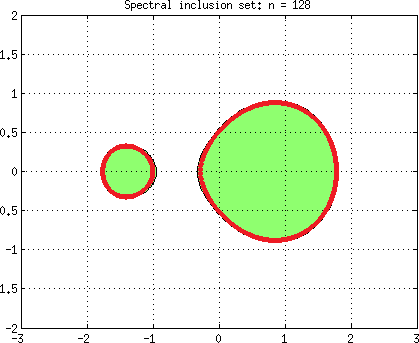}\\
$\stackrel{\text{$\pi$ method}}{\rule{0pt}{10mm}}$ &
\includegraphics[width=22mm]{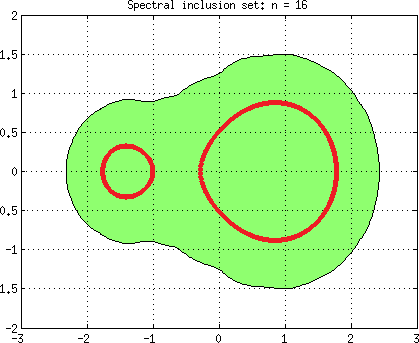} &
\includegraphics[width=22mm]{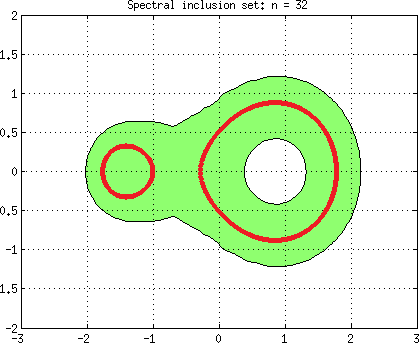} &
\includegraphics[width=22mm]{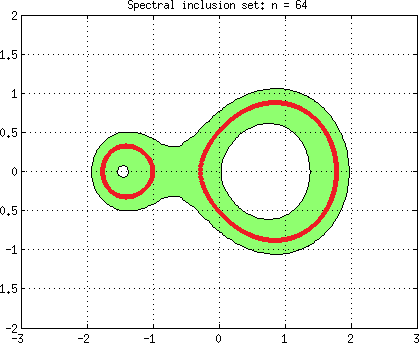} &
\includegraphics[width=22mm]{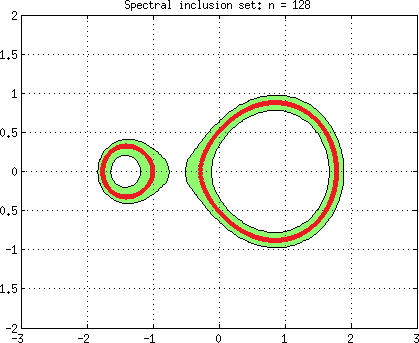}\\
$\stackrel{\text{$\tau_1$ method}}{\rule{0pt}{10mm}}$ &
\includegraphics[width=22mm]{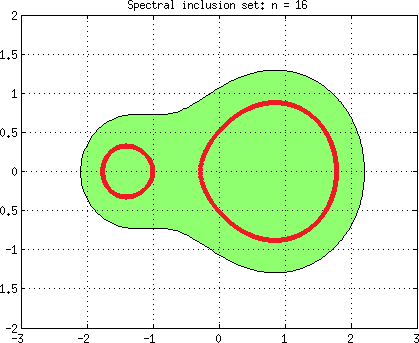} &
\includegraphics[width=22mm]{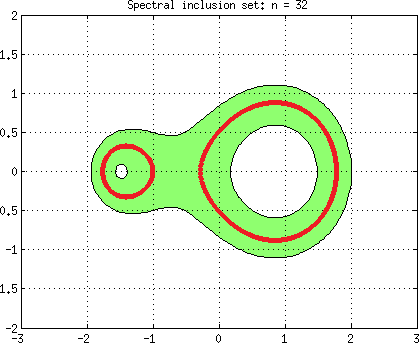} &
\includegraphics[width=22mm]{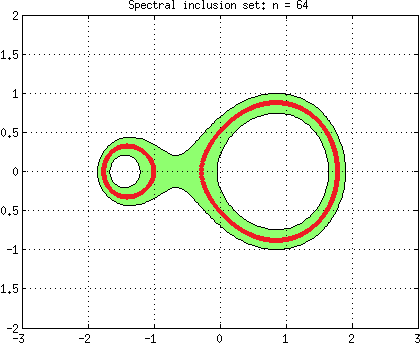} &
\includegraphics[width=22mm]{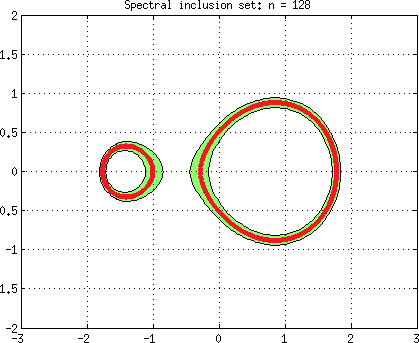}\\
\hline
\end{tabular}
\end{center}
\caption{The $\tau$, $\pi$, and $\tau_1$ inclusion sets, $\Sigma^n_{\eps_n}(A)$, $\Pi^{n,1}_{\eps'_n}(A)$, and $\Gamma^n_{\eps''_n}(A)$, respectively, for $n=16,\,32,\,64$, and $128$, in the case that $\alpha=(\cdots,0,0,0,\cdots)$,
$\beta=(\cdots,-\frac32,1,1,\cdots)$ and $\gamma=(\cdots,1,2,1,\cdots)$ are $3$-periodic. These are each inclusion sets for $\Spec A$, given by the first equation of \eqref{eq:specunion2} with $n=1$, and shown in each figure as the red curve.} \label{fig:3per1}
\end{figure}

Figures \ref{fig:3per1} and \ref{fig:3per2} are plotted in Matlab, using Matlab's inbuilt contouring routine. In particular, the plots of $\Sigma^n_{\eps_n}(A)$ are produced by noting that the boundary of  $\Sigma^n_{\eps_n}(A)$ is, by Proposition \ref{prop:same} and \eqref{eq:SigW}, and recalling \eqref{eq:sigdef*}, the contour on which
\begin{equation} \label{eq:cont}
\begin{aligned}
\mu_n^\dag(A-\lambda I) &= \min_{k=1,\ldots,p} \nu(A_{n,k}-\lambda I_n)\\
 &= \left(\min_{k=1,\ldots,p} \min \Specn\left((A_{n,k}-\lambda I_n)'(A_{n,k}-\lambda I_n)\right)\right)^{1/2}
\end{aligned} 
\end{equation}
has the value $\eps_n$. This contour is determined by Matlab's contouring routine, interpolating values of the right hand side of \eqref{eq:cont} plotted on a fine grid of values of $\lambda$ covering the respective regions shown in Figures  \ref{fig:3per1} and \ref{fig:3per2}.  $\Pi^{n,1}_{\eps'_n}(A)$ and  $\Gamma^n_{\eps''_n}(A)$ are plotted similarly, starting from the representations \eqref{eq:pidef*} and \eqref{eq:pidef*2} for  $\Pi^{n,1}_{\eps'_n}(A)$, and \eqref{eq:gamdef} and \eqref{eq:munmat} for $\Gamma^n_{\eps''_n}(A)$.

\begin{figure}[t]
\begin{center}
\begin{tabular}{|c|cccc|}
\hline
& $n=4$ & $n=8$ & $n=16$ & $n=32$\\
\hline
&&&&\\[-1em]
$\stackrel{\text{$\tau$ method}}{\rule{0pt}{10mm}}$ &
\includegraphics[width=22mm]{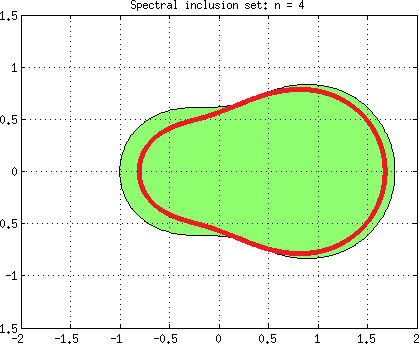} &
\includegraphics[width=22mm]{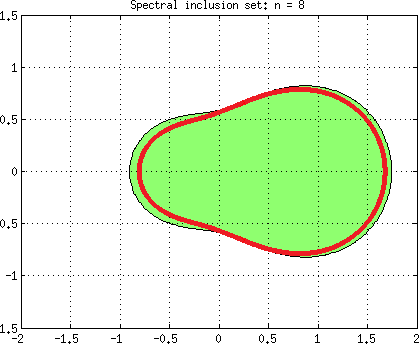} &
\includegraphics[width=22mm]{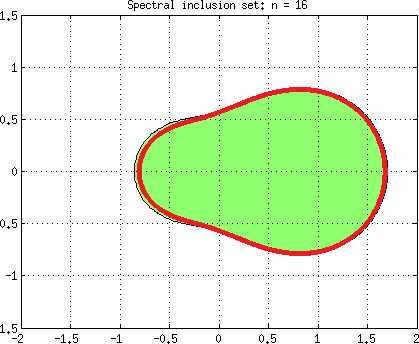} &
\includegraphics[width=22mm]{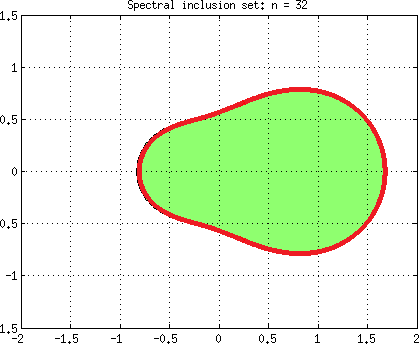}\\
$\stackrel{\text{$\pi$ method}}{\rule{0pt}{10mm}}$ &
\includegraphics[width=22mm]{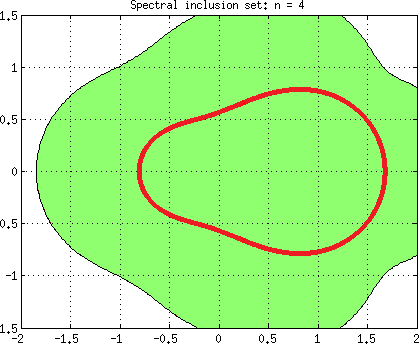} &
\includegraphics[width=22mm]{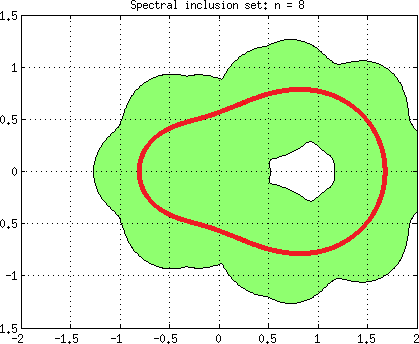} &
\includegraphics[width=22mm]{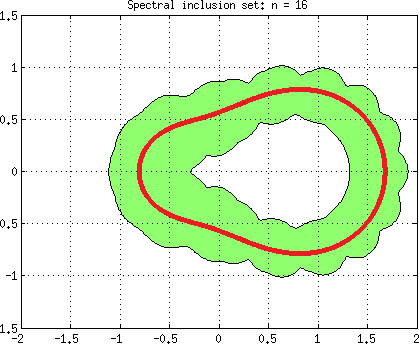} &
\includegraphics[width=22mm]{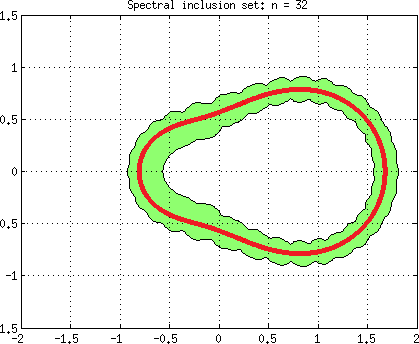}\\
$\stackrel{\text{$\tau_1$ method}}{\rule{0pt}{10mm}}$ &
\includegraphics[width=22mm]{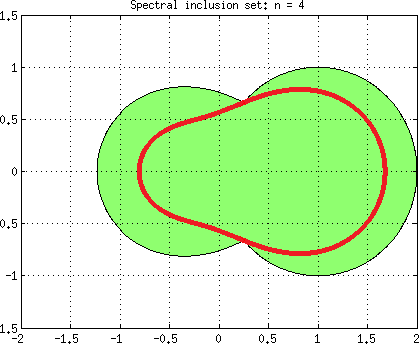} &
\includegraphics[width=22mm]{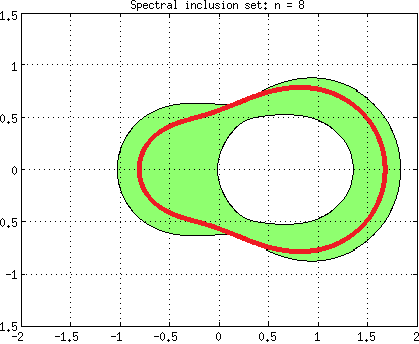} &
\includegraphics[width=22mm]{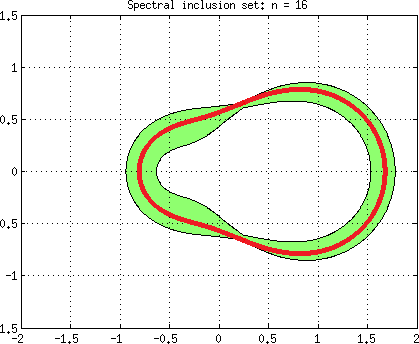} &
\includegraphics[width=22mm]{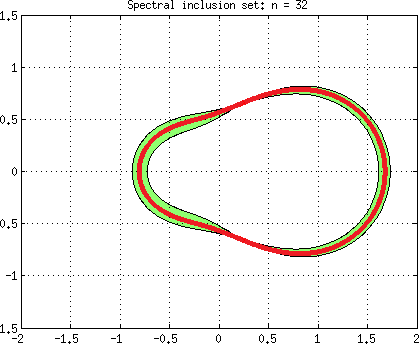}\\
\hline
\end{tabular}
\end{center}
\caption{The $\tau$, $\pi$, and $\tau_1$ inclusion sets, $\Sigma^n_{\eps_n}(A)$, $\Pi^{n,1}_{\eps'_n}(A)$, and $\Gamma^n_{\eps''_n}(A)$, respectively, for  $n=4,\,8,\,16$, and $32$, in the case that  $\alpha=(\cdots,0,0,0,\cdots)$,
$\beta=(\cdots,-\frac12,1,1,\cdots)$ and $\gamma=(\cdots,1,1,1,\cdots)$ are $3$-periodic. These are each inclusion sets for $\Spec A$, given by the first equation of \eqref{eq:specunion2} with $n=1$, and shown in each figure as the red curve.} \label{fig:3per2}
\end{figure}


\subsection{Pseudoergodic and random operators}
\label{sec:pseudoergodic} Let $\Sigma_\alpha$, $\Sigma_\beta$ and
$\Sigma_\gamma$ be nonempty compact subsets of $\C$ and let $A$ be
of the form \eqref{eq:A} with $\alpha_j\in\Sigma_\alpha$,
$\beta_j\in\Sigma_\beta$ and $\gamma_j\in\Sigma_\gamma$ for all
$j\in\Z$. Following Davies \cite{Davies2001:PseudoErg} (and see \cite{CWLi2016:Coburn} or \cite[Definition 1.1]{Colb:PE_pBC}) we say that $A$ is {\sl pseudoergodic with respect to $\Sigma_\alpha$, $\Sigma_\beta$, and $\Sigma_\gamma$}
 if, for all $\eps>0$, $n\in\N$ and all tridiagonal 
matrices $B_n\in\C^{n\times n}$ with subdiagonal entries in
$\Sigma_\alpha$, main diagonal entries in $\Sigma_\beta$,
and superdiagonal entries in $\Sigma_\gamma$,
there exists a $k\in\Z$ such that $\|A_{n,k}-B_n\|<\eps$. We say that $A$ is {\em pseudoergodic} if, for some non-empty compact $\Sigma_\alpha$, $\Sigma_\beta$, and $\Sigma_\gamma\subset \C$, $A$ is pseudoergodic with respect to $\Sigma_\alpha$, $\Sigma_\beta$, and $\Sigma_\gamma$.
 Pseudoergodicity was introduced by Davies \cite{Davies2001:PseudoErg} to study spectral properties of random operators while eliminating probabilistic arguments. Indeed (see \cite[\S5]{Davies2001:PseudoErg}), if the  matrix entries in \eqref{eq:A} are independent (or at least not fully correlated) random variables with probability measures whose supports are $\Sigma_\alpha$, $\Sigma_\beta$, and $\Sigma_\gamma$, then by the usual probabilistic arguments (sometimes termed the {\em infinite monkey theorem}) the operator $A$ is pseudoergodic almost surely.

The study of random and pseudoergodic operators, even restricting attention to the tridiagonal case that we consider here, has a large literature, both in the self-adjoint and, more recently, in the non-self-adjoint case. A readable, short introduction is \cite[{\S}36, 37]{TrefEmbBook}, and see \cite{CWLi2016:Coburn} and the references therein; see also \cite[{\S}3.4.10]{LiBook} or \cite{CW.Heng.ML:tridiag} for pseudoergodicity interpreted in the language of limit operators. Falling within the scope of the results of this section are: 
\begin{enumerate}
\item Random (self-adjoint, bi-infinite) Jacobi operators, including discrete versions of Schr\"odinger operators  with random bounded potentials (so $\Sigma_\alpha=\Sigma_\gamma = \{1\}$ and $\Sigma_\beta\subset \R$), including cases exhibiting Anderson localisation  (see, e.g., \cite{PasturFigotin} and the references therein).
\item  The (bi-infinite) non-self-adjoint version of the Anderson model,
with applications to superconductors and biological growth problems, introduced by Hatano \& Nelson
 \cite{HatanoNelson96,HatanoNelson97,HatanoNelson98}. In this  case $\Sigma_\alpha$ and $\Sigma_\gamma$ are real singleton sets (the $\alpha$ and $\gamma$ diagonals are constant) and $\Sigma_\beta\subset \R$. For this model the behaviour of  $\Spec A_{n,k}^{\per,1}$ as $n\to\infty$ has been described by Goldsheid \& Khoruzhenko
\cite{GK,GoldKoru}, $\Spec A$ in this case has been studied by a variety of methods in Davies
\cite{Davies2001:SpecNSA,Davies2001:PseudoErg,Davies2005:HigherNumRanges}
and Mart\'inez \cite{MartinezThesis,MartinezHN}. Moreover,  approximations to $\Speps A$ have been computed by Colbrook \cite{Colb:PE_pBC}, via the representation he establishes that $\Speps A =$ $\lim_{n\to\infty}$ $\Speps A_{2n+1,-n-1}^{\per,1}$.
\item The so-called Feinberg-Zee random hopping matrix, the case $\Sigma_\alpha=\{-1,1\}$, $\Sigma_\beta=\{0\}$,
$\Sigma_\gamma=\{1\}$, which has been studied in \cite{FeinZee99a,FeinZee99b,HolzOrlandZee,CW.Heng.ML:Sierp,CW.Heng.ML:tridiag,CW.Davies,Hagger:NumRange,Hagger:symmetries,AmHaNe,CW.Hagger,ColHan2019}, and which we will return to below; also, its generalisation  to the case $\Sigma_\alpha=\{\pm \sigma\}$, for some $\sigma\in (0,1]$, in Chandler-Wilde \& Davies \cite{CW.Davies} (and see \cite{AmHaNe,ColHan2019}).
\item The special case when $A$ is bidiagonal, meaning that $\Sigma_\alpha=\{0\}$ or $\Sigma_\gamma=\{0\}$; see \cite{TrefContEmb,Li:BiDiag}.
\end{enumerate}

Given non-empty compact sets $\Sigma_\alpha$, $\Sigma_\beta$, $\Sigma_\gamma\subset \C$, let $M(\Sigma_\alpha,\Sigma_\beta,\Sigma_\gamma)$ denote the set of operators $A$ of the form \eqref{eq:A} such that $\alpha_j\in \Sigma_\alpha$, $\beta_j\in \Sigma_\beta$, $\gamma_j\in \Sigma_\gamma$, for $j\in \Z$. The key result in the theory of pseudoergodic operators is the following: if $A$ is pseudoergodic with respect to $\Sigma_\alpha$, $\Sigma_\beta$, and $\Sigma_\gamma$, then
\begin{equation} \label{eq:specpse}
\begin{aligned}
\Speps A\ &=\ \bigcup_{B\in M(\Sigma_\alpha,\Sigma_\beta,\Sigma_\gamma)} \Speps B, \quad \eps \geq 0, 
\\ \speps A\ &=\ \bigcup_{B\in M(\Sigma_\alpha,\Sigma_\beta,\Sigma_\gamma)} \speps B, \quad \eps>0.
\end{aligned}
\end{equation}
The first of these equalities in the case $\eps=0$ dates back to \cite{Davies2001:PseudoErg}; the second is shown analogously as \cite[Corollary 4.12]{CWLi2016:Coburn}; the first for $\eps>0$ then follows by taking closures in the second identity\footnote{In more detail, taking closures in the second identity, since $\overline{\speps B} = \Speps B$, for any tridiagonal $B$ with scalar entries, we see that $\bigcup_{B\in M(\Sigma_\alpha,\Sigma_\beta,\Sigma_\gamma)} \Speps B \subset \overline{\bigcup_{B\in M(\Sigma_\alpha,\Sigma_\beta,\Sigma_\gamma)} \speps B} = \Speps A$. Since $A\in M(\Sigma_\alpha,\Sigma_\beta,\Sigma_\gamma)$, also $\Speps A \subset \bigcup_{B\in M(\Sigma_\alpha,\Sigma_\beta,\Sigma_\gamma)} \Speps B$.}. 

Clearly, \eqref{eq:specpse} implies that if $B\in M(\Sigma_\alpha,\Sigma_\beta,\Sigma_\gamma)$ then $\Speps B\subset \Speps A$, for $\eps\geq 0$, and  $\speps B\subset \speps A$, for $\eps> 0$. In particular this is true if $B$ is a Laurent operator, i.e.\ if, for some $\widehat \alpha\ \in \Sigma_\alpha$, $\widehat \beta\in \Sigma_\beta$, and $\widehat \gamma \in \Sigma_\gamma$, $\alpha_j=\widehat \alpha$, $\beta_j = \widehat \beta$, $\gamma_j = \widehat \gamma$, for $j\in \Z$, in which case, by \eqref{eq:specunion} with $n=1$,
$$
\Spec B = \cE(\widehat \alpha,\widehat \beta, \widehat \gamma) := \{z\widehat \alpha+\widehat \beta + z^{-1}\widehat \gamma:z\in \T\},
$$
which is an ellipse in the complex plane (degenerating to a line if $|\widehat \alpha| = |\widehat \gamma|$). It will be key to the proof that the $\pi$ method does not suffer from spectral pollution,  in Theorem \ref{thm:pseud} below, that also $\speps B\subset \speps A$, if $B\in M(\Sigma_\alpha,\Sigma_\beta,\Sigma_\gamma)$ and the coefficients $\alpha$, $\beta$, and $\gamma$ of $B$ are periodic, so that Corollary \ref{cor:periodic} applies. 

Let $\cE_\cap$ denote the set of $z\in \C$ which are in the interior of  $\cE(\widehat \alpha,\widehat \beta, \widehat \gamma)$, for every $\widehat \alpha\ \in \Sigma_\alpha$, $\widehat \beta\in \Sigma_\beta$, and $\widehat \gamma \in \Sigma_\gamma$, and let
\begin{equation} \label{eq:alpha*}
\alpha_* := \min_{z\in \Sigma_\alpha} |z|, \;\; \alpha^* := \max_{z\in \Sigma_\alpha} |z|, \;\; \gamma_* := \min_{z\in \Sigma_\gamma} |z|, \;\; \gamma^* := \max_{z\in \Sigma_\gamma} |z|.
\end{equation}
The main theoretical result of this section is the following. Recall, as is the case for any tridiagonal $A\in L(E)$ (irrespective of pseudoergodicity), that, in addition to the results shown in this theorem, the inclusions \eqref{incl:met1}, \eqref{incl:met1*}, and \eqref{incl:met2} apply to $A$.

\begin{theorem} \label{thm:pseud}
 Suppose that $X=\C$, so that $E=\ell^2(\Z)$, that $\Sigma_\alpha$, $\Sigma_\beta$, $\Sigma_\gamma\subset \C$ are non-empty and compact, and that $A$  is pseudoergodic with respect to $\Sigma_\alpha$, $\Sigma_\beta$, and $\Sigma_\gamma$. Then the $\pi$ method does not suffer from spectral pollution for $A$ for any $t\in \T$. Further, the $\tau$ method does not suffer from spectral pollution for $A$ if either: (a) $\gamma_* \leq \alpha^*$ and $\alpha_* \leq \gamma^*$; or (b) $E_\cap$ is empty. Moreover, as $n\to \infty$,
\begin{equation} \label{eq:c1}
\Gamma^n_{\eps+\eps''_n}(A)\ \Hto\ \Speps A, \quad \eps\geq 0, \quad \mbox{and} \quad \gamma^n_{\eps+\eps''_n}(A)\ \Hto\ \speps A, \quad \eps> 0;
\end{equation}
if $t\in \T$, also
\begin{equation} \label{eq:c2}
\Pi^{n,t}_{\eps+\eps'_n}(A)\ \Hto\ \Speps A, \quad \eps\geq 0, \quad \mbox{and} \quad \pi^{n,t}_{\eps+\eps'_n}(A)\ \Hto\ \speps A, \quad \eps> 0.
\end{equation}
Further, if (a) or (b) holds, also
\begin{equation} \label{eq:c3}
\Sigma^n_{\eps+\eps_n}(A)\ \Hto\ \Speps A, \quad \eps\geq 0, \quad \mbox{and} \quad \sigma^n_{\eps+\eps_n}(A)\ \Hto\ \speps A, \quad \eps> 0.
\end{equation}
\end{theorem}
\begin{proof} For $k\in \Z$, $n\in \N$, $t\in \T$, and  $\eps>0$, applying Corollary \ref{cor:periodic}, we see that $\speps A_{n,k}^{\per,t}$ $\subset \speps B$ for the unique $B\in M(\Sigma_\alpha,\Sigma_\beta,\Sigma_\gamma)$ with $n$-periodic coefficients that has $B_{n,k}^{\per,t}=A_{n,k}^{\per,t}$. But $\speps B\subset \speps A$ by \eqref{eq:specpse}. Thus the $\pi$ method does not suffer from spectral pollution for $A$ for any $t\in \T$. Let $E_+:= \ell_2(\N,\C)$ and suppose that $A_+\in L(E_+)$ is tridiagonal and pseudoergodic,  i.e., $A_+\in \Psi E_+(\Sigma_\alpha,\Sigma_\beta,\Sigma_\gamma)$, in the notation of \cite[{\S}1]{CWLi2016:Coburn}. Then, for $k\in \Z$, $n\in \N$, and  $\eps>0$, $\speps A_{n,k} \subset \speps A_+$, by \cite[Corollary 4.14 b)]{CWLi2016:Coburn}. Moreoever, $\speps A_+=\speps A$, by \cite[Proposition 4.13 c)]{CWLi2016:Coburn}, if (a) or (b) hold. Thus, if (a) or (b) hold, the $\tau$ method does not suffer from spectral pollution for $A$.

Noting that $X=\C$ trivially has the Globevnik property, \eqref{eq:c1} follows from Theorem \ref{thm:converge}, \eqref{eq:c2} follows from Theorem \ref{thm:converge3}, and, provided that (a) or (b) holds, \eqref{eq:c3} follows from Theorem \ref{thm:converge2}.
\end{proof}

\begin{remark} \label{rem:simon}
 In the case of the Feinberg-Zee random hopping matrix, i.e.\ the case that $\Sigma_\alpha=\{-1,1\}$, $\Sigma_\beta=\{0\}$, $\Sigma_\gamma = \{1\}$, we have that $\alpha_*=\alpha^*=\gamma_*=\gamma^*=1$, so that (a) holds in the above theorem and $\Sigma^n_{\eps+\eps_n}(A)\ \Hto\ \Speps A$, for $\eps \geq 0$. This result, for the specific Feinberg-Zee random hopping case, was shown previously, starting from \eqref{incl:met1}, in \cite[Theorem 4.9]{CW.Heng.ML:tridiag}.
\end{remark}

\begin{remark} \label{rem:Col} We note that Colbrook in \cite{Colb:PE_pBC} has proved results similar to the results in the above theorem for the $\pi$ method, by similar arguments. He shows in \cite{Colb:PE_pBC},  for very general classes of pseudoergodic operators $A$, including tridiagonal pseudoergodic operators as defined here, that the pseudospectra of periodised finite section approximations converge to those of $A$, precisely that
\begin{equation} \label{eq:col2}
\Speps A_{2n+1,-n-1}^{\per,1} \Hto \Speps A, \quad \mbox{as} \quad n\to\infty,
\end{equation}
for every $\eps>0$.
\end{remark}

\subsubsection{The case that $\Sigma_\alpha$, $\Sigma_\beta$, and $\Sigma_\gamma$ are finite sets} \label{sec:psFinite}
When $\Sigma_\alpha$, $\Sigma_\beta$, and $\Sigma_\gamma$ are finite sets, as holds, for example, in the Feinberg-Zee case, then the infinite collections of finite matrices, indexed by $k\in \Z$, in the definitions \eqref{eq:sigdef}, \eqref{eq:pidef}, and \eqref{eq:gamdef}/\eqref{eq:munmat}
of the 
inclusion sets $\Sigma^n_{\eps+\eps_n}(A)$, $\sigma^n_{\eps+\eps_n}(A)$, $\Pi^{n,t}_{\eps+\eps'_n}(A)$, $\pi^n_{\eps+\eps'_n}(A)$, $\Gamma^n_{\eps+\eps''_n}(A)$, and $\gamma^n_{\eps+\eps''_n}(A)$, are in fact (large) finite sets of matrices. Let $M_n(\Sigma_\alpha, \Sigma_\beta,\Sigma_\gamma)$ denote the set of all $n\times n$ tridiagonal matrices with subdiagonal entries in $\Sigma_\alpha$, diagonal entries in $\Sigma_\beta$, and superdiagonal entries in $\Sigma_\gamma$. This  set has cardinality $N_\alpha^{n-1}N_\beta^nN_\gamma^{n-1}$, where 
$N_\alpha:=|\Sigma_\alpha|$, $N_\beta:=|\Sigma_\beta|$, and $N_\gamma:=|\Sigma_\gamma|$. Further, recalling the notation \eqref{eq:cS}, it is easy to see that if $A$ is pseudoergodic with respect to $\Sigma_\alpha$, $\Sigma_\beta$, and $\Sigma_\gamma$, then
\begin{equation} \label{eq:seteq}
\cS^{(\tau)}_n := \{A_{n,k}:k\in \Z\} = M_n(\Sigma_\alpha, \Sigma_\beta,\Sigma_\gamma).
\end{equation}
Similarly, where $B^t_n(\Sigma_\alpha, \Sigma_\gamma)$, for some fixed $t\in \T$, denotes the set of all $n\times n$ matrices $B$ with entries $b_{i,j}=\delta_{i,1}\delta_{j,n}t\widehat \alpha + \delta_{i,n}\delta_{j,1}\overline{t} \widehat \gamma$, $i,j=1,\ldots,n$, for some $\widehat \alpha\in \Sigma_\alpha$ and $\widehat \gamma\in \Sigma_\gamma$, it holds that
\begin{equation} \label{eq:seteq2}
\cS^{(\pi)}_n := \{A^{\per,t}_{n,k}:k\in \Z\} = M_n(\Sigma_\alpha, \Sigma_\beta,\Sigma_\gamma) + B^t_n(\Sigma_\alpha,\Sigma_\gamma),
\end{equation}
which set has cardinality no larger than $N_\alpha^{n}N_\beta^nN_\gamma^{n}$, indeed exactly this cardinality for $n\geq 3$.

In the Feinberg-Zee case these sets have cardinalities $|\cS^{(\tau)}_n|=2^{n-1}$ and $|\cS^{(\pi)}_n|=2^n$, respectively. Thus, given $\lambda\in \C$, determining whether $\lambda \in \Sigma^n_{\eps+\eps_n}(A)$ requires half the computation of determining whether $\lambda \in \Pi^n_{\eps+\eps_n}(A)$. Since both these inclusion sets converge to $\Speps A$, for $\eps\geq 0$, by Theorem \ref{thm:pseud} and Remark \ref{rem:simon}, this suggests computing $\Sigma^n_{\eps+\eps_n}(A)$ rather than $\Pi^n_{\eps+\eps_n}(A)$, as an approximation to $\Speps A$. 

In Figure \ref{fig:FZ}, taken from \cite{CW.Heng.ML:tridiag}, we plot $\Sigma^n_{\eps_n}(A)\supset \Spec A$ for $n=6$, $12$, and $18$. Concretely, by \eqref{eq:sigmaAAlt}, \eqref{eq:sigdef*},
 and \eqref{eq:seteq},
$$
\Sigma^n_{\eps_n}(A) = \left\{\lambda\in \C: \min_{B\in M_n(\{\pm1\}, \{0\},\{1\})}\nu(B-\lambda I_n)\leq \eps_n\right\}
$$
and, cf.~\eqref{eq:ispd},
$$
\nu(B-\lambda I_n) > \eps_n \quad \Leftrightarrow \quad (B-\lambda I_n)'(B-\lambda I_n)-\eps_n^2I_n \;\; \mbox{ is positive definite}.
$$
Thus, to decide whether $\lambda \in \Sigma^n_{\eps_n}(A)$ one has to decide whether or not $2^{n-1}$  Hermitian matrices of size $n\times n$ are positive definite. For $n=18$ this requires consideration of the positive definiteness of 131,072  matrices of size $18\times 18$ for each $\lambda$ on some grid covering the domain of interest.

\begin{figure}[t]
\begin{center}
\includegraphics[width=1.\textwidth]{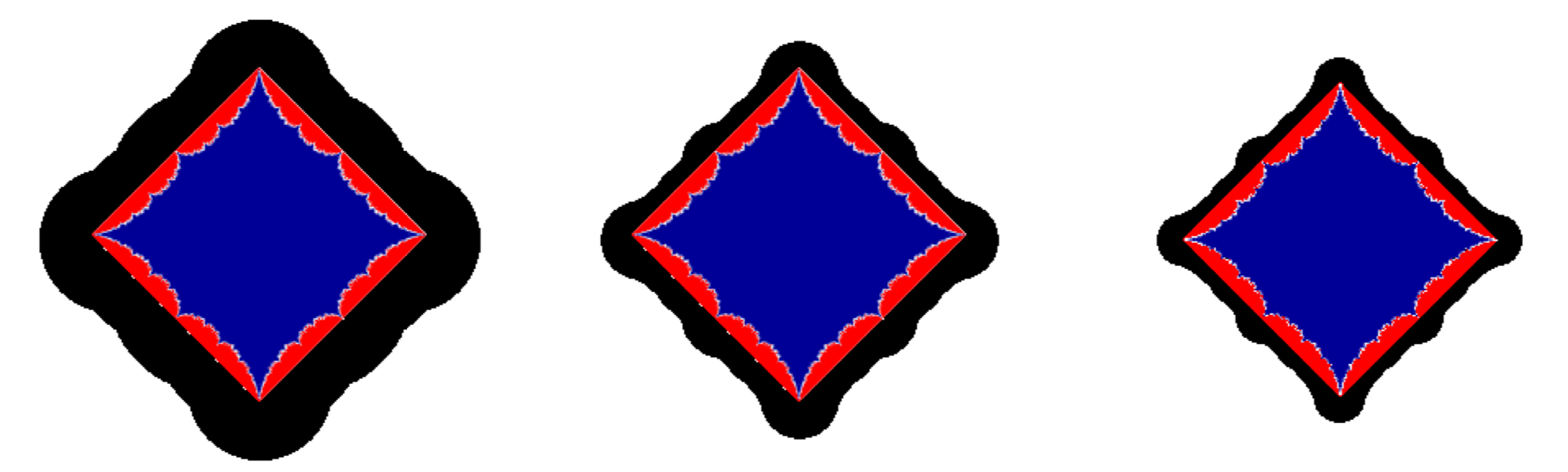}
\end{center}
\caption{Plots in black of $\Sigma^n_{\eps_n}(A)$, for $n=6$, $12$, and $18$, in the Feinberg-Zee case that  $A$ is pseudoergodic with respect to $\Sigma_\alpha=\{-1,1\}$, $\Sigma_\beta=\{0\}$, and $\Sigma_\gamma=\{1\}$. $\Sigma^n_{\eps_n}(A)$ is an inclusion set for $\Spec A$ for each $n$, by \eqref{incl:met1}, and $\Sigma^n_{\eps_n}(A)\Hto \Spec A$ as $n\to\infty$, by Theorem \ref{thm:pseud} and Remark \ref{rem:simon}. Shown in red in each plot  is the closure of the numerical range of $A$, which (see \cite[Lemma 3.1]{CW.Heng.ML:tridiag}) is the square with corners at $\pm 2$ and $\pm 2\ri$. This is also an inclusion set for $\Spec A$ (see \S\ref{sec:other}). Overlaid  in blue in each plot is the set $\pi_{30}\cup \D$, where $\pi_n$ (which can be computed by Corollary \ref{cor:periodic}) is the union of $\Spec B$ over all $B\in M(\{-1,1\},\{0\},\{1\})$ that are periodic with period $p\leq n$. For each $n\in \N$, $\pi_n\subset \Spec A$ by \eqref{eq:specpse}, and $\D\subset \Spec A$ by \cite[Theorem 2.3]{CW.Heng.ML:Sierp}.} \label{fig:FZ}
\end{figure}

In Figure \ref{fig:FZ} we also plot $\overline{\mathrm{Num}\, A}$, the closure of the numerical range of $A$, computed in \cite{CW.Heng.ML:tridiag} to be the square $S=\conv\{\pm 2,\pm 2\ri\}$. Disappointingly,  $\overline{\mathrm{Num}\, A}$ is a sharper inclusion set than $\Sigma^n_{\eps_n}(A)$, at least for these modest values of $n$, i.e.\ $\Sigma^n_{\eps_n}(A) \supset S=\overline{\mathrm{Num}\, A}\supset \Spec A$, for $n=6$, $12$, $18$. If the inclusion $\Sigma^n_{\eps_n}(A) \supset \overline{\mathrm{Num}\, A}$ held for all $n\in\N$ it would follow, since  $\Spec A \subset \overline{\mathrm{Num}\, A}$ and $\Sigma^n_{\eps_n}(A)\Hto \Spec A$ as $n\to\infty$, that $\Spec A =\overline{\mathrm{Num}\, A}=S$. To show that this is not the case we carried out calculations in \cite{CW.Heng.ML:tridiag} to demonstrate  that  $\lambda \not \in \Sigma_{\eps_{n}}^{n}(A)$ for $n=34$ and $\lambda=\frac 32 + \frac 12\ri$, a point on the boundary of the square $S$. This calculation is a matter of checking that $2^{33}\approx 8.6 \times 10^9$  matrices of size $34\times 34$ are positive definite. This is a large calculation, but one suited to a massively parallel computer architecture, with each processor assigned a subset of the nearly 9 billion matrices to check their positive definiteness. We note also that this calculation is an example of using our inclusion sets to show that a particular operator is invertible (see Remark \ref{rem:invert}). These calculations, showing that, for this particular choice of $\lambda$, $\lambda\not \in \Sigma_{\eps_{n}}^{n}(A)$ for $n=34$, prove that $\lambda\not \in \Spec A$, i.e.\ $\lambda I-A$ is invertible.
For more information about these calculations and those in Figure \ref{fig:FZ}, see \cite{CW.Heng.ML:tridiag}. For a review of what is known about $\Spec A$ for this Feinberg-Zee case, and a list of challenging open problems, see  \cite{CW.Hagger}.

\subsubsection{The case that $\Sigma_\alpha$, $\Sigma_\beta$, and $\Sigma_\gamma$ are infinite sets, and the SCI}  \label{sec:psInfinite}

Consider now the case that the sets $\Sigma_\alpha$, $\Sigma_\beta$, and $\Sigma_\gamma$ are not necessarily finite. (We have in mind applications where one or more of these sets is infinite, or all the sets are finite but the cardinality of $M_n(\Sigma_\alpha, \Sigma_\beta,\Sigma_\gamma)$ is unfeasibly large for computation as in the previous subsection.) It is easy to see that, if $A$ is pseudoergodic with respect to $\Sigma_\alpha$, $\Sigma_\beta$, and $\Sigma_\gamma$, then, where $\cS^{(\tau)}$ and $\cS^{(\pi)}$ are defined by \eqref{eq:seteq} and \eqref{eq:seteq2},
\begin{equation} \label{eq:seteqinf}
\overline{\cS^{(\tau)}_n}=\overline{\{A_{n,k}:k\in \Z\}} = M_n(\Sigma_\alpha, \Sigma_\beta,\Sigma_\gamma)
\end{equation}
and, for the chosen fixed $t\in \T$,
\begin{equation} \label{eq:seteq2inf}
\overline{\cS^{(\pi)}_n}=\overline{\{A^{\per,t}_{n,k}:k\in \Z\}} = M_n(\Sigma_\alpha, \Sigma_\beta,\Sigma_\gamma) + B^t_n(\Sigma_\alpha,\Sigma_\gamma).
\end{equation}
(These equations are equivalent to \eqref{eq:seteq} and \eqref{eq:seteq2} in the case that $\Sigma_\alpha$, $\Sigma_\beta$, and $\Sigma_\gamma$ are all finite.)

This is a case where it is relatively clear how to select a finite set $S_n\subset \overline{\cS^{(M)}_n}$, for $M=\tau$ or $\pi$, such that \eqref{eq:comp4} holds, so that the approximations to $\Speps A$ proposed in Theorem \ref{thm:fin3} can be computed. Indeed, suppose that we can find sequences of finite sets $(\Sigma^n_\alpha)_{n\in \N}\subset \Sigma_\alpha$, $(\Sigma^n_\beta)_{n\in \N}\subset \Sigma_\beta$, and $(\Sigma^n_\gamma)_{n\in \N}\subset \Sigma_\gamma$ such that
\begin{equation} \label{eq:dHSig}
d_H(\Sigma_G,\Sigma^n_G) \leq \frac{1}{3n},  \qquad \mbox{for } \;\; n\in \N, \;\; \mbox{ and } G\in \{\alpha,\beta, \gamma\}.
\end{equation}
Then $S_n=S_n^{(M)}$, as defined in the following lemma, is  finite  and satisfies \eqref{eq:comp4}.

\begin{lemma} \label{lem:Sn} Suppose that $A$ is pseudoergodic with respect to $\Sigma_\alpha$, $\Sigma_\beta$, and $\Sigma_\gamma$, and set, for $n\in \N$,
$$
S^{(M)}_n := \left\{\begin{array}{ll} M_n(\Sigma^n_\alpha, \Sigma^n_\beta,\Sigma^n_\gamma),& \mbox{if $M=\tau$},\\
M_n(\Sigma^n_\alpha, \Sigma^n_\beta,\Sigma^n_\gamma) +B^t_n(\Sigma^n_\alpha,\Sigma^n_\gamma), & \mbox{if $M=\pi$.}
\end{array}\right.
$$
Then $S^{(M)}_n$ is a finite subset of $\overline{\cS_n^{(M)}}$ and \eqref{eq:comp4} holds with $S_n=S_n^{(M)}$.
\end{lemma}
\begin{proof}
Let $M=\tau$ or $\pi$ and let $S_n := S_n^{(M)}$. It is clear that $S_n$ is finite, since $\Sigma^n_\alpha$, $\Sigma^n_\beta$, and $\Sigma^n_\gamma$ are finite. Further, for  $G=\alpha$, $\beta$, or $\gamma$, and every  $n\in \N$ and $z\in \Sigma_G$, there exists $z_n\in \Sigma^n_G$ such that $|z-z_n|\leq 1/(3n)$. But this implies, for every $B\in \cS_n^{(M)}=\{A_{n,k}^{(M)}:k\in \Z\}$, where $A_{n,k}^{(M)}$ is defined by \eqref{eq:gennot}, that there exists $B^n\in S_n$ such that $|b_{i,j}-b^n_{i,j}|\leq 1/(3n)$, for $i,j=1,\ldots,n$, where $b_{i,j}$ and $b^n_{i,j}$ denote the elements in row $i$ and column $j$ of $B$ and $B^n$, respectively. But this implies that $\sum_{j=1}^n|b_{i,j}-b^n_{i,j}|\leq 1/n$ and  $\sum_{i=1}^n|b_{i,j}-b^n_{i,j}|\leq 1/n$, which imply that $\|B-B^n\|\leq 1/n$.
\end{proof}

Combining the above lemma with Theorem \ref{thm:fin3} and with the results on absence of spectral pollution in Theorem \ref{thm:pseud}, we have immediately the following result, noting that the definition \eqref{eq:findefps} coincides with \eqref{eq:findef4} by \eqref{eqln3}.

\begin{corollary} \label{cor:pseud}
Suppose that $A$ is pseudoergodic with respect to $\Sigma_\alpha$, $\Sigma_\beta$, and $\Sigma_\gamma$, and that $S^{(M)}_n$, for $n\in \N$ and $M=\tau$ or $\pi$, is defined as in Lemma \ref{lem:Sn}, where $(\Sigma^n_\alpha)_{n\in \N}\subset \Sigma_\alpha$, $(\Sigma^n_\beta)_{n\in \N}\subset \Sigma_\beta$, and $(\Sigma^n_\gamma)_{n\in \N}\subset \Sigma_\gamma$ are finite sets satisfying \eqref{eq:dHSig}. Further, let
\begin{equation} \label{eq:findefps}
\Delta^{n,\mathrm{fin},M}_{\eps}(A) :=  \left\{\lambda\in \C : \min_{B\in S^{(M)}_n} \nu\left(B-\lambda I_n\right)\leq \eps\right\},
\end{equation}
 for $n\in \N$, $M=\tau$ or $\pi$, and $\eps\geq 0$. Then, for $\eps\geq 0$ and $n\in \N$,
\begin{equation} \label{eq:genn2p}
\Speps A\ \subset \Delta^{n,\mathrm{fin},\pi}_{\eps+\eps'_n + 1/n}(A),
\end{equation}
 and
 $\Delta^{n,\mathrm{fin},\pi}_{\eps+\eps'_n + 1/n}(A)\Hto \Speps A$ as $n\to\infty$, in particular $\Delta^{n,\mathrm{fin},\pi}_{\eps'_n+1/n}(A)\Hto \Spec A$. If, moreover, conditions (a) or (b) in Theorem \ref{thm:pseud} are satisfied, than also 
\begin{equation} \label{eq:genn2p2}
\Speps A\ \subset \Delta^{n,\mathrm{fin},\tau}_{\eps+\eps_n + 1/n}(A),
\end{equation}
 for $\eps\geq 0$ and $n\in \N$, and
 $\Delta^{n,\mathrm{fin},\tau}_{\eps+\eps_n + 1/n}(A)\Hto \Speps A$ as $n\to\infty$, in particular $\Delta^{n,\mathrm{fin},\tau}_{\eps_n+1/n}(A)$ $\Hto \Spec A$.
\end{corollary}

As we did in \S\ref{sec:SCI} and \S\ref{sec:scalar}, we can go beyond the above corollary and define a sequence of approximations $(\Pi^n_{\mathrm{fin}}(A))_{n\in \N}$ to $\Spec A$, when $A$ is pseudoergodic,  such that each $\Pi^n_{\mathrm{fin}}(A)$ is a finite set that can be computed in finitely many arithmetic operations, and $\Pi^n_{\mathrm{fin}}(A)\Hto \Spec A$ as $n\to\infty$.

Let $\Omega_\Psi$ denote the set of all tridiagonal matrices \eqref{eq:A}, with $\alpha,\beta,\gamma\in \ell^{\infty}(\Z,\C)$, that are pseudoergodic. For a given $A\in \Omega_\Psi$, with coefficients $\alpha,\beta,\gamma\in \ell^{\infty}(\Z,\C)$, let 
$$
\Sigma_\alpha := \overline{\{\alpha_j:j\in \Z\}}, \quad \Sigma_\beta := \overline{\{\beta_j:j\in \Z\}}, \quad \Sigma_\gamma := \overline{\{\gamma_j:j\in \Z\}}.
$$ 
Note that, with these definitions, $A$ is pseudoergodic with respect to $\Sigma_\alpha$, $\Sigma_\beta$, and $\Sigma_\gamma$.
As in \S\ref{sec:SCI}, let $P(S)$ denote the power set of a set $S$. The inputs we need to compute $\Pi^n_{\mathrm{fin}}(A)$, for $n\in \N$, are provided by the following mappings (cf.~\S\ref{sec:SCI}):
\begin{enumerate}
\item A mapping $\B:\Omega_\Psi\to \R^3$, $A\mapsto (\alpha_{\max},\beta_{\max},\gamma_{\max})$, such that $\alpha_{\max} \geq \|\alpha\|_\infty$, $\beta_{\max} \geq \|\beta\|_\infty$,$\gamma_{\max} \geq \|\gamma\|_\infty$.
\item A mapping $\cE:\Omega_\Psi\times \N \to (P(\C))^3$, $(A,n)\mapsto (\Sigma_\alpha^n,\Sigma_\beta^n,\Sigma_\gamma^n)$, such that  $\Sigma^n_\alpha\subset \Sigma_\alpha$, $\Sigma^n_\beta\subset \Sigma_\beta$, and $\Sigma^n_\gamma\subset \Sigma_\gamma$ are finite sets satisfying \eqref{eq:dHSig}.
\end{enumerate}
Note that $\B$ is the mapping $\B_1$ of \S\ref{sec:SCI}, restricted to $\Omega_\Psi\subset \Omega_T^1$. 

Given $A\in \Omega_\Psi$, for $n\in \N$ let (cf.~\eqref{eq:epsn_method1*})
\begin{equation} \label{eq:epsdag}
\eps_n^\dag\ :=\ (\alpha_{\max}+\gamma_{\max})\,\frac{22}{7n}\ \geq\ 2(\alpha_{\max}+\gamma_{\max})\,\sin\frac{\pi}{2n}\geq \eps'_n,
\end{equation}
where
$(\alpha_{\max},\beta_{\max},\gamma_{\max}) := \B(A)$. 
Recalling the definition \eqref{eq:grid}, define the sequence $(\Pi^n_{\mathrm{fin}}(A))_{n\in \N}$ by
\begin{equation} \label{eq:finite3}
\Pi^n_{\mathrm{fin}}(A) := \Delta^{n,\mathrm{fin},\pi}_{\eps^\dag_n+3/n}(A)\cap \mathrm{Grid}(n,R), \qquad n\in \N,
\end{equation}
where $R:= \alpha_{\max}+\beta_{\max}+\gamma_{\max}$ (so that $R$ is an upper bound for $\|A\|$), and where $\Delta^{n,\mathrm{fin},\pi}_{\eps}(A)$, for $n\in \N$ and $\eps\geq 0$, is defined by \eqref{eq:findefps}, with $S_n^{(\pi)}$ defined as in Lemma \ref{lem:Sn}, for some $t\in \T$, with $(\Sigma_\alpha^n,\Sigma_\beta^n,\Sigma_\gamma^n):= \cE(A,n)$.

\begin{proposition} \label{prop:finiteps}
For $A\in \Omega_\Psi$ and $n\in \N$, $\Pi^n_{\mathrm{fin}}(A)$ can be computed in finitely many arithmetic operations, given  $(\alpha_{\max},\beta_{\max},\gamma_{\max}) := \B(A)$ and $(\Sigma_\alpha^n,\Sigma_\beta^n,\Sigma_\gamma^n):= \cE(A,n)$. Further, $\Pi^n_{\mathrm{fin}}(A)\Hto \Spec A$ as $n\to \infty$, and also
$$
\widehat \Pi^n_{\mathrm{fin}}(A)\ :=\ \Pi^n_{\mathrm{fin}}(A) + \frac{2}{n}\overline{\D}\ \Hto\ \Spec A
$$
as $n\to \infty$, with $\Spec A \subset \widehat \Pi^n_{\mathrm{fin}}(A)\subset \Specn_{\eps^\dag_n+3/n}(A) + \frac{2}{n}\overline{\D}$ for each $n\in \N$.
\end{proposition}

We omit the proof of the above result which is a minor variant of the proof of Proposition \ref{prop:finite} in \S\ref{sec:scalar}. But let us  interpret this proposition  as a result relating to the solvability complexity index (SCI) of \cite{SCI,ColHan2023}, as we did for Proposition \ref{prop:finite} in \S\ref{sec:SCI}.  Equipping $\C^C$ with the Hausdorff metric as in \S\ref{sec:keynot},
the mappings 
$$
\Omega_\Psi\to \C^C, \qquad A\mapsto \Pi^n_{\mathrm{fin}}(A) \quad \mbox{and} \quad A\mapsto \widehat \Pi^n_{\mathrm{fin}}(A),
$$
are, for each $n\in \N$,  general algorithms in the sense of \cite{SCI,ColHan2023},  with  evaluation set (in the sense of  \cite{SCI,ColHan2023})
\begin{equation} \label{eq:Lambda2}
\Lambda:= \{\B,\cE(\cdot,n):n\in \N\};
\end{equation}
each function in $\Lambda$ has domain $\Omega$, and each can be expressed in terms of finitely many complex-valued functions; cf.~the footnote below \eqref{eq:Lambda}.
Further, $\widehat \Pi^n_{\mathrm{fin}}(A)$ can be computed in finitely many arithmetic operations and specified using finitely many complex numbers (the elements of $\Pi^n_{\mathrm{fin}}(A)$ and the value of $n$). Thus, where $\Xi:\Omega_\Psi\to \C^C$ is the mapping given by $\Xi(A) := \Spec A$, for $A\in \Omega_\Psi$, the computational problem $\{\Xi,\Omega_\Psi, \C^C,\Lambda\}$ has  arithmetic SCI, in the sense of \cite{SCI,ColHan2023}, equal to one; more precisely, since also $\Spec A \subset \widehat \Pi^n_{\mathrm{fin}}(A)$, for each $n\in \N$ and $A\in \Omega_\Psi$, this computational problem is in the class $\Pi_1^A$, as defined in \cite{SCI,ColHan2023}. 

\begin{remark}{\bf (Rate of convergence)} It follows from the last inclusions in Proposition \ref{prop:finiteps} that, if $A\in \Omega_\Psi$, i.e., if $A$ is pseudoergodic with respect to some non-empty compact sets $\Sigma_\alpha$, $\Sigma_\beta$, and $\Sigma_\gamma$, then
$$
d_H(\Spec A,\widehat \Pi^n_{\mathrm{fin}}(A)) \leq \delta_n := d_H(\Spec A,\Specn_{\eps^\dag_n+3/n}(A)) + 2/n \to 0,
$$
as $n\to\infty$ by \eqref{eq:HDconv}. Further note, by \eqref{eq:specpse}, that $\delta_n$ depends only on $n$ and the sets $\Sigma_\alpha$, $\Sigma_\beta$, and $\Sigma_\gamma$. Thus the convergence $\widehat \Pi^n_{\mathrm{fin}}(A)\Hto \Spec A$ is uniform in $A$, for all $A$ that are pseudoergodic with respect to $\Sigma_\alpha$, $\Sigma_\beta$, and $\Sigma_\gamma$. By contrast,  the rate of convergence of the approximation \eqref{eq:col2} to $\Speps A$ (note that $\Speps A$  is given in terms of $\Sigma_\alpha$, $\Sigma_\beta$, and $\Sigma_\gamma$ by \eqref{eq:specpse}) depends on the particular realisation  of $A$.

To make this concrete (see also the related discussions in \cite[\S4.3]{CW.Heng.ML:tridiag} and \cite{Gabel}), consider the Feinberg-Zee case that $\Sigma_\alpha=\{-1,1\}$, $\Sigma_\beta = \{0\}$, and $\Sigma_\gamma = \{1\}$, and suppose $A$, taking the form \eqref{eq:A}, is a random Feinberg-Zee matrix, with $\beta_j=0$ and $\gamma_j=1$, for $j\in \Z$, and with  the entries $\alpha_j\in \{-1,1\}$ i.i.d.\ random variables, with $p:=\mathrm{Pr}(\alpha_j=-1)\in (0,1)$. Then, almost surely, $A$ 
is pseudoergodic with respect to $\Sigma_\alpha$, $\Sigma_\beta$, and $\Sigma_\gamma$, and, by the above result, the convergence $\Pi^n_{\mathrm{fin}}(A)\Hto \Spec A$ is (almost surely) independent of the particular realisation of $A$ and independent of $p$.  But the rate of convergence of \eqref{eq:col2} depends on $p$. In particular,  given any $n\in \N$, it holds, with arbitrarily  high probability if $p$ is sufficiently small,  that $\alpha_j=1$, for $|j|\leq n$,   so that $A^{\per,1}_{2n+1,-n-1}$ is symmetric, $\Spec A^{\per,1}_{2n+1,-n-1} \subset \R$, and   $\Spec A^{\per,1}_{2n+1,-n-1} \subset \R+\eps \overline D$. But this implies that  $$
d_H(\Speps A,\Spec A^{\per,1}_{2n+1,-n-1}) \geq 2,$$ since  $2\ri \in \Spec A$ by  \cite[Theorem 2.5, Lemma 2.6]{CW.Hagger}, so that  $(2+\eps)\ri \in \Speps A$.
\end{remark}

\section{Open problems and directions for further work} \label{sec:ext}

We see a number of open problems and possible directions for future work. Firstly, in our own work in preparation we are considering two extensions to the work reported above, both of which are announced, in part, in the recent conference paper \cite{PAMM}. The first is that our inclusion results \eqref{incl:met1}, \eqref{incl:met1*}, \eqref{incl:met2}, for tridiagonal bi-infinite matrices of the form \eqref{eq:A}, have natural extensions to semi-infinite matrices and, indeed, finite matrices. (The inclusion sets for finite matrices can be viewed as a generalisation of the Gershgorin theorem, which (see \S\ref{sec:ideas}) was a key inspiration  for the work we report here.)
The second extension, announced in part in \cite{PAMM}, is that it is possible, building on the results in this paper for the case of the spectrum, to write down a sequence of approximations, defined in terms of spectral properties of finite matrices, that is convergent to  the {\it essential spectrum}, with each approximation also an inclusion set for the essential spectrum.

In addition to the above we see the following open questions/possible directions for future work:
\begin{enumerate}
\item {\em How can the algorithms, implicit in the definitions of our sequences of approximations,  be  implemented so as to optimise convergence as a function of computation time?} Progress towards efficient implementation of something close to the approximation  on the right hand side of \eqref{eq:genn}, building on the work in (a draft of) this paper, has been made in work by Lindner and Schmidt \cite{LiSchmidt:Givens,TorgePhD}.

\item {\em Can the algorithms in this paper be used or adapted to resolve open questions regarding the spectra of bounded linear operators arising in applications?} We have indicated one application of this sort in \S\ref{sec:psFinite}, an example which is a pseudoergodic operator where the set of $\tau$-method matrices,  $\cS_n^{(\tau)} = \{A_{n,k}:k\in \Z\}$, has cardinality $|\cS_n^{(\tau)}|$ that is finite but grows exponentially with $n$. There are interesting potential applications, notably the Fibonacci Hamiltonian and non-self-adjoint variants \cite{DEG,DGY,Gabel}, where $|\cS_n^{(\tau)}|$ grows only linearly with $n$, allowing exploration of spectral properties with much larger values of $n$. At the other extreme, where $\cS_n$ (and the corresponding $\pi$ and $\tau_1$ method sets) are uncountable and so must be approximated as in Theorems \ref{thm:fin3} and \ref{prop:finite2}, one might seek to use or adapt our methods to study the spectra of (bounded) integral operators $\A$ on $L^2(\R)$, for example those that  arise in problems of wave scattering by unbounded rough surfaces (e.g., \cite{ZCW}). Via a standard discretisation (see, e.g., \cite[{\S}1.2.3]{LiBook} or \cite[Chapter 8]{CWLi2008:Memoir}) the operator $\A$ on $L^2(\R)$ is unitarily equivalent to  an operator $A$ on $E=\ell^2(\Z,X)$ with $X=L^2[0,1]$. Thus, concretely, one might use the inclusion sets in this paper, in the manner indicated in Remark \ref{rem:invert}, to study invertibility of $\lambda I - \A$ on $L^2(\R)$, for some specific $\lambda\in \C$ of relevance to applications. 

\item {\em Can our inclusion set analysis for $E=\ell^2(\Z,X)$ be extended to $E^p:=\ell^p(\Z,X)$, for $1\leq p\leq \infty$?} Although the spectrum of an operator of the form \eqref{eq:A} does not
change if we regard $A$ as acting on $E^p$, with $p\in[1,\infty]$,
rather than on $E$ \cite[Corollary 2.5.4]{RaRoSiBook}, the pseudospectrum does depend on $p$. We note that many results on approximation of pseudospectra of operators on $E$ extend to operators on $E^p$ (see, e.g., \cite{BGS,CW.Heng.ML:tridiag,CWLi2016:Coburn,Colb:PE_pBC}), and that Proposition 6 in Lindner \& Seidel \cite{LiSei:BigQuest}, derived for other purposes, can be seen as a key step in this direction.

\item {\em Can our inclusion sets and convergence results for $E=\ell^2(\Z,X)$ be extended to $\ell^2(\Z^N,X)$ for arbitrary $N\in \N$?} We see no reason why this extension should not be possible. Indeed, Proposition 6 in Lindner \& Seidel \cite{LiSei:BigQuest} can be seen as a first version (a version without the weights $w_j$) of an extension of our key Proposition \ref{prop:gam_main} to this more general case, indeed to the case $E=\ell^p(\Z^N,X)$, for $p\in [1,\infty]$. (Note that the proof of Proposition 6 provided in \cite[pp.~909--910]{LiSei:BigQuest} takes inspiration from  the proof of Proposition \ref{prop:gam_main} in this paper and in \cite{HengThesis}.)  This extension would be attractive so as to tackle a wider range of applications (see, e.g., the multidimensional applications tackled in \cite{Colb:PE_pBC}).
\item {\em Can our results, which are all for {\em bounded} linear operators on $E=\ell^2(\Z,X)$, be extended to (unbounded) closed, densely defined operators?} Such an extension would be attractive for many applications. We note that other recent work on approximation of spectra and/or pseudospectra (e.g., \cite{DaviesPlum,Bogli1,Bogli2,Colb2022,ColHan2023}) extends to this setting.
\end{enumerate}


\begin{ack}
We are grateful for the inspiration 
and encouragement, through many enjoyable interactions at seminars, workshops, and otherwise over many years, of Brian Davies, to whom this paper is dedicated on the occasion of his 80th birthday. We thank Estelle Basor, Albrecht B\"ottcher, Matthew Colbrook, Brian Davies, Raffael Hagger, Titus Hilberdink, Michael Levitin, Marco Marletta, Markus Seidel, Torge Schmidt, Eugene Shargorodsky, and Riko Ukena for helpful discussions and feedback regarding the work reported in this paper. We are grateful also to the detailed feedback from the two anonymous referees which has substantially improved this paper. 
\end{ack}

\begin{funding}
This work was partially supported by a Higher Education Strategic Scholarship
for Frontier Research from the Thai Ministry of Higher Education to the
second author, and the Marie-Curie Grants MEIF-CT-2005-009758 and
PERG02-GA-2007-224761 of the EU to the third author.
\end{funding}


\end{document}